\documentclass[a4paper,11pt,english,reqno]{amsart}
\usepackage{amsmath,amssymb,amsfonts,dsfont,amsthm}    
\usepackage[utf8]{inputenc}
\usepackage[english]{babel}
\usepackage[T1]{fontenc} 
\usepackage{graphicx}
\usepackage{float}
\usepackage[font=small,labelfont=bf]{caption}
\usepackage[a4paper]{geometry}
\usepackage{enumerate}
\newcommand{\tLambda}{\widetilde{\Lambda}}

\newcommand{\vareps}{\varepsilon}
\newcommand{\defeq}{\stackrel{\mathrm{def}}{=}}
\newcommand{\prob}{\mathds{P}}
\newcommand{\erw}{\mathds{E}}
\newcommand{\e}{\mathrm{e}}
\newcommand{\cH}{\mathcal{H}}
\newcommand{\Hi}{\mathcal{H}}
\newcommand{\tr}{\mathrm{tr\,}}
\newcommand{\WF}{\mathrm{WF}}
\newcommand{\supp}{\mathrm{supp\,}}
\newcommand{\spec}{\mathrm{Spec}}
\newcommand{\Ima}{\mathrm{Im\,}}
\newcommand{\Rea}{\mathrm{Re\,}}
\newcommand{\dist}{\mathrm{dist\,}}
\newcommand{\mO}{\mathcal{O}}
\newcommand{\C}{\mathds{C}}
\newcommand{\R}{\mathds{R}}
\newcommand{\N}{\mathds{N}}
\newcommand{\Z}{\mathds{Z}}
\newcommand{\cN}{\mathcal{N}}
\newcommand{\cF}{\mathcal{F}}
\newcommand{\cZ}{\mathcal{Z}}
\newcommand{\tG}{\tilde G}
\newtheorem{thm}{Theorem}
\newtheorem{cor}[thm]{Corollary}
\newtheorem{con}[thm]{Conjecture}
\newtheorem{prop}[thm]{Proposition}
\newtheorem{lem}[thm]{Lemma}
\newtheorem{defn}[thm]{Definition}
\newtheorem{rem}[thm]{Remark}
\theoremstyle{remark}

\numberwithin{equation}{section}
\usepackage{hyperref}
\hypersetup{pdfborder=0 0 0, 
	    colorlinks=true,
	    citecolor=blue,
	    linkcolor=blue,
	    urlcolor=blue,
	    pdfauthor={Stephane Nonnenmacher, Martin Vogel}
	   }
\title[Local eigenvalue statistics]{Local eigenvalue statistics of one dimensional random nonselfadjoint pseudodifferential operators}
\author{St\'ephane Nonnenmacher}
\author{Martin Vogel} 
\address[St\'ephane Nonnenmacher]{Laboratoire de Math\'ematiques d'Orsay, Univ. Paris-Sud, CNRS, Universit\'e Paris-Saclay, 91405 Orsay, France}
\email{\href{stephane.nonnenmacher@u-psud.fr}{stephane.nonnenmacher@u-psud.fr}}
\address[Martin Vogel]{Mathematics Department, University of California, Berkeley, 
887 Evans Hall, Berkeley CA 94720, USA}
\email{\href{mailto:vogel@math.berkeley.edu}{vogel@math.berkeley.edu}}
\begin{document}
\maketitle
\begin{abstract}
We consider a class of one-dimensional nonselfadjoint semiclassical pseudo-differential 
operators, subject to small random perturbations, and study the statistical properties of their (discrete) spectra, in the semiclassical limit $h\to 0$. We compare two types of random perturbation: a random potential vs. a random matrix. 
Hager and Sj\"ostrand had shown that, with high probability, the local spectral density of the perturbed operator follows a semiclassical form of Weyl's law, depending on the value distribution of the principal symbol of our pseudodifferential operator.%
\par
Beyond the spectral density, we investigate the full local statistics of 
the perturbed spectrum, and show that it satisfies a form of universality:
the statistical only depends on the local spectral density, and of the type of random perturbation, but it is independent of the precise law of the perturbation. This local statistics can be described in terms of the Gaussian Analytic Function, a classical ensemble of random entire functions.
\end{abstract}
\tableofcontents
\section{Introduction}\label{sec:Intro}
The spectral analysis of linear operators defined on a Hilbert space is much more developed in the case of selfadjoint operators: one can then use powerful tools, like the spectral theorem, or variational methods. This fact has been very useful in mathematical physics, for example in quantum mechanics, where the "natural" operators are selfadjoint.
However, nonselfadjoint operators also appear in mathematical physics as well, and deserve to be investigated. For instance, in quantum mechanics, the study of scattering systems naturally leads to the concept of quantum resonances, which appear as the (complex valued) poles of the analytic continuation of the scattering matrix (or of the resolvent of the Hamiltonian) into the "nonphysical sheet" of the complex energy plane. These resonances may also be obtained as {\it bona fide} eigenvalues of a nonselfadjoint operator, obtained from the initial selfadjoint Hamiltonian through a complex deformation \cite{AgCo71}. Still in quantum mechanics, when considering the evolution of a "small system" in contact with a "large environment", one can be lead to express the effective dynamics of the small system through a nonselfadjoint {\it Lindblad operator} \cite{Lind76}. in classical statistical physics, the linear operator generating the evolution of the system are often nonselfadjoint: the Fokker-Planck, or the linearized Boltzmann equation typically contain convective as well as dissipative terms, leading to nonselfadjoint operators. In hydrodynamics, the operators appearing when linearizing the Navier-Stokes equation in the vicinity of some specific solution are generally not selfadjoint.
\par
When studying evolution problems generated by linear operators, one is naturally lead to analyze the spectrum of that operator, be it selfadjoint or not. Yet, in the nonselfadjoint case, establishing a connection between the long time evolution and a spectrum of complex eigenvalues is not so obvious as in the selfadjoint case, since eigenstates do not form an orthonormal family. This difficulty of connecting spectrum and dynamics is linked with a related characteristics of nonselfadjoint operators, namely the
possible strong instability of their spectrum with respect to small perturbations, a phenomenon nowadays commonly called "pseudospectral effect".
Traditionally this spectral instability was considered as a drawback, since it can be at 
the source of immense numerical errors, see \cite{TrEm05}. However, as we will see below,
analyzing this instability can also exhibit interesting phenomena.
Numerical analysis studies, e.g. by L.N.~Trefethen \cite{Tr97}, somewhat changed the perspective of this instability problem: they showed that
considering the {\it pseudospectrum} of the (nonselfadjoint) operator --- that is
the region where the resolvent operator is large --- is often more relevant than considering its spectrum, and can reveal important dynamical informations. 
As an example, when studying a certain class of nonlinear diffusion equations, Sandsteede-Scheel \cite{SaSc05}, Raphael-Zworski \cite{RaZw} 
and Galkowski \cite{Ga14} showed that the pseudospectrum of the (nonselfadjoint)
linearization of the equation can explain the finite time blow-up of the solutions to the full nonlinear equation. 
\par 
In physical situations, perturbation of the "pure" nonselfadjoint operator can originate from many different sources, most of them uncontrolled. Hence, it seems natural to set up a model of  perturbation by {\it random} operators, and to investigate how the spectrum of our operator reacts upon the addition of such perturbations. The spectrum of the full operator thereby becomes random; in case the spectrum is discrete, it forms a random point process on the complex plane, which can be investiaged by probabilistic methods.
This is what we will do in this article, for a particular class of "pure" nonselfadjoint operators. Not all nonselfadjoint operators enjoy the same level spectral instability: some operators are very sensitive, others are less so. The operators we will consider in this work are semiclassical pseudodifferential operators with complex valued symbols, and with some ellipticity assumption ensuring that the spectrum is discrete (at least in some region of the complex plane). In the semiclassical regime, the spectrum of such an operator is very sensitive to perturbations; quite often, the spectrum of the "pure" operator is localized along 1-dimensional curves, while the perturbed spectrum fill up the classical spectrum, which is a domain of $\mathbb{C}$. This "filling up" of the classical spectrum has been precisely studied in a series of works by Hager \cite{Ha06,Ha06b}, Sj\"ostrand
 \cite{Sj08,Sj09,HaSj08}, Bordeaux-Montrieux \cite{BM} (see also  \cite{ZwChrist10} for a similar phenomenon in the framework of Toeplitz operators on the 2-torus). These authors show that the randomly perturbed spectrum satisfies, with high probability, a complex valued version of Weyl's law: the density of eigenvalues near a point $z_0$ inside the classical spectrum, is approximately given by $(2\pi h)^{-1}\,D(z_0)$, where $D(z_0)$ is the {\it classical density} at the "energy" $z_0$, associated with the principal symbol of our operator.
 \par
This Weyl's law counts the eigenvalues in any {\it macroscopic} region of $\mathbb{C}$, so it describes the spectrum at the macroscopic scale. Since the mean density is of order $h^{-1}$, it is reasonable to think that the typical distance between nearest eigenvalues should be of order $h^{1/2}$, which we will call the {\it microscopic scale}. Our aim in this article is to investigate the distribution of eigenvalues at this microscopic scale, from a statistical point of view; in other words, we aim at studying the {\it local spectral statistics} of this family of randomly perturbed operators, in particular the type of statistical interaction between nearby eigenvalues; as a first result on this spectral statistics has been obtained by the second named author, who computed the 2-point correlation between the eigenvalues of our randomly perturbed operator \cite{Vo17}. In this article we will give an essentially complete description of this local statistics in terms of the zeros of certain Gaussian analytic functions, and prove a partial form of {\it universality}.

Before stating our result more precisely, and to motivate them, let us recall some background on the topic of spectral statistics, from a mathematical physics perspective. 
In the 1950s Wigner had the idea, when studying the spectra of nuclear Hamiltonians, to replace these complicated operators by large random matrices \cite{Wig55}. The random matrix model could not reproduce the large scale density fluctuations of the nuclear spectra, which depend on specific features of the system, but they could (empirically) reproduce the local statistical properties of the spectra, at the scale of the mean spacing between eigenvalues. Wigner and Dyson understood that these local statistical properties only depend on global symmetries of the Hamiltonian, like time reversal invariance, but not on the fine details of the nuclear Hamiltonian: these statistical properties were thus said to be {\it universal} \cite{Dy62}. In the 1980s, this universality conjecture was extended to simpler Hamiltonians, namely Laplacians on Euclidean domains with specific shapes: Bohigas-Giannoni-Schmidt observed that if the geodesic flow in the domain is "chaotic", then the local spectral statistics of the corresponding Laplacian correspond to Dyson's Gaussian Orthogonal ensemble \cite{BoGiSch84}. In parallel, a large variety of non-Gaussian random Hermitian matrix ensembles were developed and studied, notably the Wigner random matrices (all entries are i.i.d. non-Gaussian, up to Hermitian symmetry), for which the local spectral statistics was recently shown to be identical with that of the Gaussian ensembles \cite{Erd+10}. 

How about nonselfadjoint operators? Various random ensembles of
nonhermitian matrices have also been introduced in the theoretical
physics literature. The main objective has been to understand the
distribution of quantum resonances for various types of scattering or
dissipative systems, see for instance
\cite{FySo97,ZySo00,KeNoSch08,GoeSk11} (a short recent review can be
found in \cite{Fy16}). For most of these models, the focus has been to
derive the mean spectral density, without investigating the
correlation between the eigenvalues. The "historical" nonhermitian
random matrix model, for which the full eigenvalue statistics has been
derived in closed form, is the complex Ginibre ensemble \cite{Gin65},
where all entries are i.i.d. complex Gaussian; the nearby eigenvalues
then exhibit a statistical repulsion between the nearby eigenvalues,
similar to the case of Dyson's GUE ensemble of hermitian matrices. For
certain non-Gaussian ensembles, recent results \cite{BoYaYi14} have
been obtained on the eigenvalue distribution at the microscopic scale,
but they still fail to prove the universality of the local
statistics. Before coming to our work, let us mention a model studied
recently by Capitaine and Bordenave \cite{BorCap16} (see also \cite{DaHa09}), namely the case of a large Jordan block perturbed by a random Ginibre matrix: the authors prove that most eigenvalues of the perturbed matrix lie close to the unit circle, but they also prove that the "outliers" (the rare eigenvalues away from the unit circle) are distributed like the zeros of a "hyperbolic" Gaussian analytic function (GAF). Our results will involve instead a "Euclidean" Gaussian analytic function, yet the mechanisms through which GAFs appear in these two models are very similar. 

\subsection{Presentation of the results for a simple model case}\label{s:model-case}
Before stating our results in full generality, we will illustrate them by first focussing on a simple case. 
Call $h\in]0,1]$ the effective Planck's "constant", and consider the nonselfadjoint semiclassical harmonic oscilator 
\begin{equation}\label{eq0.1}
	P_h \defeq  -h^2\partial_x^2 + i x^2 \quad \text{ on } L^2(\R)\,.
\end{equation}
The (semiclassical) principal symbol of $P_h$ is given by the function
\begin{equation}\label{eq0.2}
	p(x,\xi) = \xi^2 + ix^2 \quad\text{on the phase space $\in \R^2\ni \rho=(x,\xi)$.}
\end{equation}
We call the set 
\begin{equation}\label{eq0.3}
	\Sigma \defeq  \overline{p(\R^2)}\subset \C \quad\text{the classical spectrum of $P_h$}.
\end{equation}
Here $\Sigma$ is obviously the upper right quadrant of $\C$. The spectrum  of 
$P_h$ is purely discrete, and is  contained in $\Sigma$ (actually, it is explicitly given by $\{z_n=e^{i\pi/4}h(n+1/2);\ n\in\N\}$). Take $\Omega\Subset\overset{\circ}{\Sigma}$ an open 
subset. Then, for any $z=X+iY\in\Omega$, an important data for our construction will be the structure of the "energy shell" \footnote{We will refer to the values $p(x,\xi)$ as "energies", eventhough they are complex.} $p^{-1}(z)\subset \R^2$. Since $p:\R^2\to \C$ is a local diffeomorphism for $z\in\Omega$, this energy shell made of discrete points; in the case of the harmonic oscillator, $p^{-1}(z)$  consists of the 4 points:%
\begin{equation}\label{eq0.4}
\rho^1_+=(Y^{1/2},-X^{1/2}),\,\rho^2_+= (-Y^{1/2},X^{1/2}),\,\rho^1_-=(Y^{1/2},X^{1/2}),\, \rho^2_+= (-Y^{1/2},-X^{1/2})\,. 
\end{equation}
We have labelled those points according to the sign of the Poisson bracket $\{\Rea p, \Ima p\}(\rho)= 4x\xi$: at the points $\rho^j_+$ the bracket is negative, while at the points $\rho^j_-$ it is positive. 
From this bracket condition, one can construct \cite{Da99,NSjZw04}, for each $j=1,2$ a semiclassical family of functions $(e_{+}^j(z,h) \in L^2(\R))_{h\in ]0,1]}$, $\|e_{+}^j(z,h)\|=1$, satisfying
\begin{equation}\label{eq0.5}
	 \|(P_h-z)e_+^j(z,h)\|=\mO(h^{\infty}),
\end{equation}
and such that $e_+^j(z,h)$ is microlocalized at the point $\rho^j_+(z)$\footnote{This microlocalization means that
the function $x\mapsto e_{+}^j(x;z,h)$ is concentrated near $x_{+}^j(z)$ when $h\to 0$,  while its semiclassical Fourier transform $(\cF_h e_{+}^j)(\xi;z,h)$ is concentrated near $\xi_{+}^j(z)$. }. (Here and in the entire text, all norms without index are either norms in $L^2$ or in
$\mathcal{B}(L^2)$, the set of bounded linear operators $L^2\to L^2$). 
We call each family $(e_{+}^j(z,h))$  an $h^{\infty}$-quasimode 
of $P-z$, or for short a quasimode of $P-z$.
Similarly, there exists quasimodes $e_{-}^j(z,h)\in L^2(\R)$
for the adjoint operator $(P_h-z)^*$, microlocalized at the points $\rho^j_-(z)$.
From the quasimode equation \eqref{eq0.5} it is easy to exhibit an operator $Q$ of norm unity and a parameter $\delta=\mO(h^{\infty})$, such that the perturbed operator $P_h+\delta Q$ has an eigenvalue at $z$ (for instance, if we call the error $r_+^j=(P_h-z)e_+^j$, we may take the rank 1 operator $\delta Q=- r_+^j \otimes (e^j_+)^*$). The set $\Omega$ (and hence the full interior of $\Sigma$, since $\Omega$ was chosen arbitrarily) 
is thus a zone of strong spectral instability for $P_h$. For this reason $\overset{\circ}{\Sigma}$ is called the ($h^\infty$-)pseudospectrum of $P_h$.
\\
\par 
Let us now explain how we construct \emph{random} perturbations, following \cite{HaSj08}. 
Let $\{e_k\}_{k\in\N}$  denote an orthonormal basis of $L^2(\R)$ comprised out of the eigenfunctions 
of the \emph{nonsemiclassical} harmonic oscillator $H=-\partial_x^2 + x^2$, and let $\{v_j\}_{j\in\N}$, $\{q_{jk}\}_{j,k\in\N}$ be independent and identically distributed (i.i.d.) complex Gaussian random 
variables with expectation $0$ and variance $1$ (that is, with distribution $\cN_{\C}(0,1)$). Let $N(h)=C_1/h^2$, with $C_1>0$ 
large enough. We define two types of random operators $Q$:
\begin{enumerate}
	\item \textbf{A random, Ginibre-type matrix}
		\begin{equation}\label{e:Ginibre}
		Q=M_{\omega}=\frac{1}{N(h)}\sum_{0\leq j,k<N(h)} q_{j,k} e_j \otimes e_k^* 
		\end{equation}
	\item \textbf{A random (complex valued) potential}
		\begin{equation*}
		Q=V_{\omega}=\frac{1}{N(h)}\sum_{0\leq j<N(h)} v_{j} e_j.
		\end{equation*}
\end{enumerate}
The coupling variable $\delta=\delta(h)$ will be assumed to be in the range 
\begin{equation}\label{e:delta}
h^{M} \leq \delta \leq h^{\kappa},
\end{equation}
where $\kappa > 3$, and $M>\kappa $ is an arbitrarily large  but fixed constant.
Although the random operator $Q$ and $\delta$ depend on $h$, we will skip the dependence in our notations.
We are interested in the spectrum of the perturbed operator
\begin{equation*}
	P_h^{\delta} = P_h + \delta Q, 
\end{equation*}
where the random operator $Q$ is either $M_{\omega}$ or $V_{\omega}$. With probability exponentially close to 
$1$,  $\|M_{\omega}\|_{\mathrm{HS}}, \|V_{\omega}\|_{\infty} \leq Ch^{-1}$ \cite{Ha06b,HaSj08}. 
\smallskip

Our objective will be to study the spectrum of $P_h^{\delta}$ in a microscopic neighbourhood of some given point $z_0\in\Omega$. As explained in the previous section, the probabilistic Weyl's law proved in \cite{Ha06b,HaSj08} shows that the mean density of eigenvalues near $z_0$ is of order $h^{-1}$, so we expect nearby eigenvalues to be at distances $\sim h^{1/2}$ from one another. In order to test the statistical interaction between nearby eigenvalues, we zoom to this scale $h^{1/2}$ at the point $z_0$, and define the rescaled spectral point process:
 \begin{equation*}
  	\mathcal{Z}_{h,z_0}^Q \defeq \sum_{z\in\spec (P_h+\delta Q)} \delta_{(z-z_0)h^{-1/2}}\,.
 \end{equation*}
Our main result is to show that, in the semiclassical limit $h\to 0$, this rescaled point process converges in distribution the point process formed by the zeros of a certain
random analytic function. In order to state our result, we need to define the building block of these random analytic functions, which is the (Euclidean) Gaussian analytic function (GAF). 
\subsection{The Gaussian analytic function}\label{s:GAF}
Let $(\alpha_n)_{n\in\N}$ be independent and identically distributed 
normal complex Gaussian random variables, i.e. $\alpha_n\sim\mathcal{N}_{\C}(0,1)$ for 
every $n\in\N$. For a given $\sigma>0$, we consider the random entire series  
	\begin{equation}\label{eq2.0}
		g_{\sigma}(w)\defeq \sum_{n=0}^{\infty} \alpha_n\frac{\sigma^{n/2}w^{n}}{\sqrt{n !}}, \quad 
		z\in\C\,.
	\end{equation}
With probability one, this series converges absolutely on the full
plane, and defines a Gaussian analytic function (GAF) on $\C$: $g_{\sigma}$ 
is a random entire function, so that for any $n\in\N$ and any $w_1,\dots,w_n\in\C$ 
the random vector $(g(w_1),\dots,g(w_n))$ is a centred complex Gaussian
\begin{equation}
	(g_{\sigma}(w_1),\ldots,g_{\sigma}(w_n)) \sim \mathcal{N}_{\C}(0,\Gamma),
\end{equation}
where the \emph{covariance matrix} $\Gamma\in GL_n(\C)$ has the entries
\begin{equation}
	\Gamma_{i,j} = \erw\left[g_{\sigma}(w_i)\overline{g_{\sigma}(w_j)}\right] 
	\overset{\mathrm{def}}{=}K_\sigma(w_i,\overline{w}_j) 
	= \exp(\sigma w_i \overline{w}_j)\,.
\end{equation}
The function $\C^2\ni(u,v)\mapsto K_\sigma(u,\overline{v})$ is called the \emph{covariance kernel}
of the GAF $g_{\sigma}$, it completely determines its distribution. As a result, $K_\sigma$ also
completely determines the distribution of 
\begin{equation*}
	\mathcal{Z}_{g_\sigma} \defeq \sum_{w\in g^{-1}_{\sigma}(0)} \delta_w,
\end{equation*}
the random point process given by zeros of the GAF $g_{\sigma}$, see for instance 
\cite{HoKrPeVi09}. 
In Section \ref{sec:RAF}, we will review basic notions and results 
concerning zero point processes of random analytic functions, making the above statements more precise. 

The GAF zero process has interesting geometric properties. Indeed, its covariance kernel shows that for any $w_0\in\C$, the translated function $g_\sigma(w+w_0)$ is equal in distribution to the function 
$e^{\sigma (w\bar{w}_0+|w_0|^2)}g_\sigma(w)$, which has the same zeros as $g_\sigma(w)$: this implies that the zero process $\mathcal{Z}_{g_\sigma}$ is invariant by translation. The average density (1-point function) of $\mathcal{Z}_{g_\sigma}$ is thus constant over the plane, it is equal to $\sigma/\pi$ (see section~\ref{s:2-point}). The linear dependence in $\sigma$ is coherent with the scaling covariance $g_\sigma(w) \stackrel{d}{=} g_1(\sqrt{\sigma} w)$: dilating the zero process $\mathcal{Z}_{g_1}$ by $1/\sqrt{\sigma}$ multiplies the average density by $\sigma$.
\medskip

Let us give a short historical background of the GAF. It has appeared before in the context of holomorphic representations
of quantum mechanics, when investigating the properties of {\it random states}. 
In the framework of Toeplitz quantization on a compact K\"ahler manifold $M$, one defines a positive line bundle $L$ over $M$, and for any integer $N\geq 1$ 
a "quantum" Hilbert space $\mathcal{H}_N$ is formed by the holomorphic sections on the bundle $L^{\otimes N}$; the limit $N\to\infty$ is then a form of semiclassical limit. 
In the case of the 1-dimensional projective space $M=\C P^1$, which is the phase space of the quantum spin, Hannay \cite{Ha96} defined a natural ensemble of random holomorphic sections in $\mathcal{H}_N$, and studied the point process formed by their zeros (topological constraints impose that any section has exactly $N$ zeros). He explained how to compute the $k$-point correlation function of this process, and computed explicitly the limit (after microscopic rescaling) of the 2-point correlation function, which coincides with the 2-point function of the GAF. A few years later, Bleher-Schiffman-Zelditch \cite{BlShiZe00} proved that, for a general Toeplitz quantization $(M,L)$, the zeros of random holomorphic sections converge, when $N\to\infty$, to a {\it universal} process depending only on the dimension of $M$. In dimension $1$, this process is given by the zero process of the GAF. 

\smallskip

We are now equipped to state our Theorem concerning the spectrum of $P_h^\delta$.

\begin{thm}[Complex harmonic oscillator]\label{thm01} 
Fix  $z_0=X_0+iY_0\in\Omega\Subset\mathring{\Sigma}$, and   
define the classical density for the symbol $p(x,\xi)=\xi^2+ix^2$ at the points $\rho^j_{\pm}\in p^{-1}(z_0)$:
 \begin{equation*}
\sigma(z_0)  \defeq \frac{1}{ | \{\Rea p,\Ima p\}(\rho_{\pm}^j(z_0)) | }=\frac{1}{4 \sqrt{X_0 Y_0}}\,.
 \end{equation*}
Let the perturbation $Q$ be either $M_{\omega}$ or $V_{\omega}$, and take $\delta$ as in \eqref{e:delta}.
Then, for any bounded domain $O\Subset\C$, we have the convergence in distribution of the spectral point processes:
 \begin{equation*}
  	\mathcal{Z}_{h,z_0}^Q
	\stackrel{d}{\longrightarrow} 
	\mathcal{Z}_{G_{z_0,Q}} \quad  \text{ on } O, \quad 
	\text{ as } h\to 0.
 \end{equation*}
(This convergence of point processes is explained in Thm~\ref{thm_m2}).  
Here  $\mathcal{Z}_{G_{z_0,Q}}$ is the zero point process for the random entire function $G_{z_0,Q}$ described below:
 \begin{enumerate}
  \item if the perturbation $Q=V_{\omega}$ then 
 \begin{equation*}
 G_{z_0,V}(w)=g^1_{z_0}(w)g^2_{z_0}(w), \quad w\in \C, 
 \end{equation*}
 where $g^1_{z_0},g^2_{z_0}$ are two independent copies of the GAF $g_{\sigma(z_0)}$.%
\item if $Q=V_{\omega}$ then 
 \begin{equation*}
 G_{z_0,M}(w)=\det \big(g_{z_0}^{i,j}(w)\big)_{1\leq i,j \leq 2}, \quad w\in \C, 
 \end{equation*}
 where $g_{z_0}^{i,j}$, $1\leq i,j\leq 2$, are 4 independent copies of the GAF $g_{\sigma(z_0)}$.  %
\end{enumerate}
\end{thm}
As we will explain below (see section~\ref{s:k-point}), the convergence of point processes implies that all $k$-point measures converge as well to the
limiting ones.
%
\section{Main results -- general framework}\label{sec:MainRes}
The above theorem can be generalized to a large class of 1-dimensional nonselfadjoint $h$-pseudodifferential operators, and with random perturbations which are not necessarily Gaussian. We first present the class of operators we will be dealing with. 

\subsection{Semiclassical framework}
We begin by recalling the definition of  the \emph{pseudospectrum} of an operator, an important notion which quantifies its spectral instability.

For $P:L^2\to L^2$, a densely defined closed linear operator with resolvent 
set $\rho(P)$ and spectrum $\spec(P)=\C\backslash\rho(P)$. For any 
$\varepsilon>0$, we define the $\varepsilon$-pseudospectrum of $P$ by 
	\begin{equation}\label{eq1.1}
	\spec_{\varepsilon}(P) \defeq \spec(P)\cup 
	\{z\in\rho(P); \|(P-z)^{-1}\| >\varepsilon^{-1}\}.
	\end{equation}
When $\varepsilon$ is small, the set \eqref{eq1.1} describes a region of spectral instability of the operator $P$, 
since any point in the $\varepsilon$-pseudospectrum of $P$ lies in the spectrum of a certain 
$\varepsilon$-perturbation of $P$ \cite{TrEm05}. Indeed, $\spec_{\varepsilon}(P)$ can also be defined by
\begin{equation}\label{eq1.2}
   \spec_{\varepsilon}(P)= \bigcup_{\substack{Q\in\mathcal{B}(L^2) \\
    					\lVert Q\rVert < 1}}
					\spec(P+\varepsilon Q).
\end{equation}
A third, equivalent definition of the $\varepsilon$-pseudospectrum of $P$ is via the existence of
approximate solutions to the eigenvalue problem $P-z$: 
\begin{equation}\label{eq1.3}
   z \in \spec_{\varepsilon}(P)
   \Longleftrightarrow 
   \exists u_z\in D(P) \text{ s.t. } \|(P-z)u_z\| < \varepsilon \|u_z\|, 
\end{equation}
where $D(P)$ denotes the domain of $P$. Such a state $u_z$ is called an 
$\varepsilon$-quasimode or simply a \textit{quasimode}. 
\\
\par
Next, let us fix the type of unperturbed operators we will consider in this paper. 
We will use the notation $\rho=(x,\xi)\in\R^2$ for phase space points. We start by considering
an \emph{order function} $m\in\mathcal{C}^{\infty}(\R^2; [1,\infty[)$:
 \begin{equation}\label{eq1.4}
 \exists C_0\geq 1,~ \exists N_0 >0 : \quad 
 m(\rho) \leq C_0 \langle \rho-\mu\rangle^{N_0} m(\rho), 
 \quad \forall \rho,\mu \in \R^2,
 \end{equation}
with the usual notation $\langle \rho-\mu\rangle\defeq\sqrt{1+|\rho-\mu|^2}$. 
To this order function is associated a semiclassical symbol class (cf \cite{DiSj99,Zw12}):
 \begin{multline}\label{eq1.5}
S(\R^2,m)=\left\{
q\in\mathcal{C}^{\infty}(\R^2_\rho\times ]0,1]_h); ~\forall \alpha\in\N^2, ~\exists C_{\alpha}: \right. \\
\left.	|\partial^{\alpha}_{\rho}q(\rho;h)|\leq C_{\alpha} m(\rho), ~ \forall \rho\in\R^2,\ \forall h\in]0,1]
\right\}.
 \end{multline}
We assume that the symbol  $p\in S(\R^2,m)$ is "classical", namely it satisfies an asymptotic expansion in the limit $h\to 0$:%
\begin{equation}\label{eq1.6}
	p(\rho;h) \sim p_0(\rho)+ hp_1(\rho) + \dots \quad \text{in } S(\R^2,m),
\end{equation}
where each $p_j\in S(\R^2,m)$ is independent of $h$. In this case
we call $p_0$ the (semiclassical) principal symbol of $p$. We then define two subsets of $\C$ associated with $p_0$:
\begin{equation}\label{eq1.8}
		\Sigma \defeq \overline{p_0(\R^2)} \subset \C, \qquad
		\Sigma_{\infty}\defeq \{
		z\in\Sigma; ~\exists (\rho_j)_{j\geq 1} \text{ s.t. } |\rho_j|\to\infty, ~ p_0(\rho_j)\to z
		\}.
\end{equation}
$\Sigma$ is the classical spectrum of the operator $P_h$ defined below.
Furthermore, we suppose that the principal symbol $p_0$ is \emph{elliptic} at some point 
$z_{out}\in\C\backslash \Sigma$: 
\begin{equation}\label{eq1.7}
\exists C_0>0,\quad	|p_0(\rho)-z_{out}|\geq m(\rho)/C_0, \quad \forall \rho\in\R^2.
\end{equation}
For $h\in ]0,1]$ we let $P_h$ denote the $h$-Weyl quantization of 
the symbol $p$, 
\begin{equation}\label{eq1.7.5}
		P_hu(x)=p^w(x,hD_x;h)u(x) = \frac{1}{2\pi h}\iint \e^{\frac{i}{h}(x-y)\cdot\xi}\,
		p\left(\frac{x+y}{2},\xi;h\right)u(y)dy d\xi, 
\end{equation}
which makes sense for $u\in \mathcal{S}(\R^d)$ the Schwartz space. The closure of $P_h$ as an unbounded 
operator on $L^2$, has domain 
$H(m)\defeq (P_h-z_{out})^{-1}(L^2(\R))\subset L^2(\R)$, we will still denote this closed operator by $P_h$. 
Moreover, we will denote by $\|u\|_m\defeq \|(P_h-z_{out}) u\|$ the associated norm on $H(m)$\footnote{Although this norm depends on the choice of the symbol $p-z_{out}$, it is equivalent to the norm defined from any elliptic operator in $q\in S(m)$, so that the space $H(m)$ only depends on the order function $m$.}.
\\
\par
Let $\widetilde{\Omega}\subset\C\backslash\Sigma_{\infty}$ be open simply connected, 
not entirely contained in $\Sigma$, and such that 
$\overline{\widetilde{\Omega}}\cap\Sigma_{\infty}=\emptyset$. Then, 
the spectrum  of $P_h$ inside $\widetilde{\Omega}$ satisfies the following properties in the semiclassical limit \cite{Ha06b,HaSj08}:
	\begin{itemize}
	\item for $h>0$ small enough,  $\spec (P_h)\cap\widetilde{\Omega}$ is discrete 
	\item for all $\varepsilon>0$, $\exists h(\varepsilon)>0$ such that 
		\begin{equation}
			\spec(P_h)\cap \widetilde{\Omega} \subset \Sigma + D(0,\varepsilon),
			\quad \forall 0<h<h(\varepsilon).
		\end{equation}
	\end{itemize}
Here, $D(0,\varepsilon)\subset\C$ denotes the open disc of radius $\varepsilon>0$ 
centred at $0$. In this work we will study the spectrum of small random perturbations of 
$P_h$, in the semiclassical limit $h\to 0$, in the interior of $\Sigma$.
\subsection{Pseudospectrum and the energy shell}\label{sec:PsEs}
Let $\widetilde{\Omega}$ be as above and let 
\begin{equation}\label{eq1.8.2}
\Omega\Subset\widetilde{\Omega}\cap\mathring{\Sigma} 
\text{ be  open, simply connected.}
\end{equation}
Recall that $p_0$ is the principal symbol of $p$, see \eqref{eq1.6}. We also assume that:
\begin{equation}\label{eq1.0}
	\text{for every } \rho \in p_0^{-1}(\overline{\Omega}),\  \text{the $1$-forms }dp_0,d\overline{p_0} \text{ are linearly independent.}
\end{equation}
Since the dimension $d=1$, this condition is equivalent to: 
\begin{equation}\label{eq1.0.1}
	\text{ for every } \rho \in p_0^{-1}(\overline{\Omega}),\ 
	\{\Rea p_0,\Ima p_0 \}\neq 0 ,
\end{equation}
where $\{\cdot,\cdot\}$ denotes the Poisson bracket of the two functions:
$$
\{p,q\}(\rho)\defeq \partial_{\xi}p(\rho)\,\partial_xq(\rho) - \partial_{\xi}q(\rho)\,\partial_x p(\rho),\qquad \rho=(x,\xi)\in\R^2\,.
$$

It was observed by Dencker, Sj\"ostrand and Zworski \cite{NSjZw04}, and 
Sj\"ostrand \cite{Sj16} that since $\Omega$ is relatively compact and simply connected, 
\eqref{eq1.0}, or equivalently \eqref{eq1.0.1}, implies that 
there exists $J\in\N^*$ depending only on $\Omega$, so that for any $z\in\Omega$, the "energy shell" $p_0^{-1}(z)$ consists of exactly $2J$ points: 
\begin{equation}\label{eq1.8.1}
	\begin{split}
	p_0^{-1}(z)=\{\rho_{\pm}^j(z); j=1,\dots,J\},
	&\text{ with }\pm \{\Rea p,\Ima p\}(\rho_{\pm}^j(z)) <0, \\ 
	& \rho_{\pm}^i(z)\neq \rho_{\pm}^j(z) \ \text{ if } \  i\neq j,
	\end{split}\tag{HYP}
\end{equation}
and the points $\rho_{\pm}^j(z)\in\R^2$ depend smoothly on $z$. 
\par
We shall make the further (generic) assumption
\begin{equation}\label{eq1.8.1a}
	\forall z\in\Omega, \quad x_{\pm}^i(z)\neq x_{\pm}^j(z) \ \text{ if } \  i\neq j\,,\tag{HYP-x}
\end{equation}
which will play a role when studying the perturbation by a random potential.%
 
It was shown by Davies \cite{Da99} and Dencker, Sj\"ostrand and 
Zworski \cite{NSjZw04} that \eqref{eq1.8.1} implies, for each $z\in\Omega$ and each $j=1,\ldots,J$,
the existence of an $h^{\infty}$-quasimode for the problem 
$P_h-z$ (resp. $(P_h-z)^*$), microlocalized on $\rho_{+}^j(z)$ (resp. $\rho_{-}^j(z)$): there exist $e_{\pm}^j=e_{\pm}^j(x,z;h) \in L^2(\R)$,  
$\|e_{\pm}^j\|=1$, such that 
\begin{equation}\label{eq1.8.3}
	 \|(P_h-z)e_+^j\|=\mO(h^{\infty}) \quad \text{and} \quad 
	 \WF_h(e_+^j)=\{\rho_+^j(z)\},
\end{equation}
respectively
\begin{equation}\label{eq1.8.4}
	 \|(P_h-z)^*e_-^j\|=\mO(h^{\infty}) \quad \text{and} \quad 
	 \WF_h(e_-^j)=\{\rho_-^j(z)\}.
\end{equation}
Recall that, for  $u=(u(h))_{h\in ]0,1]}$ a bounded family in $L^2$, its semiclassical wavefront set $\WF_h(u)$ is defined by 
\begin{equation*}
 \WF_h(u)\stackrel{\mathrm{def}}{=}\complement
 \left\{(x,\xi)\in T^*\R ; ~ \exists 
 a\in\mathcal{C}_c^{\infty}(T^*\R),~a(x,\xi)=1,~
 \lVert a^w(x,hD_x) u(h)\rVert_{L^2}=\mO(h^{\infty})
 \right\}
\end{equation*}
where $a^w$ denotes the Weyl quantization of $a$. 
\par
In view of the characterisation \eqref{eq1.3} of the pseudospectrum, we see 
that the assumption \eqref{eq1.0} implies that $\Omega$ is contained in the 
$h^{\infty}$-pseudospectrum of $P_h$, a spectrally highly unstable region. 
\subsection{Adding a random perturbation}\label{sec:AddRandPert}
We are interested in random perturbations of the operator $P_h$ which are given by 
either a random matrix or a random potential. We will construct 
those in the following way, which generalizes the construction made in section~\ref{s:model-case}. 
We consider $\{e_k\}_{k\in\N}$ the orthonormal eigenbasis of the (nonsemiclassical) harmonic oscillator $H=-\partial_x^2 + x^2$ on $L^2(\R)$.
\begin{rem}
The choice of this orthonormal basis is a convenient one for 
us. However, it will become clear later on in the text that what we need is 
a system of states (not necessarily orthonormal) such that the 
first $N(h)$ states microlocally cover a sufficiently large part of 
phase space, namely a neighbourhood of $p_0^{-1}(\Omega)$. We also need to avoid states which would 
have a large overlap with some of the quasimodes $e_{\pm}^j$, cf. \eqref{eq1.8.3}, 
\eqref{eq1.8.4}. We refer the reader in particular to the proofs of Propositions 
\ref{lem6.3.3} and \ref{prop8.2} below.
\end{rem}
Let $\alpha$ be a complex valued random variable defined on some probability space $(\mathcal{M},\mathcal{F},\prob)$, %
with the properties
	\begin{equation}\label{eq1.9}
		\erw[\alpha] = 0, \quad \erw[\alpha^2] = 0,\quad \erw[|\alpha|^2]=1, 
		\quad  \erw[|\alpha|^{4+\epsilon_0}] < +\infty,
	\end{equation}
where $\varepsilon_0>0$ is an arbitrarily small but fixed constant. Here, 
$\erw[\cdot]$ denotes the expectation with respect 
to the probability measure $\prob$. 
The Markov inequality implies the following tail estimate: 
there exists a constant $\kappa_{\alpha}>0$ such that 
	\begin{equation}\label{eq1.9.1}
		\prob[|\alpha |\geq\gamma] \leq \kappa_{\alpha}\,\gamma^{-(4+\varepsilon_0)}, 
		\quad \forall \gamma > 0.
	\end{equation}
\begin{rem}
For instance, a complex centred Gaussian random variable as in eq.~\eqref{eq2.0} satisfies
the above assumptions.
\end{rem}
\noindent
\textbf{Random Matrix.} Let $N(h)=C_1/h^2$, $C_1>0$ large enough (we'll be more precise about this condition later). 
Let $q_{j,k}$,  $0\leq j,k <N(h)$ be independent copies of the random variable $\alpha$
satisfying the conditions \eqref{eq1.9}. We consider the "random matrix"
	\begin{equation}\label{eq1.10}
		M_{\omega}=\frac{1}{N(h)}\sum_{0\leq j,k<N(h)} q_{j,k}\, e_j \otimes e_k^*, \tag{RM}
	\end{equation}
where $e_j \otimes e_k^*u=(u|e_k)e_j$ for $u\in L^2(\R)$. 
For some coupling parameter $0<\delta\ll1$, we define the randomly perturbed perturbed operator 
\begin{equation}\label{eq1.11}
	P^{\delta}_M = P_h+\delta M_{\omega}. 
\end{equation}
\\[2ex]
\textbf{Random Potential.} Take $N(h)=C_1/h^2$, $C_1>0$ as above. Let $v_{j}$, 
$0\leq j <N(h)$ be independent copies of the random variable $\alpha$. Still using the same orthonormal family $(e_k)_{k\in\N}$, we define the random function  
	\begin{equation}\label{eq1.13}
		V_{\omega}=\frac{1}{N(h)}\sum_{0\leq j<N(h)} v_{j}\, e_j. \tag{RP}
	\end{equation}
For $0<\delta\ll1$, write the perturbed operator
\begin{equation}\label{eq1.14}
	P^{\delta}_V = P_h+\delta V_{\omega}.
\end{equation}
We call this perturbation a "random potential", eventhough $V_\omega$ is complex valued. 
When we consider this type of perturbation, we will make the additional 
symmetry assumption:
\begin{equation}\label{eq.15}
	p(x,\xi;h)=p(x,-\xi;h).\tag{SYM}
\end{equation}
This hypothesis implies that we can group the points $\rho_{\pm}^j$ forming $p_0^{-1}(z)$, see \eqref{eq1.8.1},  
such that $\rho_{\pm}^j = (x^j,\pm\xi^j)$; as a result, the centres of microlocalization 
of the quasimodes $e_+^j$ and $e_-^j$ are located on the same fibre $T^*_{x^j}\R = \{ (x_j,\xi), \, \xi\in\R\}$.
\begin{rem}
	We could relax the assumption \eqref{eq.15} into requiring this symmetry only at the 
	level of the principal symbol, i.e. $p_0(x,\xi)=p_0(x,-\xi)$. However, for the simplicity of the presentation we prefer to make the above stronger hypothesis. 
\end{rem}

\noindent\textbf{Restricting to bounded perturbations.} 
For both types of perturbations, it will be easier for us to restrict the random variables to large discs $D(0,C/h)$, i.e. 
assume that 
\begin{equation}\label{eq1.15.2}
		|v_{i}|, ~|q_{i,j}| \leq C/h, \quad      0\leq i,j <N(h),\quad \text{for some $C>0$ sufficiently large}.	    
\end{equation}
This restriction induces the boundedness of the perturbations $M_\omega$, $V_\omega$.
Using \eqref{eq1.9.1} to estimate the probability of this restriction, we deduce that the boundedness of the perturbations occurs with high probability.
Namely, there exists $C_2>0$ such that, for the random matrix, 
\begin{equation}\label{eq1.11.1}
\prob\big[ \|M_{\omega}\|_{HS} \leq Ch^{-1}\big] \geq 1 - C_2 h^{\varepsilon_0},
\end{equation}
where $\|M_{\omega}\|_{HS} $ denotes the Hilbert-Schmidt norm of $M_{\omega}$. Respectively, in the case of the random potential, 
\begin{equation}\label{eq1.15.1}
	\prob\big[ \|V_{\omega}\|_{\infty} \leq Ch^{-1}\big] \geq 1 - C_2 h^{2+\varepsilon_0}.
\end{equation}
We will take the coupling parameter $\delta$ as in \eqref{e:delta}.
\smallskip

We will see in Section \ref{sec:GP} that the spectra of $P^{\delta}_M$ and 
$P^{\delta}_V$ in $\Omega$ are purely discrete. The principal aim of this 
paper is to show that the statistical properties of these spectra, in a microscopic neighbourhood of any $z_0\in \Omega$, are universal, in a sense that we will specify later on. 
\par
Since $p_0-z$ is elliptic for every $z\in\C\backslash\Sigma$, we 
have that the resolvent norm $\|(P_h-z)^{-1}\| = \mO(1)$, uniformly as 
$h\to 0$. Therefore, in view of the characterisation \eqref{eq1.2} of the 
pseudospectrum, the spectra of $P^{\delta}_M$ and 
$P^{\delta}_V$ are contained in any neighbourhood $\Sigma+D(0,\epsilon)$, for any given $\epsilon>0$ and $h>0$ small enough.  
Moreover, since $\Omega\Subset\mathring{\Sigma}$, we will not feel the effects of the boundary of $\Sigma$; 
we will simply say that $\Omega$ lies in the \textit{bulk} of the spectrum of 
the perturbed operator. 
\subsection{Probabilistic Weyl's law and local statistics}
In a series of works by Hager \cite{Ha06,Ha06b,HaSj08} and Sj\"ostrand \cite{Sj08,Sj09}, 
the authors considered randomly perturbed operators $P^{\delta}$ as 
given in \eqref{eq1.11} and \eqref{eq1.14}. Under more restrictive assumptions on 
the random variables than \eqref{eq1.9}, they have shown the following result. 
\begin{thm}[Probabilistic Weyl's law]\label{thm:PWL}
Let $\Omega$ be as in \eqref{eq1.8.2}, \eqref{eq1.0}. Let $\Gamma\Subset\Omega$ 
be open with $\mathcal{C}^2$ boundary. Let $P^{\delta}$ be either of the randomly 
perturbed operators $P^{\delta}_M$ or $P^{\delta}_V$ with $\delta$ as in \eqref{e:delta}
 with $\kappa>0$ sufficiently large. Then, in the limit $h\to 0$,
\begin{equation}
	\#\big(\spec(P^{\delta})\cap \Gamma\big) = \frac{1}{2\pi h}
		\left( \iint_{p_0^{-1}(\Gamma)} dxd\xi + o(1) \right)\quad \text{with probability $\geq 1 - C h^{\eta}$},
\end{equation}
for some fixed $\eta>0$.
\end{thm}
Furthermore the authors give an explicit control 
over both the error term in Weyl's law, and the error term in the probability estimate.  
\begin{figure}[ht]
 \begin{minipage}[b]{0.45\linewidth}
  \centering
  \includegraphics[width=\textwidth]{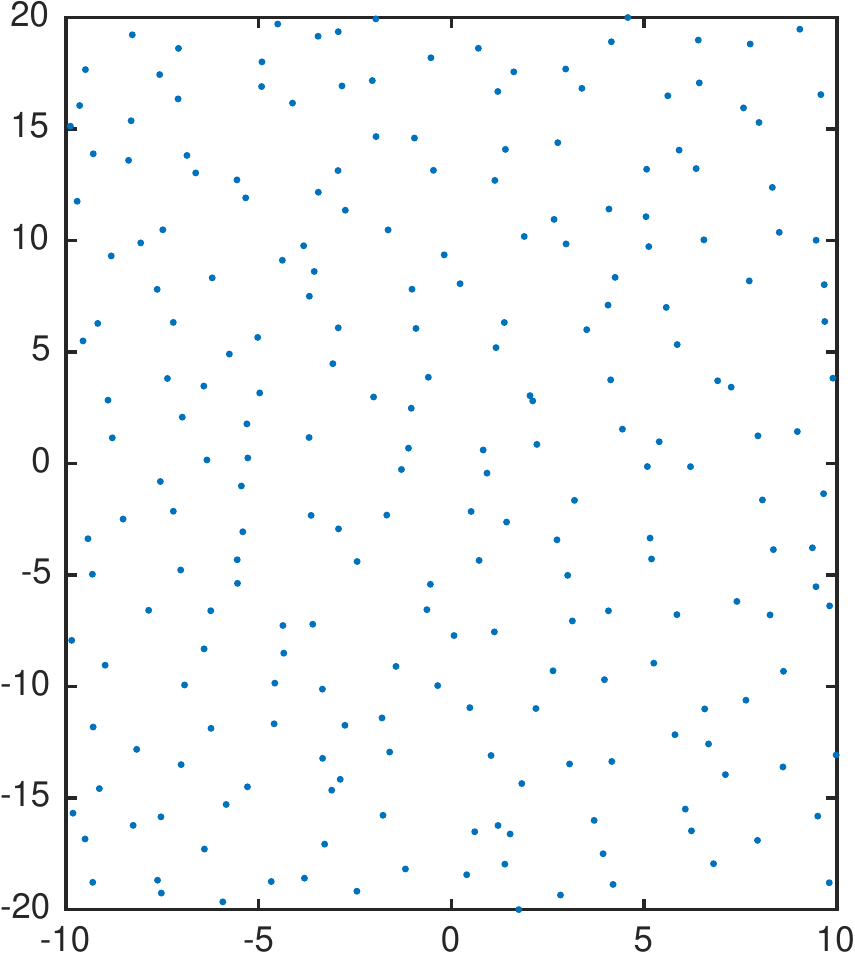}
 \end{minipage}
 \hspace{0.04\linewidth}
 \begin{minipage}[b]{0.45\linewidth}
  \centering 
  \includegraphics[width=\textwidth]{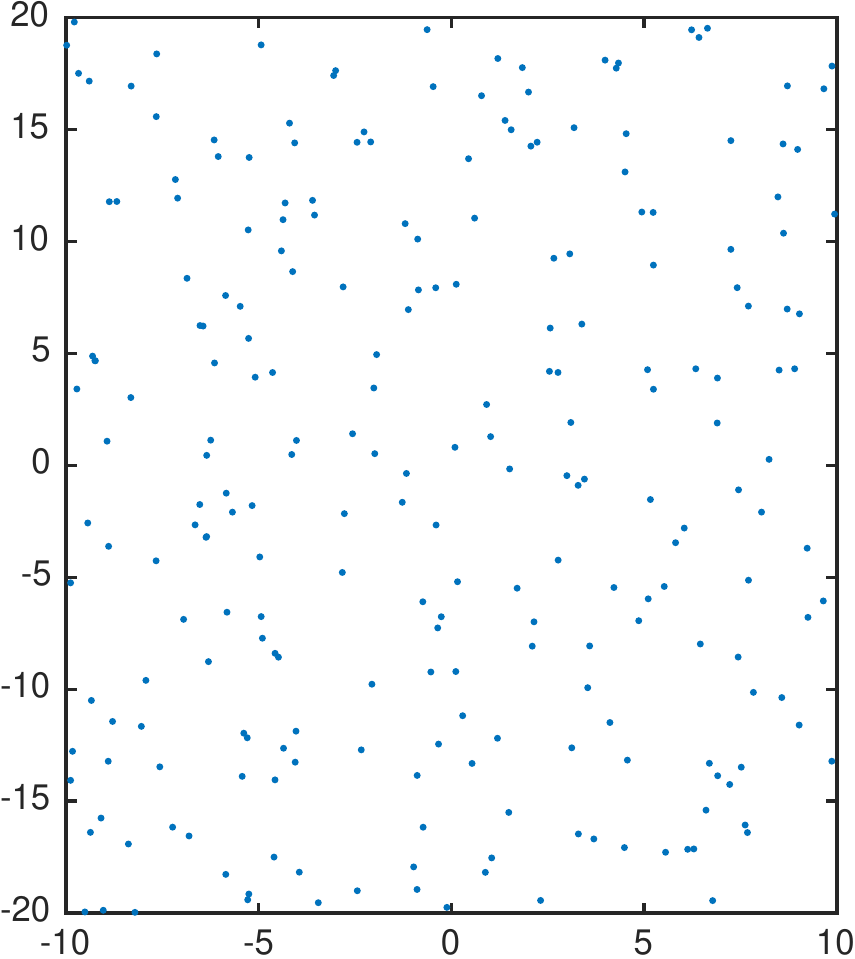}
 \end{minipage}
 \caption{Numerically computed spectra of the operator $-h^2\partial_x^2+\e^{3ix} +\delta Q$ on 
 	    $\mathbb{T}$, with $h=10^{-3}$ and $\delta= 10^{-12}$. The left plot shows the case where 
	    $Q$ is given by a complex Gaussian random 
	    matrix and the right plot the case where $Q$ is given by a random potential as in 
	    \eqref{eq1.13} with coefficients given by complex Gaussian random variables $\sim \cN_{\C}(0,1)$. }
	    \label{fig1}
\end{figure}
\par
This probabilistic Weyl's law shows that, with probability close $1$, the number 
of eigenvalues of the perturbed operator $P^{\delta}$ in any \emph{fixed} compact subset of $\Omega$ is 
of order $\asymp h^{-1}$. Hence, the spectrum of $P^{\delta}$ will spread out across
$\Omega$, with an average spacing between nearby eigenvalues of order $h^{1/2}$. 
Figure~\ref{fig1} illustrates this behaviour for some choice of perturbed operators $P^{\delta}_M$ 
and $P^{\delta}_V$. We observe in Figure~\ref{fig1} that in the case of $P^{\delta}_V$ (right plot),
the distribution of the eigenvalues looks "less uniform" than in the case of $P^{\delta}_M$, including the presence of small clusters of eigenvalues. 

To quantify this difference of "uniformity" between the spectra of $P^{\delta}_M$ 
and $P^{\delta}_V$, we study the local statistics of the eigenvalues, that is the statistics of the 
eigenvalues on the scale of their mean level spacing. 
For this purpose, we fix a point $z_0\in \Omega$. In both cases  $Q=M_\omega$ and $Q=V_\omega$, 
we view the rescale the spectrum of the randomly perturbed operator $P^\delta_Q$ as a random point process
 \begin{equation}\label{eq:PPM}
  	\mathcal{Z}_{h,z_0}^Q\defeq \sum_{z\in\spec (P^{\delta}_Q)}
				 \delta_{(z-z_0)h^{-1/2}},\qquad Q=M_\omega\ \ \text{or}\ \ Q=V_\omega,
 \end{equation}
where the eigenvalues are counted according to their algebraic multiplicities. 
\par
Notice that the rescaled eigenvalues $(z_j-z_0)h^{-1/2}$
have a mean spacing of order $\asymp 1$. 
The principal aim of this paper is to show that, under the assumption 
\eqref{eq1.9} on the random coefficients, in the limit $h\to 0$ 
the correlation functions of the processes $\mathcal{Z}_{h,z_0}^M$ and $\mathcal{Z}_{h,z_0}^V$ 
are \emph{universal}, in the sense that they %
\begin{itemize}
	\item \textit{depend} only on the structure of the energy shell $p_0^{-1}(z)$ and 
		on the type of random perturbation used, either $M_\omega$ or 
		$V_\omega$;
	\item \textit{are independent} of the random variables
		$q_{j,k}$ and $v_j$ used to define the random perturbations, 
		as long as they satisfy \eqref{eq1.9}. 
\end{itemize}
Finally, let us stress once more that our results concern solely the eigenvalues 
in the bulk of the spectra of $P^\delta_h$, that is
in the interior of the $h^{\infty}$-pseudospectrum of $P_h$. 
Near the boundary of the pseudospectrum, we expect the statistical properties of the 
eigenvalues to change drastically. It has been shown by the second
author \cite{Vo14} in the case of a model operator, that 
the probabilistic Weyl's law breaks down in the vicinity of $\partial\Sigma$, in fact, the density of eigenvalues explodes near that boundary.
\subsection{Perturbation by a random potential}
We begin with the case  of the operator $P^{\delta}_V$, \eqref{eq1.14}, perturbed by a  
random potential $V_{\omega}$. 
From the formula \eqref{eq1.8.1} defining the energy shell $p_0^{-1}(z)$, we find that the classical spectral density at the energy $z$ is given by 
\begin{equation}\label{eq2.1}
	(p_0)_*(|d\xi\wedge dx|) 
	=\sum_{j=1}^J(\sigma_+^j(z)+ \sigma_-^j(z)) L(dz), 
	\quad \sigma_{\pm}^j(z) = \frac{1}{\mp \{\Rea p_0,\Ima p_0\}(\rho_{\pm}^j(z))}\,.
\end{equation}
Here $|d\xi\wedge dx|$ denotes the measure induced by the symplectic volume form on $T^*\R \cong \R^2$. Each density
$\sigma_{\pm}^j(z)>0$ is associated with the point $\rho^j_{\pm}$ of the energy shell. These densities depend smoothly on $z$. 
\par
If we additionally assume the symmetry hypothesis \eqref{eq.15} and 
group the points such that $\rho_{\pm}^j = (x^j,\pm\xi^j)$, we find that
$\sigma_+^j(z)= \sigma_-^j(z)$ for all $j=1,\ldots,J$.
\subsubsection{Universal limiting point process}
Let us now state our main theorem for the perturbed operators $P^\delta_V$. It provides the asymptotic behaviour of the rescaled spectral point processes $\mathcal{Z}^V_{h,z_0}$ in the semiclassical limit.
\begin{thm}\label{thm_m2} 
 Let $\Omega\Subset\mathring{\Sigma}$ be as in \eqref{eq1.8.2},
 let $p$ be as in \eqref{eq1.6} satisfying \eqref{eq1.0} and \eqref{eq.15}. 
 Fix $z_0\in\Omega$. Then, for any open, connected, relatively compact 
 domain $O\Subset\C$, we have that 
 \begin{equation*}
  	\mathcal{Z}_{h,z_0}^V
	\stackrel{d}{\longrightarrow} 
	\mathcal{Z}_{G_{z_0}} \quad  \text{ in } O, \quad 
	\text{ as } h\to 0,
 \end{equation*}
This convergence of point processes means that for any test function $\phi\in\mathcal{C}_c(O,\R)$,
 \begin{equation*}
  	\langle \mathcal{Z}_{h,z_0}^V,\phi \rangle = \sum_{z\in\spec (P^{\delta}_h)} \phi((z-z_0)h^{-1/2})
	\stackrel{d}{\longrightarrow} 
	\langle \mathcal{Z}_{G_{z_0}},\phi \rangle = \sum_{z\in G_{z_0}^{-1}(0)}  \phi(z), \quad 
	\text{ as } h\to 0\,.
 \end{equation*}
 Here $\mathcal{Z}_{G_{z_0}}$ is the zero point process for the random analytic function
  \begin{equation*}
 G_{z_0}(z)=\prod_{j=1}^J g^j_{z_0}(z), \quad z\in \C, 
 \end{equation*}
 where the $g^{j}_{z_0}$ are $J$ independent GAFs $g^{j}_{z_0}\sim g_{\sigma_+^j(z_0)}$ (see section~\ref{s:GAF}), with $\sigma_{+}^i(z_0)$ the local spectral density computed in in \eqref{eq2.1}.
\end{thm}
For the reader's convenience, in Section \ref{sec:RAF} we present a short review of the 
probabilistic notions used in this paper, such as convergence in distribution. The definition and basic properties of the GAFs have been presented in section~\ref{s:GAF}.
\par
This theorem tells us that at any given point 
$z_0\in\Omega$ in the bulk of the 
pseudospectrum, the local rescaled point process of
eigenvalues is given, in the limit $h\to 0$, by the point process given by the zeros 
of the product of $J$ independent GAFs. Hence, this limiting point process is the superposition of $J$ independent point processes, 
each one generated by a GAF $g^j$. The latter only depends only on the part of the 
 classical spectral density coming from the pair of points 
 $\rho_{\pm}^j = (x^j,\pm\xi^j)$. In particular, this limiting process is independent 
of the specific probability distribution of the coefficients $(v_j)$, or of the orthonormal family $(e_j)$ used to generate the random potential $V_{\omega}$; this process
only depends on cardinal $2J$ of the energy shell $p_0^{-1}(z)$ and of the local spectral densities $\{\sigma_+^j(z_0),\ j=1,\dots, J\}$. 
 \par
It is known that the zero process of a single GAF exhibits a local repulsion between the nearby points (see the section~\ref{s:2-point}). On the other hand, as a  superposition of $J$ \emph{independent} point processes, the limiting process $\mathcal{Z}_{G_{z_0}}$ authorizes the presence of clusters of at most $J$ points. In the next section we analyze this clustering by computing the correlations between the points of the process.
\subsubsection{Scaling limit of the $k$-point measures}\label{s:k-point}
An explicit way to obtain information about the statistical interaction of 
$k$ nearby eigenvalues of $P^{\delta}_V$ is by analysing 
the \textit{$k$-point measures} of the point process $\mathcal{Z}_{h,z_0}^V$, which quantify the
correlations within $k$-point subsets of the point process. 
These are positive measures $\mu_h^{k,V,z_0}$ on $O^k\backslash\Delta$, 
where $O$ is as in Theorem \ref{thm_m2} and 
$\Delta=\{z\in\C^k; \exists~ i\neq j ~\text{s.t. }z_i=z_j\}$ is a generalized diagonal set. These measures are defined through their action on an arbitrary test funtion $\phi\in\mathcal{C}_c(O^k\backslash\Delta,\R_+)$:
\begin{equation}\label{eq2.11}
	\begin{split}
	\erw \left[ (\mathcal{Z}_{h,z_0}^V)^{\otimes k} (\phi) \right]
		&= 
		\erw \left[ \sum_{z_1,\dots,z_k\in\spec (P^{\delta})} 
			\phi \left((z_1-z_0)h^{-1/2},\dots,(z_k-z_0)h^{-1/2}\right)  \right] \\ 
		& \defeq \int_{O^k\backslash\Delta} \phi(w)\,\mu_{h,z_0}^{k,V}(dw)\,.
			\end{split}
 \end{equation}
 We have stamped out the diagonal $\Delta$ in order to avoid trivial self-correlations. 
 When these $k$-point measures are absolutely continuous with respect to the Lebesgue 
 measure on $\C$, we call their densities the \textit{$k$-point functions}.
\begin{thm}\label{thm_m3}
	Let $\mu_{h,z_0}^{k,V}$ be the $k$-point measure of $\mathcal{Z}_{h,z_0}^V$, 
	defined in \eqref{eq2.11}, and let $\mu^{k,V}_{z_0}$ be the $k$-point measure of 
	the point process $\mathcal{Z}_{G_{z_0}}$, given in 
	Theorem \ref{thm_m2}. Then, for any connected domain $O\Subset\C$ and for all $\phi\in\mathcal{C}_c(O^k\backslash\Delta,\R_+)$,  
	\begin{equation*}
		\int_{O^k\backslash\Delta} \phi(w)\, \mu_{h,z_0}^{k,V}(dw) \longrightarrow \int_{O^k\backslash\Delta} \phi(w)\, \mu^{k,V}_{z_0}(dw), 
		\quad h\to 0.
	\end{equation*}
	Moreover, $\mu^{k,V}_{z_0}$ is absolutely continuous with respect to the Lebesgue 
	measure on $\C$. Its density $d^{k,V}_{z_0}$ is given by the following formula:
\begin{equation}\label{eq2.11.4}
	 d^{k,V}_{z_0}(w_1,\ldots,w_k) = \sum_{\substack{\alpha\in\N^J, \\ \sum_j\alpha_j=k }}
	 	  \sum_{\tau\in \mathfrak{S}_k} \frac{1}{\alpha !}
		  \prod_{j=1}^J
	 	   d^{\alpha_j}_{g_j}(w_{\tau(\alpha_1+\cdots+\alpha_{j-1}+1)},\dots,w_{\tau(\alpha_1+\cdots+\alpha_{j})})
\end{equation}
where $\mathfrak{S}_k$ is the symmetric group of $k$ elements, and for all $1\leq j \leq J$ and all $r\in\N^*$, 
\begin{equation}\label{eq2.11.5}
 	d_{g^j}^r(w) = \frac{\mathrm{perm}[C_j^r(w)- B_j^r(w)(A_j^r)^{-1}(w)(B_j^r)^*(w) ] }{\det \pi A_j^r(w)},\ \ \text{while }d_{g^j}^0(w)\equiv 1\,. 
 \end{equation}
 Here, $\mathrm{perm} $ denotes the permanent of a matrix; $A_j^r, B_j^r,C_j^r $ 
 are complex $r\times r$-matrices given by 
  \begin{equation*}
  	(A_j^r)_{n,m} = K^j(w_n,\bar{w}_m), \quad 
	(B_j^r)_{n,m} = (\partial_wK^j)(w_n,\bar{w}_m), \quad 
	(C_j^r)_{n,m}= (\partial^2_{w\bar{w}}K^j)(w_n,\bar{w}_m),
  \end{equation*}
where $K^j(w,\bar{w})=\exp(\sigma_+^j(z_0)w\bar{w})$ is the covariance function of the GAFs $g^j_{z_0}$ appearing in Thm~\ref{thm_m2}.
\end{thm}
The function $d_{g^j}^r(z)$ in \eqref{eq2.11.5} is the $r$-point function 
for the zero process of the Gaussian analytic function $g^j$. Thm~\ref{thm_m3} 
tells us that the limiting $k$-point measures admit densities with respect to 
the Lebesgue measure, and that those can be determined by concatenating the $r$-point 
functions, for $1\leq r \leq k$, of each GAF $g^j$. 
\par
A result by Nazarov and Sodin \cite[Theorem 1.1]{NaSo12} implies the following estimate for the $r$-point densities of a single GAF.
\begin{prop}\label{thm_m3a}\cite{NaSo12}
Let $O \Subset \C$ be a compact set. Let $(g^j=g^j_{z_0})_{1\leq j\leq J}$ be the GAFs
appearing in Thm~\ref{thm_m2}, and let 
$d^r_{g^j}(w)$, $1\leq r\leq k$, be as in \eqref{eq2.11.5}. Then there exists a constant $C = C(r, g^j, O)>1$ such that, for any configuration of pairwise distinct points
$w_1,\ldots, w_k\in O$, 
\begin{equation*}
	C^{-1} \prod_{i<j}|w_i-w_j|^2\leq 
	d_{g^j}^r(w_1,\dots,w_k) 
	\leq 
	C\prod_{i<j}|w_i-w_j|^2.
\end{equation*}
\end{prop}
In formula \eqref{eq2.11.4} we see that if $k>J$, then each summand has at least one 
factor $d^{\alpha_j}_{g^j}$ has $\alpha_j\geq 2$. Hence, we immediately 
conclude from Thm~\ref{thm_m3} and Prop.~\ref{thm_m3a} the following 
\begin{cor}\label{cor:NoClu}
	Let $O\Subset \C$ be a compact set, let $k>J$ and let $d^{k,V}_{z_0}(w)$ be as in 
	\eqref{eq2.11.4}. Then there exists a positive constant $C = C(r, O)$ such that, 
	for any configuration of pairwise distinct points $w_1, \ldots , w_k\in O$, 
\begin{equation*}
	d^{k,V}_{z_0}(w_1,\ldots,w_k) 
	\leq 
	C\sum_{i<j}|w_i-w_j|^2.
\end{equation*}
\end{cor}
We have seen by Theorem \ref{thm_m3} that the limiting point process of the rescaled 
eigenvalues is given by the superposition of $J$ independent processes given by 
the zeros of independent Gaussian analytic functions. Due to this independence, $k$ points, each originating from a different GAF process, may approach each other without any statistical repulsion: this authorizes the formation 
of clusters of a most $J$ points. As a result, for $k\leq J$ the limiting $k$-point functions do not decay 
to zero as the distances between the $k$ points gets smaller. This behaviour is made explicit in the next section in the case $k=2$. 

On the opposite, if $k>J$ then at least two points must originate from the same GAF process, and therefore repel each other quadratically when approaching. This is exactly what Corollary~\ref{cor:NoClu} 
tells us: the probability to find more than $J$ points 
close together decays quadratically with the distance. Therefore, finding  
clusters of more than $J$ eigenvalues very close together is very unlikely. 
\subsubsection{$2$-point correlation function}\label{s:2-point} The $2$-point \emph{correlation function}
of a point process is defined by the $2$-point function, renormalized by the local $1$-point functions (or local average densities): 
\begin{equation*}
	K^{2,Q}_{z_0}(w_1,w_2) = \frac{ d^{2,Q}_{z_0}(w_1,w_2)}{d^{1,Q}_{z_0}(w_1)d^{1,Q}_{z_0}(w_2)}, \quad w_1\neq w_2 \in O,\quad Q=V_\omega,\ M_\omega\,.
\end{equation*}
By Theorem \ref{thm_m3}, the limiting local $1$-point function $d^{1,V}_{z_0}(w)$ is a constant function, given by 
\begin{equation*}
	 d^{1,V}_{z_0}(w) = \sum_{j=1}^J \frac{\sigma_+^j(z_0)}{\pi},\quad\forall w\in O.
\end{equation*}
This average density of eigenvalues (at the \emph{microscopic} scale near $z_0$) exactly corresponds to the \emph{macroscopic} density predicted by the probabilistic Weyl's law in Theorem \ref{thm:PWL}, see also \eqref{eq2.1}. 
\par
The limiting $1$-point and $2$-point functions of the zero process generated by a single GAF $g_\sigma$ (see section~\ref{s:GAF}) are given by
\begin{equation*}
	d^1_{g_\sigma}(w_1) = \frac{\sigma}{\pi},\quad\text{respectively}\quad
	d^2_{g_\sigma}(w_1,w_2) = \Big(\frac{\sigma}{\pi}\Big)^2
					   \kappa \Big(\frac{ \sigma |w_1-w_2|^2}{2} \Big),
\end{equation*}
with the universal scaling function
\begin{equation}\label{eq_KapLim}
	\kappa (t)  \defeq \frac{(\sinh^2 t + t^2)\cosh t - 2t \sinh t}{\sinh^3 t}, \quad \forall t\geq 0.
\end{equation}
\begin{figure}[ht]
  \centering 
  \includegraphics[width=0.5\textwidth]{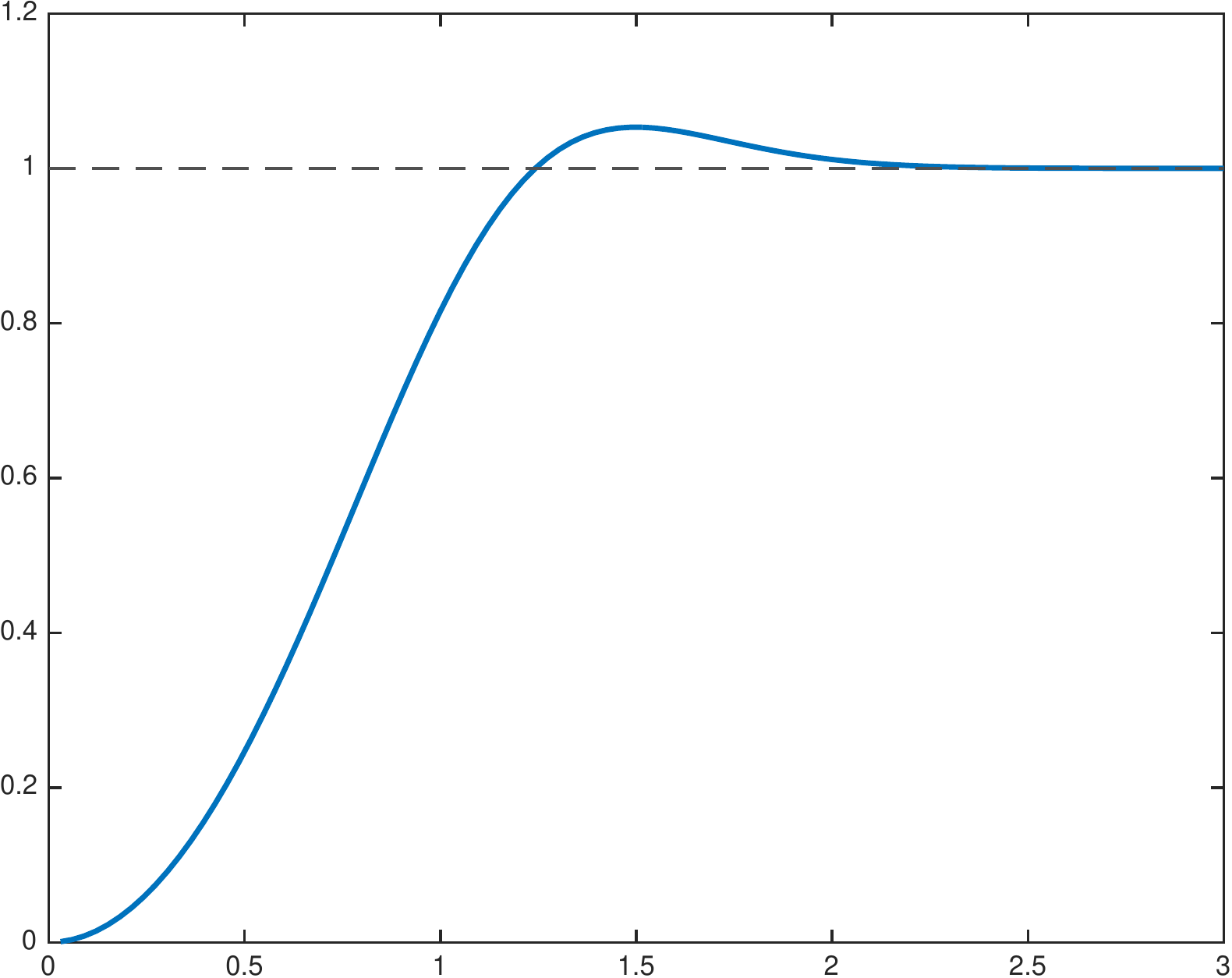}
 \caption{Plot of the scaling limit function $\kappa(t^2)$, see \eqref{eq_KapLim}.}
 \label{fig2}
\end{figure}
The function $\kappa \left(\sigma |w_1-w_2|^2/2 \right)$ 
describes the $2$-point correlation function of the zeros of the GAF $g_\sigma$. A remarkable property of this function is its isotropy: it only depends on the distance between the points $w_1,w_2$. In Figure~\ref{fig2} we plot the function $t\mapsto \kappa(t^2)$. This function behaves as $\kappa(t^2)=t^2(1+\mO(t^4))$ when $t\to 0$, which reflects the quadratic repulsion between the nearby zeros of $g_\sigma$. On the opposite, when $t\gg 1$ it converges exponentially fast to unity, showing a fast decorrelation between the zeros. 

To our knowledge, the function $\kappa$ first appeared as the scaling limit $2$-point correlation function for the zeros of certain ensembles of random polynomials describing random spin states, see J.H. Hannay~\cite{Ha96}. In the work by 
P. Bleher, B. Shiffman and S. Zelditch \cite{BlShiZe00}, $\kappa$ describes the scaling limit 
$2$-point correlation function for the zeros of random holomorphic 
sections of large powers of a positive Hermitian line bundle over a compact complex 
K\"ahler surface. 

In the present work, $\kappa$ appears as a building block for the scaling limit $2$-point 
correlation function of the eigenvalues of $P_V^{\delta}$:
\begin{equation}\label{e:K^2V}
	K^{2,V}_{z_0}(w_1,w_2) = 1
			     +\sum_{j=1}^J\frac{(\sigma_+^j(z_0))^2}
			     {\left(\sum_{j=1}^J\sigma_+^j(z_0)\right)^2}
			      \left[\kappa\left(\frac{ \sigma^j_+(z_0)|w_1-w_2|^2}{2} \right)
			      -1\right]
\end{equation}
Let us study more closely this limiting $2$-point correlation function between the rescaled eigenvalues of $P^\delta_V$ near $z_0$:\\ \\
\textbf{Long range decorrelation}: For $|w_1-w_2| \gg 1$, in the limit $h\to 0$, the $2$-point correlation function converges exponentially fast to unity	%
\begin{equation*}
	K^{2,V}_{z_0}(w_1,w_2) = 1+
	\mO\left(\e^{-\min\limits_j\sigma^j_+(z_0)|w_1-w_2|^2 }\right).
\end{equation*}
	This shows that two points at distances $|w_1-w_2|\gg 1$ are statistically uncorrelated.
\\
\\
 \textbf{A weak form of repulsion}:  When $|w_1-w_2| \ll 1 $, in the limit $h\to 0$, 
 	there is a weak form of repulsion between two nearby eigenvalues, 
\begin{equation}\label{eq.2ptCor}
	K^{2,V}_{z_0}(w_1,w_2) =1 - \sum_{j=1}^J\frac{\sigma_+^j(z_0)^2}
			     {\big(\sum_{l=1}^J\sigma_+^l(z_0)\big)^2}
			      \Big[1 - \frac{ \sigma^j_+(z_0)|w_1-w_2|^2}{2} + 
			      \mO( |w_1-w_2|^4 )
			      \Big].
\end{equation}
This formula shows that the probability to find two rescaled eigenvalues $w_1,w_2$ at distance $\ll 1$ is 
smaller than the one to find them at large distances: pairs of rescaled eigenvalues 
show a weak repulsion at short distance. However, the correlation function does not converge zero when $|w_1-w_2|\to 0$, but 
to the positive value  $1 - \frac{\sum_{j=1}^J \sigma_+^j(z_0)^2}{(\sum_{l=1}^J\sigma_+^l(z_0))^2}$. This weak repulsion can be explained by the fact that the random function $G_{z_0}$  is the product of $J$ independent GAFs: two zeros $w_1,w_2$ will not repel each other if they originate from different GAFs, while they will repel quadratically if the come from the same GAF. The net result is this weaker form of repulsion.
The larger the number of quasimodes $J$, the weaker this repulsion becomes, since two zeros $w_1,w_2$ chosen at random will have a smaller chance to come from the same GAF.
\par
In Figure~\ref{fig3} we compare the theoretical $2$-point correlation functions $K^{2,V}_{z_0}$ with the one obtained from numerical computations for two model operators on the torus $\mathbb{T}=\R/(2\pi\Z)$:
\begin{equation}\label{eq:NumOp}
	P_{h,q}^{\delta} = -h^2\partial_x^2 + \e^{-iqx} + \delta V_{\omega}, \quad q =1,3, \quad x\in \mathbb{T}. 
\end{equation}
We took the parameters $h=10^{-3}$, $\delta=10^{-12}$, and the Gaussian random potential $V_{\omega}$ 
as in section~\ref{s:model-case}.  We use operators defined
on $\mathbb{T}$ because they are numerically easier to diagonalize than operators defined on $\R$. For each operator $P_{h,q}$, we computed 1000 realizations of the random potential to extract the correlation function.
 
The analysis of the principal symbols $p_{q,0}$ shows that the classical spectrum is, in both cases, given by $\Sigma=\R_+ + U(1)$. At the energy $z_0=1.6$ we selected, the operator $(P_{h,q}-z_0)$ admits $J=2q$ quasimodes. Figure \ref{fig3} compares the 
numerically obtained $2$-point correlation functions (shown as blue dots) of the operators $P_{h,1}^{\delta}$
(left) and $P_{h,3}^{\delta}$ (right), with the theoretical scaling limit 
$2$-point correlation described in \eqref{eq.2ptCor}. For the two operators, the theoretical curve fits quite well the 
numerical points, including in the short distance limit $|w_1-w_2|\ll 1$.
\begin{figure}[ht]
 \begin{minipage}[b]{0.45\linewidth}
  \centering
  \includegraphics[width=0.9\textwidth]{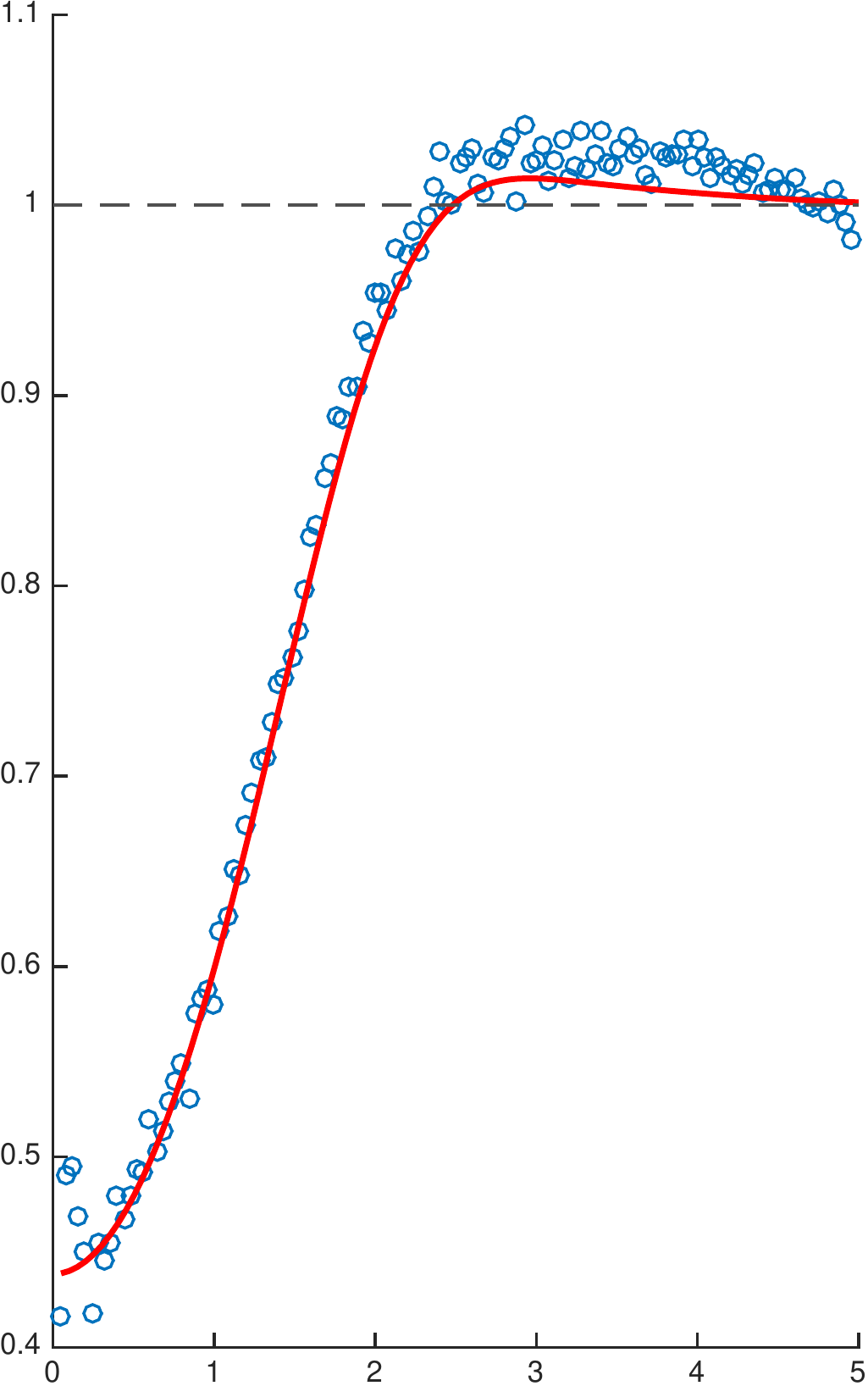}
 \end{minipage}
 \hspace{0.04\linewidth}
 \begin{minipage}[b]{0.45\linewidth}
  \centering 
  \includegraphics[width=0.88\textwidth]{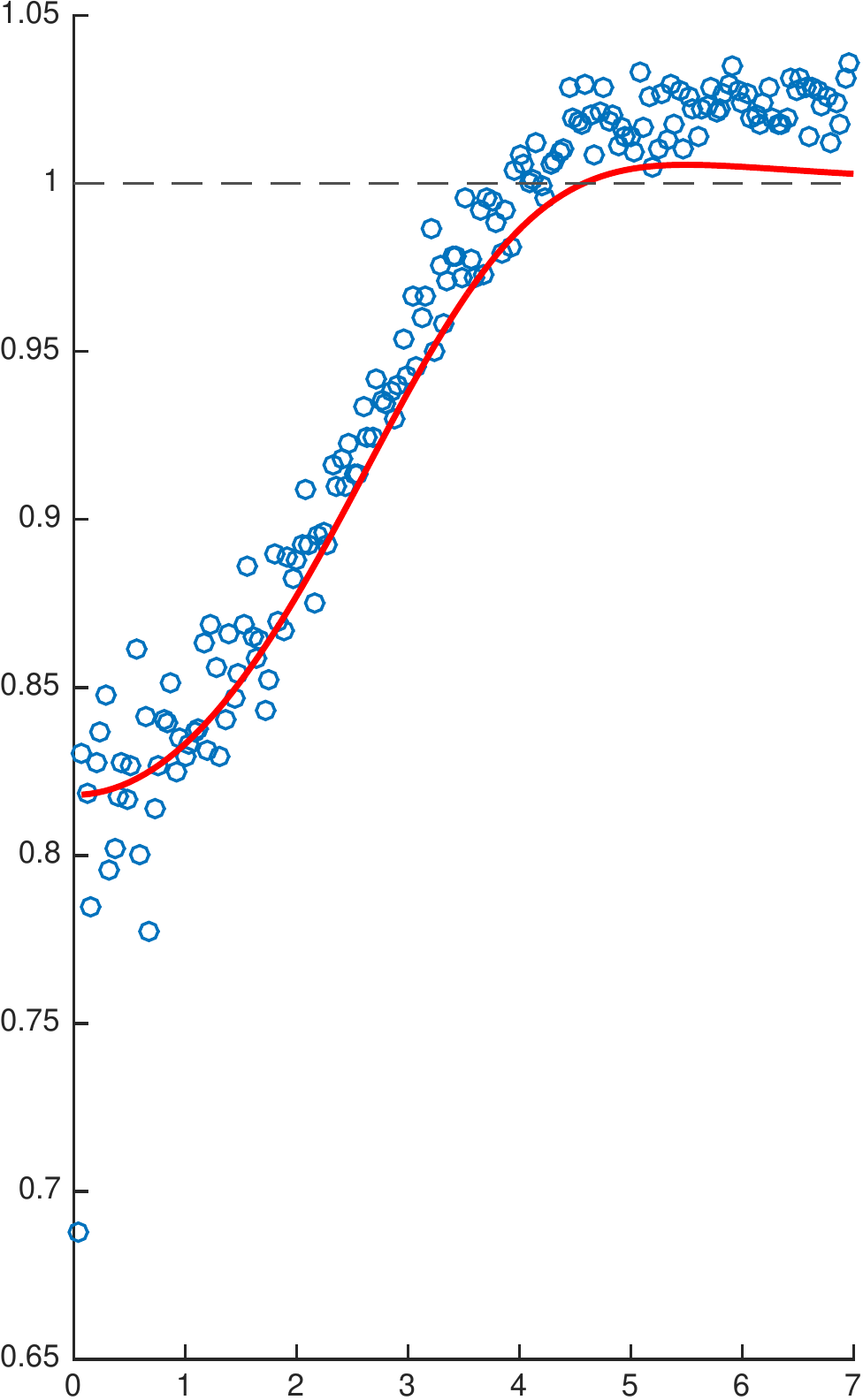}
 \end{minipage}
 \caption{Blue points: values of the $2$-point correlation functions at the energy $z_0=1.6$, obtained by numerically computing the spectra of the two operators $P_{h,1}^{\delta}$ (left) and $P_{h,3}$ (right) perturbed by a Gaussian random potential $\delta V_\omega$. Red curves: scaling limit $2$-point correlation functions $K^{2,V}_{z_0}$ for both operators, as given in \eqref{e:K^2V}; the horizontal coordinate is the rescaled square distance $|w_1-w_2|^2$.}
 \label{fig3}
\end{figure}
\subsection{Perturbation by random matrix}
We now describe the situation where the operator $P_h$ is perturbed by a small random matrix $\delta M_\omega$, as described in (\ref{eq1.10},\ref{eq1.11}). In this section we do not need to assume the symmetry property \eqref{eq.15} for the symbol $p_0$.
Here as well, we can prove a convergence of the rescaled spectral point process $\mathcal{Z}_{h,z_0}^M$ (see \eqref{eq:PPM}) towards a limiting zero process when $h\to 0$.

\subsubsection{Universal limiting point process}
\begin{thm}\label{thm_m1}
 Let $\Omega\Subset\mathring{\Sigma}$ be as in \eqref{eq1.8.2}. 
 Let $p$ be as in \eqref{eq1.6} satisfying \eqref{eq1.0}. 
 Choose $z_0\in\Omega$. Then, for any $O\Subset\C$ open 
 connected domain, the rescaled spectral point process  $\mathcal{Z}_{h,z_0}^M$ converges in distribution towards the zero point process
 associated with a random analytic function $\tG_{z_0}$ described below:
 \begin{equation}
  	\mathcal{Z}_{h,z_0}^M
	\stackrel{d}{\longrightarrow} 
	\mathcal{Z}_{\tG_{z_0}} \text{ on } O \quad 
	\text{ as } h\to 0.
 \end{equation}
The random function $\tG_{z_0}$ is defined as %
  \begin{equation*}
 \tG_{z_0}(w)\defeq\det (g_{z_0}^{i,j}(w))_{1\leq i,j \leq J}, \quad w\in \C, 
 \end{equation*}
 where $g_{z_0}^{i,j}$, for $1\leq i,j\leq J$, are $J^2$ independent GAFs $g_{z_0}^{i,j}\sim g_{\sigma^{i,j}_{z_0}}$, for the parameters %
 \begin{equation}\label{e:Gaf_ij}
\sigma^{i,j}_{z_0}= \frac{1}{2}(\sigma_+^i(z_0)+\sigma_-^j(z_0))\,.
\end{equation}
(we recall that the local classical densities $\sigma_{\pm}^i(z_0)$ associated with the points $\rho_{\pm}^i(z_0)\in p_0^{-1}(z_0)$ are defined in \eqref{eq2.1}).
\end{thm}
Theorem \ref{thm_m1} tells us that at any given point 
$z_0\in\widetilde{\Omega}\cap\mathring{\Sigma}$ in the bulk of the 
pseudospectrum, the local rescaled point process of the 
eigenvalues of $P_M^{\delta}$ is given, in the limit $h\to 0$, by the 
zero process associated with the determinant of a $J\times J$ matrix whose 
entries are independent GAFs. 
The GAF situated at the position $i,j$ of the matrix only depends
on the portion of the classical spectral densities due to the points 
 $\rho_{+}^i(z_0)$ and  $\rho_{-}^j(z_0)$.  %
\begin{figure}[ht]
 \begin{minipage}[b]{0.45\linewidth}
  \centering
  \includegraphics[width=\textwidth]{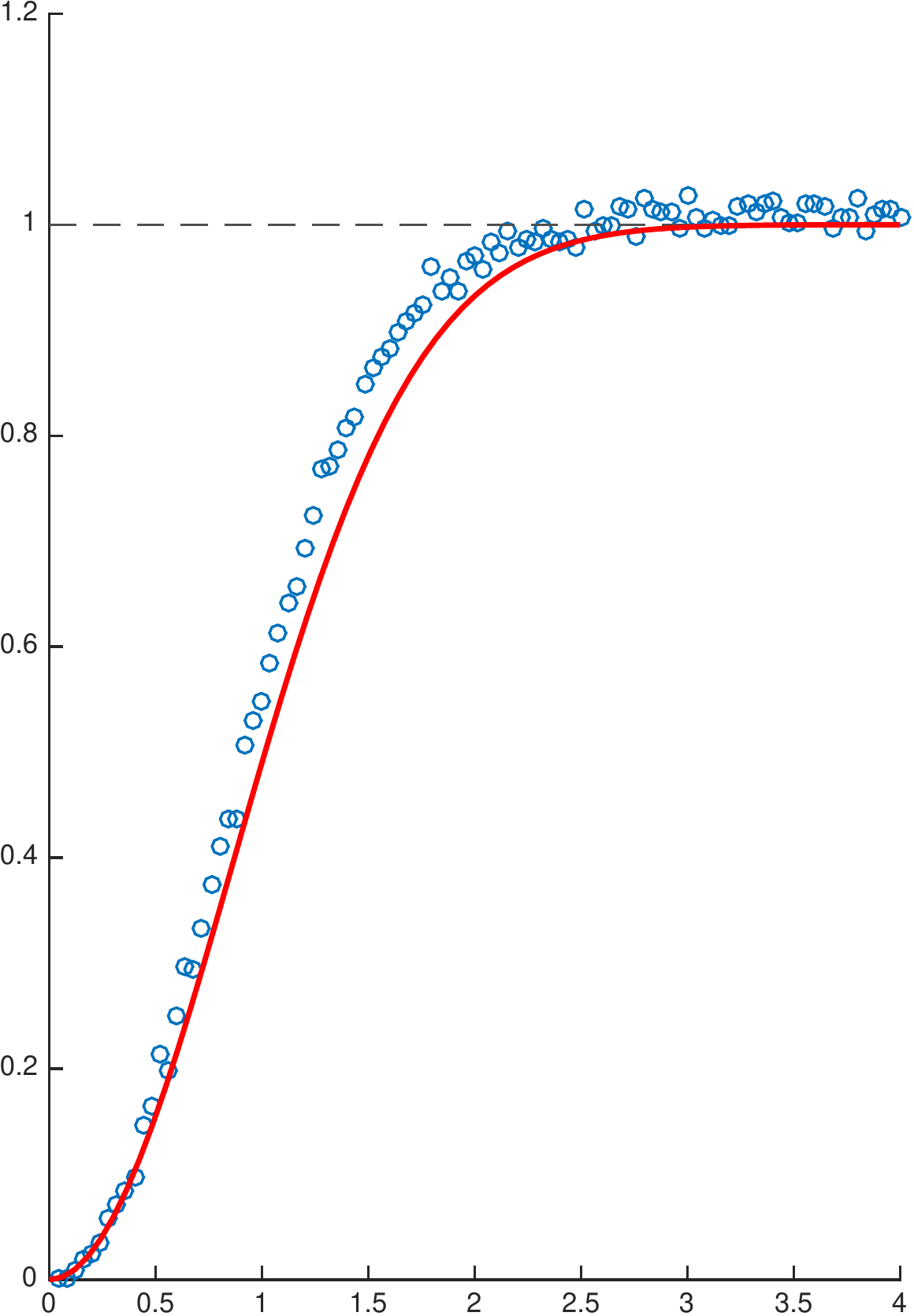}
 \end{minipage}
 \hspace{0.04\linewidth}
 \begin{minipage}[b]{0.45\linewidth}
  \centering 
  \includegraphics[width=\textwidth]{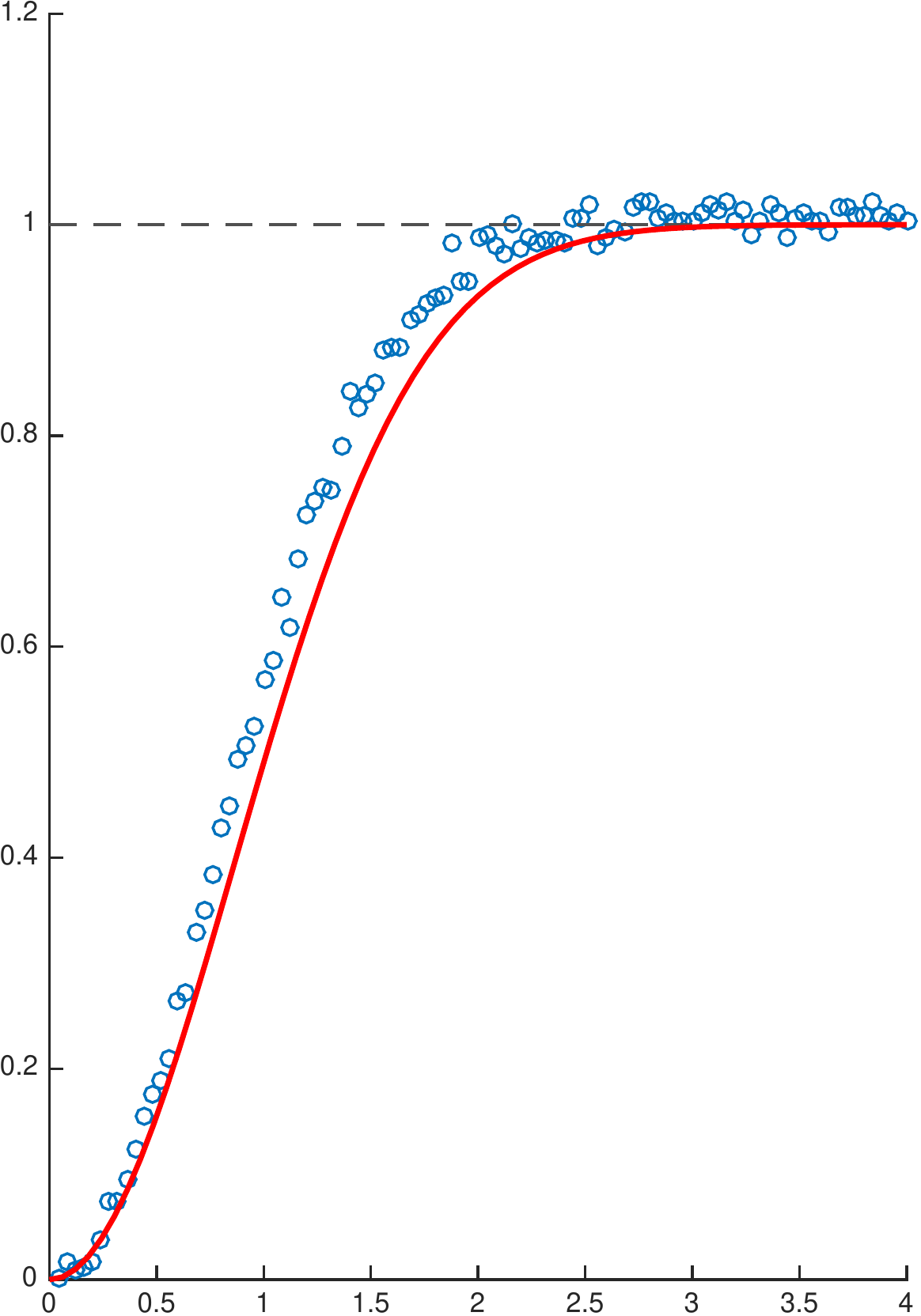}
 \end{minipage}
 \caption{Blue points: values of the $2$-point correlation functions, obtained by numerically computing the spectra of the two operators $P_{h,1}$ (left) and $P_{h,3}$ (right) perturbed by a Gaussian random matrix $\delta M_\omega$. The parameters $z_0,h,\delta$ are as in Figure~\ref{fig3}. Red curves: the 2-point correlation function for the Ginibre ensemble, $K^{2,Gin}_{z_0}$, as given in \eqref{eq:Cor}.  The horizontal coordinate is the rescaled square distance $|w_1-w_2|^2$.
 }\label{fig5}
\end{figure}
 \par
The rescaled point process $\mathcal{Z}_{h,z_0}^M$ of the eigenvalues 
of the perturbed operator $P^{\delta}_M$ has a universal limit, which is independent 
of the specific probability distribution of the potential $M_{\omega}$ \eqref{eq1.9}, but only
depends on the cardinal $2J$ of the energy shell $p_0^{-1}(z)$ the local classical densities $\{\sigma^j_{\pm}(z_0); j=1,\dots,J\}$ (we notice that without the assumption \eqref{eq.15}, the densities $\sigma^j_{+}(z_0)$ and $\sigma^j_{-}(z_0)$ are a priori unrelated).

This limiting process is of a different type from the universal limit of $\mathcal{Z}_h^V$ studied in the previous section. In particular the function $\tG_{z_0}$ is not given by a simple product of GAFs, but by a more complicated expression, namely a determinant. As we will see below, we  conjecture that the zero process of $\tG_{z_0}$  exhibits a quadratic repulsion between the nearby points, as opposed to the zero process of the function $G_{z_0}$ in Thm~\ref{thm_m2}.
\subsubsection{Scaling limit $k$-point measures}
A direct consequence of the convergence of the zero processes $\mathcal{Z}^{M}_{h,z_0}$ is the convergence of their $k$-point measures converge to those of the limiting point process. 
\begin{cor}
	Let $\mu_{h,z_0}^{k,M}$ be the $k$-point measure of $\mathcal{Z}_{h,z_0}^M$, 
	defined as in \eqref{eq2.11}, and let $\mu^{k,M}_{z_0}$ be the $k$-point measure of 
	the point process $\mathcal{Z}_{\tG_{z_0}}$ described in Thm~\ref{thm_m1}. Then, 
	for any $O\Subset\C$ open connected domain and 
	for all $\phi\in\mathcal{C}_c(O^k\backslash\Delta,\R_+)$, 
	\begin{equation*}
		\int_{O^k\backslash\Delta} \phi(w)\, \mu_{h,z_0}^{k,M}(dw) \longrightarrow \int_{O^k\backslash\Delta} \phi(w)\, \mu_{z_0}^{k,M}(dw), 
		\quad h\to 0.
	\end{equation*}
\end{cor}
One can calculate the Lebesgue densities of the limiting $1$-point measures $ \mu_{z_0}^{1,M}$:
\begin{equation}\label{eq:MatDens}
	d^{1,M}_{z_0}(w) = \sum_{i=1}^J \frac{\sigma_+^i(z_0)+\sigma_-^i(z_0)}{2\pi},\quad \forall w\in O.
\end{equation}
This microscopic density exactly coincides with the macroscopic density of eigenvalues at $z_0$ predicted by the probabilistic Weyl's law in Thm~\ref{thm:PWL}, see also \eqref{eq2.1}. In the case of an operator $P_h$ satisfying the symmetry assumption \eqref{eq.15}, this local density is equal to the one obtained for the operator perturbed by a random potential: for such symmetric symbols, the microscopic densities $d^{1,Q}_{z_0}$ therefore cannot distinguish between the type of perturbation imposed on $P_h$. 
\par
On the opposite, we believe that the $k$-point densities $d^{k,Q}_{z_0}$ (equivalently, the $k$-point correlation functions $K^{k,Q}_{z_0}$) for $k>1$ should distinguish between the two types of perturbation (still assuming the symmetry \eqref{eq.15}). Unfortunately, we have not been able, so far, to compute in closed form the Lebesgue densities of the limiting $k$-point measures 
$ d^{k,M}_{z_0}$ associated with the random function $\tG_{z_0}$ (the determinant of a matrix of independent GAFs). However, the numerical experiments presented 
in Figure~\ref{fig5}, as well as the Proposition~\ref{thm_m3a}, lead us to the following 
\begin{con}\label{con:Cor}
	The $k$-point densities $d^{k,M}_{z_0}(w)$ of the zero point process of the random function $\tG_{z_0}$ 
	as in Theorem~\ref{thm_m1} exhibit a quadratic repulsion at short distance. Namely, for any 
	compact set $O\Subset \C$, there exists a constant $C>1$ depending only 
	on $O$ and $k$ such that, for all pairwise distinct points $w_1,\ldots,w_k\in O$, 
	\begin{equation*}
	C^{-1} \prod_{i<j}|w_i-w_j|^2\leq 
	d^{k,M}_{z_0}(w_1,\dots,w_k) 
	\leq 
	C\prod_{i<j}|w_i-w_j|^2.
	\end{equation*}
\end{con}
In Figure~\ref{fig5} we compare our experimental values of the $2$-point correlation function with the $2$-point correlation function of a well-known spectral point process on $\C$, namely the one for the Ginibre ensemble of random matrices \cite{Gin65}. This ensemble corresponds to the spectra of the random matrices $M_\omega$ alone, in the case where their coefficients are Gaussian (see eq.~\eqref{e:Ginibre}), in the limit $h\to 0$, or equivalently the limit or large matrices. It is known since the work of Ginibre \cite{Gin65} that the eigenvalues of these matrices repel each other quadratically at short distance. When the eigenvalues are rescaled such that the mean local density is given by $d^1(w)=\sigma/\pi$, the 2-point correlation function takes the simple form
\begin{equation}
	K^{2,Gin}_{\sigma}(w_1,w_2) 
	= 1 - \exp(-\sigma|w_1-w_2|^2)\,.
\end{equation}
Hence, for our local density \eqref{eq:MatDens}, the 2-point function we draw on Fig.~\ref{fig5} is
\begin{equation}\label{eq:Cor}
	K^{2,Gin}_{z_0}(w_1,w_2) 
	= 1 - \exp\left[-\frac{1}{2}\big(\sum_{i=1}^J \sigma_+^i(z_0)+\sigma_-^i(z_0)\big) |w_1-w_2|^2\right].
\end{equation}
This function is markedly different from the scaling function $\kappa(t^2)$ corresponding to the zero GAF process \eqref{eq_KapLim}. It seems rather close to the experimental data of Fig.~\ref{fig5}, eventhough we observe a deviation for values $|w_1-w_2|\sim 1$. Is this deviation due to the finite value of $h$ used in the experiment? Or does the deviation persist when $h\to 0$, that is in the correlation function $K^{2,M}_{z_0}$?
We conjecture that the limiting correlation functions $K^{2,M}_{z_0}$ are not strictly equal to the Ginibre function $K^{2,Gin}_{z_0}$, but that they are becoming closer and closer to them when the number of quasimodes $J$ gets larger and larger (a property which purely concerns the classical symbol $p_0$): in this regime $\tilde G_{z_0}$ is the determinant of a large matrix of independent GAFs. In any case, computing the $k$-point densities for the process $\mathcal{Z}_{G_{z_0}}$ seems to us to be an interesting open problem.
\subsection*{Translation invariance}
One easy property of the limiting point processes obtained in Theorems \ref{thm_m3} and \ref{thm_m1}, is that they
are homogeneous and isotropic. This property is inherited from the similar invariance of the zero process of a single GAF \cite{HoKrPeVi09}.
\begin{prop}\label{thm_Tran}
	The limiting point processes $\mathcal{Z}_{G_{z_0}}$, resp. $\mathcal{Z}_{\tG_{z_0}}$ obtained in Theorems \ref{thm_m3} and 
	\ref{thm_m1} are invariant in distribution under the action of the group of translations and rotations 
	on $\C$. More precisely, for arbitray $\alpha,\beta \in \C$ with $|\alpha|=1$ let us define the plane isometry $\tau(w) = \alpha w + \beta$, $w\in\C$.
	
	Then, the zero processes of $G_{z_0}$ and $\tG_{z_0}$ satisfy 
	\begin{equation*}
		\mathcal{Z}_{G_{z_0}} \stackrel{d}{=} \mathcal{Z}_{G_{z_0}\circ\tau}, \quad\text{resp.} \quad \mathcal{Z}_{\tG_{z_0}} \stackrel{d}{=} \mathcal{Z}_{\tG_{z_0}\circ\tau}\,.
	\end{equation*}
\end{prop}
\subsection{Outline of the proofs}
The proof of the main results is built on two distinct parts. In the \textit{first part} we will 
reduce the eigenvalue problem of the perturbed operator $P^{\delta}_h$ to studying 
the zeros of a certain random analytic function. To do this, we will start by constructing quasimodes for 
the unperturbed operator $P_h$ in Section \ref{sec:Quas}, and study their interactions 
in Section \ref{sec:QuasInter}. Next we will use these quasimodes in Section \ref{sec:GP} 
to construct a well-posed Grushin problem for the perturbed operator $P^{\delta}$, which will yield 
an effective description of the eigenvalues of $P^{\delta}$ as the zeros of a random 
analytic function. \par
The \textit{second part} has a more probabilistic flavour. We first provide 
in Section \ref{sec:RAF} an overview of the notions and results from the theory of 
random analytic functions used in this paper. The sections \ref{sec:LS_M} and  \ref{sec:LS_V} 
then use these results to finish the proofs of our main theorems. 
%
%
\section{Quasimodes}\label{sec:Quas}
The aim of this section is provide the ingredients for a well-posed Grushin 
problem of the unperturbed operator $P_h$, cf. \eqref{eq1.7.5}. Our main objective in this section is the construction of the 
quasimodes $e^j_{\pm}(z_0)$ for the operator $(P_h-z_0)$ (resp. $(P_h-z_0)^*$). For this aim we will first factorize our symbol $p$ into a "nice" form, using semiclassical analysis: this factorized form will allow very explicit expressions of our quasimodes.
\subsection{Malgrange preparation theorem}
We begin by giving an asymptotic factorization of the symbol $p$, \eqref{eq1.6} in a neighbourhood of each of the points 
$\rho_{\pm}^j=(x^j_{\pm},\xi_{\pm}^j)\in p_0^{-1}(z_0)$, see \eqref{eq1.8.1}. The method
presented here is an adaptation of \cite{Ha06b}.
\begin{prop}\label{prop4.1}
Let $\Omega\Subset\C$ be as in \eqref{eq1.8.2}, let $z_0\in\Omega$ and 
let $p(h)\sim \sum_{k\geq 0} h^k\, p_k$ in $S(\R^2,m)$, satisfying 
\eqref{eq1.8.1}. Let $j=1,\dots,J$ and let $U_{\pm}^j$ be open neighbourhoods of 
$\rho_{\pm}^j(z_0)$. 
Then, there exists an open, relatively compact neighbourhood of $z_0$, denoted 
by $W(z_0)$, open sets $V_{\pm}^j\subset U_{\pm}$ containing 
$\rho_{\pm}^j(\overline{W(z_0)})$, and symbols in $S(V_{\pm}^j,1)$:
\begin{equation}\label{eq4.1}
q^{\pm,j}(x,\xi,z;h)\sim\sum_{k\geq 0} h^kq_k^{\pm,j}(x,\xi,z), \quad 
g^{\pm,j}(x,z;h)\sim\sum_{k\geq 0} h^kg_k^{\pm,j}(x,z)\,,
\end{equation}
depending smoothly on $z\in W(z_0)$, such that for all $z\in W(z_0)$,
\begin{equation}\label{eq4.2}
\begin{split}
p(x,\xi;h)-z \sim q^{+,j}(x,\xi,z;h)\# (\xi+g^{+,j}(x,z;h)) \quad \text{ in } S(V_{+}^j,m), \\
p(x,\xi;h)-z \sim (\xi+g^{-,j}(x,z;h))\#q^{-,j}(x,\xi,z;h) \quad \text{ in } S(V_{-}^j,m).
\end{split}
\end{equation}
Furthermore,  the principal symbols $q_0^{\pm,j}(x_{\pm}^j(z),\xi_{\pm}^j(z),z)\neq 0$ and $g_0^{\pm,j}(x_{\pm}^j(z),z)=-\xi_{\pm}^j(z)$. 
\end{prop}
We recall that $\#$ indicates the Moyal product, which translates the operator composition to the symbolic level \cite[Chapter 7]{DiSj99}. For symbols $a_j\in S(\R^2,\widetilde{m}_j)$, for $j=1,2$, then  
\begin{equation*}
 	a_1^w\circ a_2^w = (a_1\# a_2)^w.
 \end{equation*}
The Moyal product $\#: S(\R^2,\widetilde{m}_1)\times S(\R^2,\widetilde{m}_2) 
\to S(\R^2,\widetilde{m}_1\widetilde{m}_2)$ is a bilinear and continuous map. 
\begin{proof}
We will focus on a single point $\rho^j_+(z)$, and will omit the $\pm$ and $j$ sub/superscripts in the proof. The case of the points $\rho^j_-$ can be treated identically.
\par
The condition \eqref{eq1.8.1} implies that for any $z\in\Omega$ we have that $p(\rho(z))-z=0$ 
and that $\partial_{\xi}p(\rho(z))\neq 0$. Let $z_0\in\Omega$. By the Malgrange preparation 
theorem \cite[Theorem 7.5.6]{Ho83}, there exist open neighbourhoods $V\subset U$ of $\rho(z_0)$ and 
$W(z_0)\subset \C$ of $z_0$, as well as smooth functions $q_0\in\mathcal{C}^{\infty}(V\times W(z_0))$ 
and $g_0\in\mathcal{C}^{\infty}(\pi_x(V)\times W(z_0))$, such that 
\begin{equation}\label{eq4.0}
	p(x,\xi)-z = q_0(x,\xi,z)(\xi+g_0(x,z)), \quad \text{for all } (x,\xi,z)\in V\times W(z_0),
\end{equation}
and $q_0(x(z_0),\xi(z_0),z_0)\neq 0$, while $g_0(x(z_0),z_0)=-\xi(z_0)$. We can suppose that 
$q_0(x,\xi,z)\neq 0$ in $V\times W(z_0)$ by potentially shrinking $V$ and $W(z_0)$. 
Up to shrinking $W(z_0)$ we may also assume that $\rho(W(z_0))\subset V$, so that
$g_0(x(z),z)=-\xi(z)$ for all $z\in W(z_0)$.
\par
Next, we make the formal Ansatz 
\begin{equation*}
q(x,\xi,z;h)=\sum_{k\geq 0} h^kq_k(x,\xi,z), \quad 
g(x,z;h)=\sum_{k\geq 0} h^kg_k(x,z),
\end{equation*}
and group together the terms of the Moyal product
\begin{equation*}
\left. q(x,\xi,z;h)\# (\xi+g(x,z;h))	= \e^{\frac{ih}{2}(D_{\xi}D_y-D_xD_{\eta})}
	q(x,\xi,z;h)(\eta+g(y,z;h))\right|_{\substack{y=x \\ \eta=\xi}}
\end{equation*}
with the same power of $h$. For any $N\geq 1$, equating the coefficient of $h^N$ of this asymptotic development 
with the symbol $p_N$, we obtain 
\begin{equation}\label{eq4.3}
	G_N(x,\xi,z) = \frac{q_N(x,\xi,z)}{q_0(x,\xi,z)}\big(\xi+g_0(x,z)\big)+g_N(x,z),
\end{equation}
where
\begin{equation}\label{eq4.3b}
	\begin{split}
	&G_N(x,\xi,z) = \frac{1}{q_0(x,\xi,z)}
	\left(p_N(x,\xi)- \sum_{l=1}^{N-1} q_{N-l}(x,\xi,z)g_l(x,z) \right.\\
	& \left.\left.+ \sum_{k=0}^{N-1} \Big( \frac{i}{2}(D_{\xi}D_y - D_xD_{\eta})\Big)^{N-k}
		\Big(q_k(x,\xi,z)\eta + \sum_{r=0}^kq_{k-r}(x,\xi,z)g_r(y,z) \Big)\right)
		\right|_{\substack{y=x \\ \eta=\xi}}
	\end{split}
\end{equation}
Notice that $G_N$ only depends on $q_k$, $g_k$,  for $k< N$, and on $p_N$. 
We can determine the functions $q_k$ and $g_k$ inductively: since 
\begin{equation*}
	\xi(z)+g_0(x(z),z)=0, \quad \partial_\xi(\xi+g_0(x,z))=1,
\end{equation*}
the Malgrange preparation theorem  implies the existence 
of smooth functions $q_N/q_0$ and $g_N$ in $V\times W(z_0)$ satisfying \eqref{eq4.3}. 
Iterating the procedure, we obtain functions $q_k \in\mathcal{C}^{\infty}(V\times W(z_0))$, 
$g_k \in\mathcal{C}^{\infty}(\pi_x(V)\times W(z_0))$, $k\geq 1$, which allow us to construct full symbols $q(x,\xi,h)$ and $g(x,h)$ by Borel summation, which satisfy \eqref{eq4.1} in $S(V,1)$.
\end{proof}
\subsection{Almost holomorphic extensions}
In general the symbol $p(x,\xi,h)$ is not real analytic in the variables $x,\xi$, so that the functions $q(\rho,z)$ and $g(x,z)$ constructed in Prop.~\ref{prop4.1} are, a priori, not holomorphic in $z$. Yet, we show below that they are almost holomorphic.

We begin by recalling the notion of an almost holomorphic extension of a smooth function. 
It has been introduced by H\"ormander \cite{Ho69} and Nirenberg \cite{Ni71} in different contexts.
\par
\begin{defn}\label{def:AH}
Let $X\subset \C^n$ be an open set and let $\Gamma\subset X$ be closed. If 
$f\in\mathcal{C}^{\infty}(X)$, we say that $f$ is almost holomorphic at $\Gamma$ 
if $\partial_{\overline{z}}f$ vanishes to infinite order there, i.e. for any $N\in\N$ there 
exists a constant $C_N>0$ such that for all $z$ in a small neighbourhood of $\Gamma$ 
in $X$
\begin{equation*}
	|\partial_{\overline{z}}f(z)| \leq C_N~\dist(z,\Gamma)^N.
\end{equation*}
In this case we write $\partial_{\overline{z}}f(z) = \mO(\dist(z,\Gamma)^{\infty})$. 
If $\Gamma = X\cap \R^n$ then we simply say that $f$ is almost holomorphic. \par
If $f,g\in\mathcal{C}^{\infty}(X)$ are almost holomorphic at $\Gamma$ and if 
$f-g$ vanishes to infinite order there, then we say that that $f$ and $g$ are 
equivalent at $\Gamma$. If $\Gamma = X\cap \R^n$, then we simply say that 
$f$ and $g$ are equivalent and we write $f\sim g$. 
\end{defn}
Any $f\in\mathcal{C}^{\infty}(X\cap\R^n)$ admits an almost holomorphic extension, uniquely 
determined up to equivalence, see e.g \cite{Ho69,MelSj74}. Before we continue we recall parts of 
a technical lemma from \cite[Lemma 1.5]{MelSj74}. 
\begin{lem}\label{lem:T}
	Let $\Omega\subset\R^d$ be an open set and suppose that $u\in\mathcal{C}^{\infty}(\Omega)$. Let 
	$v(x)$ be a Lipschitz continuous function on $\Omega$ so that $|v(x)-v(y)| \leq C_{\Omega'} |x-y|$ 
	when $x,y\in\Omega'$ and $\Omega'\Subset \Omega$. Suppose that for all $N\in\N$, we have 
	\begin{equation*}
		|u(x)| \leq C_{N,\Omega'} |v(x)|^N, \quad x\in\Omega'.
	\end{equation*}
	Then for all $\Omega'\Subset\Omega$, $N\in\N$ and any multi-index $\alpha$, there is a constant 
	$C_{N,\Omega',\alpha}$ such that 
	\begin{equation*}
		 |v(x)|^{|\alpha|}\,|D_x^{\alpha} u(x)| \leq C_{N,\Omega',\alpha} |v(x)|^N, \quad x\in\Omega'.
	\end{equation*}
\end{lem}
Applying this Lemma to our construction, we will obtain almost holomorphic extensions of the functions $g_0(x,z)$ appearing in our construction.
\begin{lem}\label{lem:AH}
Under the hypotheses of Proposition \ref{prop4.1}, let $\widetilde{g}_0^{\pm,j}$ 
be an almost $x$-holomorphic extension of $g_0^{\pm,j}$, for $j=1,\dots, J$. Then, 
there exists an open relatively compact neighbourhood $W(z_0)$ of $z_0$,  
open relatively compact sets $X_{\pm}^j \subset \R$ and small complex neighbourhoods 
$\widetilde{X}_{\pm}^j$ of $X_{\pm}^j$ , such that 
$x_{\pm}^j(\overline{W(z_0)})\subset X_{\pm}^j$, and such for any $N\in \N$ and any $\alpha,\beta,\gamma\in\N$,
there exists a constant $C_N^{\alpha,\beta,\gamma} >0$ such that for all $z\in W(z_0)$ and all $x\in\widetilde{X}_{\pm}^j$
\begin{equation}\label{eq4.1a}
| \partial^{\alpha}_x\partial^{\beta}_z \partial_{\bar{z}}^{\gamma+1} \widetilde{g}^{\pm,j}_0(x,z)|
 \leq C_N^{\alpha,\beta,\gamma} | x-x_{\pm}^j(z)|^{N},\qquad \forall z\in W(z_0),\quad \forall x\in\widetilde{X}_{\pm}^j.
\end{equation}
Moreover, for any symbol $g_k^{\pm,j}$, $k\in\N$ (see Proposition~\ref{prop4.1}), 
for any $\alpha\in \N$ and any $N\in\N$, there exists a constant $C^{\pm}_{N,\alpha,k,j}>0$ such that 
\begin{equation}\label{eq4.1aa}
	|\partial^{\alpha}_x\partial_{\bar{z}}\, g^{\pm,j}_k(x,z)| \leq C^{\pm}_{N,\alpha,k,j} |x-x_{\pm}^j(z)|^{N}, \quad x\in X_{\pm}^j,\ z\in W (z_0).
\end{equation}
\end{lem}
Notice in particular that for $x\in X^j_{\pm}$ \eqref{eq4.1a} reads 
\begin{equation}\label{eq4.1b}
| \partial_{\bar{z}}g^{\pm,j}_0(x,z)|
 \leq C_N | x-x_{\pm}^j(z)|^{N}.
\end{equation}
\begin{proof}[Proof of Lemma \ref{lem:AH}] Again, we focus only on the 
$+$ case and omit the superscripts $j$ and $+$. 
\par
Let $q_0(x,\xi,z)\in \mathcal{C}^{\infty}(V\times W(z_0))$ and 
$g_0(x,z)\in \mathcal{C}^{\infty}(\pi_x(V)\times W(z_0))$ be as in Prop.~\ref{prop4.1}, with $V$ an open relatively compact set containing 
$\rho_+(\overline{W(z_0)})$. As in the proof of that Proposition, we may suppose 
that $q_0(x,\xi,z)\neq 0$ in $V\times W(z_0)$. 
\par
For $\widetilde{V}$ a small relatively compact complex 
neighbourhood of $V$,  let $\widetilde{q}_0(x,\xi,z)$ denote an almost $(x,\xi)$-holomorphic extension of 
$q_0(x,\xi,z)$ defined on $\widetilde{V}\times W(z_0)$, such that 
\begin{equation}\label{eq4.1bb}
\widetilde{q}_0(x,\xi,z)\neq 0, \quad (x,\xi,z)\in \widetilde{V}\times W(z_0).
\end{equation}
Similarly, let $\widetilde{g}_0(x,z)$, for $x\in \pi_x(\widetilde{V})$, denote an 
almost $x$-holomorphic extension of $g_0(x,z)$.
Using these two functions, we may extend \eqref{eq4.0} into  
an almost $(x,\xi)$-holomorphic extension of $p_0(x,\xi)-z$, by defining
\begin{equation}\label{eq4.1c}
	\widetilde{p}_0(x,\xi)-z \defeq \widetilde{q}_0(x,\xi,z)(\xi+\widetilde{g}_0(x,z)), \quad 
	(x,\xi,z)\in \widetilde{V}\times W(z_0).
\end{equation}
By \eqref{eq1.8.1}, we have that $\partial_{\xi} p_0(\rho_+(z))\neq 0$ for all $z\in W(z_0)$. By potentially shrinking $V$, $W(z_0)$ and $\widetilde{V}$ we can arrange that 
$\rho_+(\overline{W(z_0)})\subset V$ and that 
$\partial_{\xi} \widetilde{p}_0(x,\xi)\neq 0$ for all $(x,\xi)\in\widetilde{V}$.
\par
Recall from Proposition \ref{prop4.1} that $g_0(x_+(z),z)=-\xi_+(z)$ for all $z\in W(z_0)$. Hence, 
by potentially shrinking $V$ and $W(z_0)$ and by restricting $\widetilde{g}_0(\cdot,z)$ 
to an open relatively compact convex complex neighbourhood 
$\widetilde{X}$ of $X\defeq\pi_x(V)\subset \R$, 
with $\widetilde{X}\Subset\pi_x(\widetilde{V})$, we can arrange that $x_+(\overline{W(z_0)})\subset X$ and $(x,-\widetilde{g}(x,z))\in\widetilde{V}$, for all $x\in \widetilde{X}$.
\par
Taking $\xi=-\widetilde{g}_0(x,z)$ in \eqref{eq4.1c} and then taking the $\partial_{\bar{z}}$ derivative of that equation, 
we get that for all $x\in\widetilde{X}$ and all $z\in W(z_0)$,
\begin{equation*}
	\partial_{\xi}\widetilde{p}_0(x,-\widetilde{g}_0(x,z))\partial_{\bar{z}}\widetilde{g}_0(x,z) 
	+\partial_{\bar{\xi}}\widetilde{p}_0(x,-\widetilde{g}_0(x,z))\overline{\partial_{z}\widetilde{g}_0(x,z) } =0.
\end{equation*}
Since $\widetilde{p}_0$ is almost $(x,\xi)$-holomorphic, we have that for any $N\in\N$%
\begin{equation*}
|\partial_{\bar{z}}\widetilde{g}_0(x,z) | 
\leq C_N \big(\,|\Ima \widetilde{g}_0(x,z)|^{N} + |\Ima x|^N\big),\quad x\in\widetilde{X},\ z\in W (z_0)\,.
\end{equation*}
Since $\Ima \widetilde{g}_0(x_+(z),z)=0$, see Proposition \ref{prop4.1}, and since 
$\widetilde{g}_0$ is a bounded smooth function on $\overline{\widetilde{X}}\times\overline{W(z_0)}$, 
it follows by Taylor expansion that 
\begin{equation}\label{eq4.1dd}
	|\partial_{\bar{z}}\widetilde{g}_0(x,z)| \leq C_N' |x-x_+(z)|^{N},\quad x\in\widetilde{X},\ z\in W (z_0).
\end{equation}
This proves \eqref{eq4.1a} in the case $\alpha=\beta=\gamma=0$.
Now, observe that, since $(x-x_+(z))$ is a smooth function of $(x,z)$, 
it follows by Lemma \ref{lem:T} that after slightly shrinking $\widetilde{X}$ and $W (z_0)$, 
for any $\alpha,\beta,\gamma\in\N$, 
\begin{equation}\label{eq4.1dd2}
 	|\partial^{\alpha}_x\partial^{\beta}_z \partial_{\bar{z}}^{\gamma+1}\widetilde{g}_0(x,z)| 
	\leq C_N^{\alpha,\beta,\gamma} |x-x_+(z)|^{N}.
\end{equation}
In particular, restricting $x$ to the value $x_+(z)\in X_+$, that 
\begin{equation}\label{eq4.1ddd}
 	(\partial^{\alpha}_x\partial^{\beta}_z \partial_{\bar{z}}^{\gamma+1}g_0)(x_+(z),z)=0.
\end{equation}
Next, using  \eqref{eq4.0} and that $g_0 (x_+(z),z)=-\xi_+(z)$, 
see Proposition \ref{prop4.1}, we obtain by a direct computation of 
$\{\bar{p}_0,p_0\}(\rho_+(z))$ that 
\begin{equation}\label{eq4.1d}
\Ima (\partial_x g_0)(x_+(z),z) 
 = \frac{\{\bar{p}_0,p_0\}(\rho_+(z))}{2 i |q_0(\rho_+(z),z)|^2} < 0, 
\end{equation}
where the last inequality is a consequence of \eqref{eq1.8.1}. By differentiating \eqref{eq4.1c} 
with respect to $x$ and $\overline{z}$ and by evaluating it at the point $(x_+(z),\xi_+(z),z)$ 
we get that 
\begin{equation*}
 	(\partial_{\bar{z}}q_0)(x_+(z),\xi_+(z),z) =0.
\end{equation*}
By repeated differentiation of \eqref{eq4.1c} and induction using the Leibniz rule, we obtain 
\begin{equation}\label{eq4.1e}
 	(\partial_{\rho}^{\eta}\partial_{\bar{z}}q_0)(x_+(z),\xi_+(z),z) = 0, \quad \forall \eta\in \N^2.
\end{equation}
Let us now consider the higher order symbols. For $k\geq 0$, taking $q_k\in \mathcal{C}^{\infty}(V\times W(z_0))$ and 
$g_k\in \mathcal{C}^{\infty}(\pi_x(V)\times W(z_0))$ as in Prop.~\ref{prop4.1}, 
with the sets $V$, $W(z_0)$ as above, let $\widetilde{q}_k$ (resp. $\widetilde{g}_k$)
denote their respective almost $(x,\xi)$-holomorphic (resp. almost $x$-holomorphic) extensions. 

Recall \eqref{eq1.6}, and let $\widetilde{p}_N$ denote an almost $(x,\xi)$-holomorphic 
extension of $p_N$. The almost holomorphic extensions $\widetilde{g}_k$, 
$\widetilde{q}_k$ and $\widetilde{p}_N$ define an almost $(x,\xi)$-holomorphic extension of $G_N$ 
via \eqref{eq4.3b}. Set 
$f=\widetilde {G}_N- \frac{\widetilde{q}_N}{\widetilde{q}_0}(\xi+\widetilde{g}_0)-\widetilde{g}_N$. 
Since $f$ vanishes at $V$ by \eqref{eq4.3} and since $\partial_{\overline{x}}f$ and 
$\partial_{\overline{\xi}}f$ vanish to infinite order at $V$, we see by induction and 
Taylor expansion that all derivatives of $f$ vanish at $V$, hence 
\begin{equation}\label{eq4.1f}
	\widetilde {G}_N(x,\xi,z) 
	\sim \frac{\widetilde{q}_N(x,\xi,z)}{\widetilde{q}_0(x,\xi,z)}(\xi+\widetilde{g}_0(x,z))
	+\widetilde{g}_N(x,z),
\end{equation}
in the sense of equivalence stated in Definition \ref{def:AH}. 
\par
We are going to prove \eqref{eq4.1aa} by an induction argument. Suppose 
that \eqref{eq4.1ddd}, \eqref{eq4.1e} hold for $g_k$, $q_k$ for 
$0\leq k < N$. Then, by  \eqref{eq4.3b}, \eqref{eq4.1bb}, \eqref{eq4.1f} we see that 
for any $\alpha \in \N$, 
\begin{equation}\label{eq4.1g}
	0= \partial_x^{\alpha}\partial_{\bar{z}} \widetilde {G}_N(x,-\widetilde{g}_0(x,z),z) 
	\big|_{\substack{x=x_+(z)} }
	=\partial_x^{\alpha}\partial_{\bar{z}} \widetilde {g}_N(x,z) 
	\big|_{\substack{x=x_+(z)} }
	=(\partial_x^{\alpha}\partial_{\bar{z}} g_N)(x_+(z),z).
\end{equation}
Similarly as for \eqref{eq4.1e}, we obtain by repeated differentiation of \eqref{eq4.3b} 
and induction that 
\begin{equation}\label{eq4.1h}
	(\partial_{\rho}^{\eta}\partial_{\bar{z}} q_N)(x_+(z),\xi_+(z),z) = 0, \quad \forall \eta\in\N^2,
\end{equation}
and we have shown that \eqref{eq4.1ddd}, \eqref{eq4.1e} hold for $g_k$, $q_k$ with $0\leq k\leq N$. 
Since $X_{+}$ and $W(z_0)$ are relatively compact, a Taylor expansion shows that for any 
$k\in\N$, any $\alpha\in \N$ and any $N\in\N$ there exists a constant $C_{N,\alpha,k}>0$ such that 
\begin{equation*}
	|\partial^{\alpha}_x\partial_{\bar{z}} g_k(x,z)| \leq C_{N,\alpha,k} |x-x_+(z)|^{N}, \quad x\in X_+,\ z\in W (z_0). 
	\qedhere
\end{equation*}
\end{proof}
\subsection{Construction of the quasimodes}
From now on we will always assume that the symbol $p_0$ satisfies the hypothesis \eqref{eq1.8.1}. Using the 
decomposition given in Proposition \ref{prop4.1} we will construct quasimodes for the operators 
$P_h-z$ and $(P_h-z)^*$, following the WKB method. 
\begin{prop}[Quasimodes]\label{prop4.2}
Let $\Omega\Subset\C$ be as in \eqref{eq1.8.2}, let $z_0\in\Omega$ and 
let $p(\cdot;h)\in S(\R^2,m)$ be as in \eqref{eq1.6} and satisfy \eqref{eq1.8.1}. 
Let $W(z_0)$ and $X_{\pm}^j$, $j=1,\dots,J$, be as in Proposition~\ref{prop4.1}. 
 Let $\chi_{\pm}^j\in\mathcal{C}^{\infty}_0(X_{\pm}^j,[0,1])$,  
 such that $\chi_{\pm}^j\equiv 1$ in a small neighbourhood of $x_{\pm}^j(\overline{W(z_0)})$.
Then, there exist functions
\begin{equation}\label{eq4.2.0}
\begin{split}
	&e_{\pm}^{j,hol}(x,z;h)=a_{\pm}^j(x,z;h)\chi_{\pm}^j\e^{\frac{i}{h}\varphi_{\pm}(x,z)}, 
	\quad  z\in W(z_0)\\
	&\varphi_{\pm}^j(x,z)= - \int_{x_{+}^j(z_0)}^x g_0^{j,+}(y,z)dy, \quad 
	\varphi_{-}^j(x,z)= - \int_{x_{-}^j(z_0)}^x \overline{ g_0^{j,-}(y,z)}dy,
\end{split}
\end{equation}
where  $g_0^{\pm,j}$ are as in Proposition \ref{prop4.1}, which depend 
smoothly on $x\in X_{\pm}^j$ and $z\in W(z_0)$, with    
$a_{\pm}^j(x,z;h)\sim (a_0^{\pm,j}(x,z)+ha_1^{\pm,j}(x,z)+\dots)$ 
depending smoothly on $x$ and $z$ such that all derivates
are uniformly bounded as $h\to 0$.
\par
Moreover, the states $e_{\pm}^{j,hol}$ have the following properties:
\begin{enumerate}
\item 
For all $z\in W(z_0)$,  and any $n\in\N$ 
\begin{equation*}
\begin{split}
	&(\partial_x^n\partial_{\bar{z}}a_{+}^j)(x_+^j(z),z;h) = \mO(|z-z_0|^{\infty}+h^{\infty}), \\
	&(\partial_x^n\partial_{z}a_-^j)(x_-^j(z),z;h) = \mO(|z-z_0|^{\infty}+h^{\infty})
\end{split}
\end{equation*}
\item The $L^2$ norms of those states satisfy
\begin{equation}\label{eq4.2.1}
	\begin{split}
	 &\|e_{\pm}^{j,hol}(z;h)\| =\e^{\frac{1}{h}\Phi_{\pm}^j(z;h)}, \\
	&\Phi_{\pm}^j(z;h)=-\Ima \varphi_{\pm}^j(x_{\pm}^j(z),z) + h\log \big(h^{\frac{1}{4}}A_{\pm}^j(z;h)\big),
	\end{split}
\end{equation}
with $\Ima \varphi_{\pm}^j(x_{\pm}^j(z),z)\leq 0 $, with equality iff $z=z_0$, 
and $A_{\pm}^j(z;h)\sim A_0^{j,\pm}(z)+hA_1^{j,\pm}(z)+\dots$ depending smoothly on $z$ 
such that all derivatives with respect to $z$ and $\bar{z}$ are bounded when $h\to 0$. 
\item Normalizing those states,
\begin{equation}\label{eq4.2.2}
e_{\pm}^j(x,z;h) \defeq e_{\pm}^{j,hol}(x,z;h)\e^{-\frac{1}{h}\Phi_{\pm}^j(z;h)},
\end{equation}
then we see that they are $h^\infty$-quasimodes:
\begin{equation}\label{eq4.5}
	\|(P_h-z)e_+^j\| = \mO(h^\infty), \quad \|(P_h-z)^*e_-^j\| = \mO(h^\infty).
\end{equation}
\item For all $\psi\in\mathcal{C}^{\infty}_0(\R^2,[0,1])$, such that $\psi \equiv 1$ near $\rho^j_{\pm}$, 
	and any order function $m'$, 
\begin{equation}\label{eq4.6}
	\|(1-\psi^w)e_+^j\|_{H(m')} = \mO(h^\infty), \quad \|(1-\psi^w)e_-^j\|_{H(m')} = \mO(h^\infty).
\end{equation}
\end{enumerate}
\end{prop}
For future use, we voluntarily introduced two versions of the quasimodes: the normalized ones $e_+^j(z;h)$, and the almost holomorphic ones $e_{\pm}^{j,hol}(z;h)$. 

\begin{proof}[Proof of Proposition \ref{prop4.2}]
We will give the proof of this Proposition only in the "+" case, since 
the "-" case is similar. To simplify notation, we will suppress the superscript 
$j$ until further notice. We begin with the following result:
\begin{lem}\label{lem4.1}
Let $\Omega\Subset\C$ be as in \eqref{eq1.8.2}, let $z_0\in\Omega$ and 
let $p(\cdot;h)\in S^2(\R^2,m)$ be as in \eqref{eq1.6} and satisfy \eqref{eq1.8.1}. 
Let $W(z_0)$ and $X_+$ with $x_+(\overline{W(z_0)})\subset X_+ $ be as in Proposition~\ref{prop4.1}. Let $g^+(x,z;h)$ be the symbol
constructed in Proposition~\ref{prop4.1}.

Then, the equation 
\begin{equation}\label{e:simple}
(hD_x+g^+(x))f_{+}(x,z;h) =0,\quad (x,z)\in X_+ \times W(z_0) \,,
\end{equation}
admits a solution $f_{+}^{hol}(x,z;h)$ of the form 
\begin{equation}\label{eq4.4}
	\begin{split}
	f_{+}^{hol}(x,z;h)&=a_{+}(x,z;h)\,\e^{\frac{i}{h}\varphi_{+}(x,z)}, \quad 
	(x,z)\in X_+ \times W(z_0) \\ 
	\text{with}\quad\varphi(x)&= - \int_{x_+(z_0)}^x g_0^{+}(y,z)dy\,.
	\end{split}
\end{equation}
The symbol 
$a_{+}(x,z;h)\sim a_0^{+}(x,z)+ha_1^{+}(x,z)+\cdots$ 
depends smoothly on $x$ and $z$, such that all derivates are bounded as $h\to 0$.
Moreover, for all $z\in W(z_0)$ 
and any $n\in\N$,
\begin{equation*}
	(\partial_x^n\partial_{\bar{z}}a^+)(x_+(z),z;h) = \mO(|z-z_0|^{\infty}+h^{\infty}).
\end{equation*}
\end{lem}
\begin{proof}
For any $h\in(0,1]$ and $z\in W(z_0)$, the first order equation~\eqref{e:simple} can be easily solved by the Ansatz
$$
f_+^{hol}(x,z;h)\defeq \exp\Big( -\frac{i}{h}\int_{x_0}^x g^+(y,z;h)\,dy\Big)\,,
$$
where we choose to take the reference point $x_0=x_+(z_0)$ independent of $z$.
Taking into account the expansion \eqref{eq4.1} of the symbol $g^+$, its primitive may be expanded as
$$
-\int_{x_0}^x g^+(y,z;h)\,dy \sim  \varphi^+_0(x,z)+ h \varphi_1^+(x,z)+\cdots,\qquad \varphi^+_k(x,z)\defeq  - \int_{x_0}^x g^+_k(y,z)\,dy \,.
$$
Separating the first term $\varphi_+=\varphi^+_0$ from the subsequent ones, we may write 
$$
f_+^{hol}(x,z;h) = \e^{\frac{i}{h}\varphi_+(x,z)}\,a^+(x,z)\,,
$$
with the symbol $a^+(x,z)\in C^\infty(X_+\times W(z_0))$ admitting an expansion
$$
a^+(h)\sim a_0^+ + h a_1^+ + h^2 a_2^+ +\cdots\,,
$$
where each term $a_j^+\in\mathcal{C}^{\infty}(X_+\times W(z_0))$ depends on the functions $\{\varphi^+_k, k=1,\ldots,j+1\}$.

Alternatively, the expansion $\sum h^j a^+_j$ can be constructed order by order through a WKB construction (see e.g. \cite{DiSj99}): one solves iteratively the \textit{transport equations}
\begin{equation*}
	-i\partial_x a^+_n(x,z) + \sum_{k=0}^n g^+_{n+1-k}(x,z)a^+_k(x,z)=0, \quad n\geq 0\,,
\end{equation*}
by the expressions
\begin{equation}\label{eq4.9.5}
\begin{split}
	a_0^+(x,z)&= \e^{-i\int_{x_0}^x g^+_1(y,z;h)\,dy}\,,\\
	a_n^+(x,z)&=-i \, a_0^+(x,z) \int_{x_0}^x \sum_{k=0}^{n-1}\frac{g^+_{n+1-k}a^+_k}{a_0^+}\,(y,z)\,dy\,.
\end{split}
\end{equation}

Let us make some remarks about the phase function $\varphi_+$. It is the unique solution to the 
eikonal equation $\partial_x \varphi_+(x,z) + g_0^+ (x,z)= 0$, 
satisfying the boundary condition $\varphi_+(x_0,z)=0$. 
By Proposition \ref{prop4.1},
$\partial_x \varphi_+(x_+(z);z)=\xi_+(z)\in\R$, therefore $x_+(z)$ is a critical point of 
$\Ima \varphi_+$. Furthermore, by \eqref{eq4.1e},
\begin{equation}\label{eq4.7}
 \Ima \partial^2_x\varphi_+(x_+(z),z)=-\Ima (\partial_x g_0^+)(x_+(z),z) >0,
\end{equation}
hence $x_+(z)$ is a non-degenerate critical point of $\Ima \varphi(\cdot,z)$. By 
potentially shrinking $W(z_0)$  and $X_+$, we can arrange that 
$x_{+}(\overline{W(z_0)})\subset X_+$ and that \eqref{eq4.7} holds for all $x\in X_+$, so that
$x_+(z)$ is the unique critical point of $\Ima\varphi_+(\cdot,z)$ in $X_+$.

By repeated differentiation of 
\eqref{eq4.9.5} we check by using \eqref{eq4.1aa} that for all $x\in X_+$, all $z\in W(z_0)$, any $k\in\N$ and any $n\in\N$ 
\begin{equation*}
	(\partial_x^n\partial_{\bar{z}} a_k^+)(x,z) = \mO(|x-x_+(z)|^{\infty} + |x_+(z_0)-x_+(z)|^{\infty}).
\end{equation*}
By Taylor expansion of $x_+(z)$ around $z_0$, we obtain at the critical point:
\begin{equation*}
	(\partial_x^n\partial_{\bar{z}} a_k^+)(x_+(z),z) = \mO(|z-z_0|^{\infty}).
\end{equation*}
Summing over all the symbols, we obtain for the full symbol:
\begin{equation*}
\forall n\in \N,\ \forall z\in W(z_0),\quad	(\partial_x^n\partial_{\bar{z}}a^+)(x_+(z),z;h) = \mO(|z-z_0|^{\infty}+h^{\infty}).
\end{equation*}
\end{proof}
\begin{rem}
In the "$-$" case we construct a solution for $(hD_x + \overline{g^-(x,z)})f_-=0$. Hence, the phase function reads
$\varphi_-(x;z)=-\int_{x_-(z_0)}^x  \overline{g_0^{-}(y,z)}dy$. 
Moreover, the transport equations depend on $\overline{g_0^{-}(y,z)}$, which is almost anti-holomorphic w.r.t. $z$ at the point $(x_+(z),z)$. Hence, for any $n\in\N$ we obtain 
$(\partial_x^n\partial_{z}a^-)(x_-(z),z;h) = \mO(|z-z_0|^{\infty}+h^{\infty})$.
\end{rem}
Let us proceed with the proof of the Proposition. Let $\chi_+\in\mathcal{C}^{\infty}_0(X_+,[0,1])$, such that 
$\chi_+\equiv 1 $ in a small neighbourhood 
of $x_+(W(z_0))$. We define the smooth function
\begin{equation*}
	e_+^{hol}(x,z)\defeq \chi_+(x,z)\,f_+^{hol}(x,z;h),\quad (x,z)\in X_+\times W(z_0)\,.
\end{equation*}
Recall from \eqref{eq4.7} and the discussion afterwards, 
that $x_+(z)$ is the unique critical point of $\Ima \varphi(\cdot,z)$ in $X_+$, 
and that
\begin{equation}\label{eq4.11}
 \Ima \partial^2_x \varphi(x_+(z),z) >0.
\end{equation}
In particular we see that $\Ima \varphi(x,z)-\Ima \varphi(x_+(z),z) \geq 0$ for 
all $x\in X_+$, with a strict inequality for $x\neq x_+(z)$. Hence, applying the method of stationary phase, we find that
\begin{equation}\label{eq4.12}
	\|e_+^{hol}\| =h^{1/4}A_+(z;h)\,\e^{\frac{1}{h}\Phi_{+,0}(z)} \defeq \e^{\frac{1}{h}\Phi_{+}(z;h)}
\end{equation}
where $$\Phi_{+,0}(z) = -\Ima \varphi_+(x_+(z),z),$$
while
the symbol $A_+(z;h)\sim A_0^+(z)+hA_1^+(z)+\dots$ depends smoothly on $z$, and 
all derivatives in $z,\bar z$ are bounded when $h\to 0$. The principal symbol 
\begin{equation}\label{eq4.13.2}
	A_0^+(z)=\left(\frac{\pi |a_0^+(x_+(z),z)|^2}
			{\Ima \partial^2_x\varphi_+(x_+(z),z)}\right)^{1/4} >0.
\end{equation}
It follows from \eqref{eq4.7} and the property $\varphi^+(x_+(z_0),z)=0$ that $\Phi_{+,0}(z)\geq 0$ for any $z\in W(z_0)$, with equality precisely when 
$z=z_0$. Hence, for points such that $|z-z_0|\geq 1/C$, the norms of the states $e_+^{hol}(\bullet,z)$ are exponentially large. 
Using Lemma \ref{lem4.1}, we have 
\begin{equation}\label{eq4.13.1}
	(hD_x +g^+)\chi_+f_+^{hol}= 
	[hD_x,\chi_+]f_+^{hol}.
\end{equation}
Using that $x_+(z) \notin \supp\chi_+'$ and the fact that $-\Ima \varphi_+(x,z)$ reaches its maximum only at $x=x_+(z)$, we get
\begin{equation*}
 \|[hD_x,\chi_+]f_+^{hol}\| = \mO(\e^{-1/Ch}) \e^{\frac{1}{h}\Phi_{+,0}(z)}\,,
\end{equation*}
which shows that 
$$
\|(hD_x +g^+(z))\,e_+^{hol}(z)\|=\mO(e^{-C/h})\,\|e_+^{hol}(z)\|\,.
$$
Next, using \eqref{eq4.12}, we define the $L^2$-normalised state
\begin{equation}\label{eq4.14}
e_+(x,z;h) \defeq e_+^{hol}(x,z;h)\,\e^{-\frac{1}{h}\Phi_+(z;h)}
\end{equation}
which is $\mathcal{C}^{\infty}_0(X_+)$, and depends smoothly on $z\in W(z_0)$. 
This normalized state enjoys precise microlocalization properties, as shown in the following
\begin{lem}\label{lem4.3}
Under the assumptions of Lemma \ref{lem4.1}, if we take any
order function $\widetilde{m}$, and for an arbitrary $z\in W(z_0)$, take any cutoff $\psi\in\mathcal{C}^{\infty}_0(\R^2,[0,1])$ such that
$\psi=1$ near $\rho_+(z)$, we have
\begin{equation}
	\| (1-\psi)^w e_+(z;h)\|_{H(\widetilde{m})} = \mO(h^{\infty}).
\end{equation}
This shows that $e_+(z;h)$ is microlocalized on the point $\rho_+(z)$. %
\end{lem}
\begin{proof}
We smoothly extend $g_+\in\mathcal{C}^{\infty}(X_+)$  to\footnote{This notation should not be confused with the notation $\widetilde{g}_0^+$ used for the almost holomorphic extension of $g_0^+$ in Lemma~\ref{lem:AH}.} $\widetilde{g}_+$ in $\mathcal{C}^{\infty}(\R)$, such that 
$\widetilde{g}_+(x,z;h) = -\frac{i}{C}(x-x_+(z))$, for $|x|\geq C$, $C>0$, and such 
that $\xi+\widetilde{g}_+ \in S(\R^2,\langle \rho \rangle)$ is elliptic outside $\rho_+(z)$. 
This is possible due to \eqref{eq4.7}. Thus, the operator $hD_x + \widetilde{g}_+ = (\xi+\widetilde{g}_+)^w$.

Since $\supp\chi_+\subset X_+$, we have that $\widetilde{g}_+=g_+$ 
on the support of $e_+$ with respect to $x$. Hence, by \eqref{eq4.13.1}, \eqref{eq4.14}, 
\begin{multline}\label{eq4.15}
	(hD_x +\widetilde{g}^+)^*(hD_x +\widetilde{g}^+)e_+ = \eta\,,\\
	\eta= \e^{-\frac{1}{h}\Phi_+(z;h)} \Big( [hD_x, [hD_x,\chi_+]] f_+^{hol}  - 2i \Ima \widetilde{g}^+ [hD_x,\chi_+] f_+^{hol} \Big)
\end{multline}
Using that $x_+(z) \notin \supp\chi_+'$, as well as \eqref{eq4.11}, 
the norm  $\|\eta \| = \mO(\e^{-C/h})$.
\par
Let $\widetilde{\chi}\in\mathcal{C}^{\infty}_0(\R^2,[0,1])$, $\widetilde{\chi}=1$ near $\rho_+(z)$, and define 
\begin{equation}\label{eq4.16}
	q^w = (hD_x +\widetilde{g}^+)^*(hD_x +\widetilde{g}^+) + \widetilde{\chi}^w.
\end{equation}
Notice that its symbol
\begin{equation*}
	q(x,\xi,z;h) = |\xi +\widetilde{g}_0^+(x,z)|^2 + \widetilde{\chi}(x,\xi) + \mO(h)_{S(\R^2,\langle \rho\rangle^2)}, 
\end{equation*}
is elliptic. Hence, for $h>0$ small enough, $q^w$ admits a bounded inverse $b^w$, 
$b\in S(\R^2, \langle \rho\rangle^{-2})$. Notice that by \eqref{eq4.15}, \eqref{eq4.16}, 
\begin{equation*}
	e_+ = b^w \eta+ (b\#\widetilde{\chi})^w e_+.
\end{equation*}
Let $\psi\in\mathcal{C}^{\infty}_0(\R^2,[0,1])$, $\psi=1$ near $\supp\widetilde{\chi}$, 
thus for every order function $\widetilde{m}$,
\begin{equation*}
	(1-\psi)\# b\#\widetilde{\chi} =\mO(h^{\infty})_{S(\R^2, \widetilde{m}^{-1})}.
\end{equation*}
Let us consider the term $((1-\psi)\#b)^w\eta$. 
By repeated integration by parts from \eqref{eq4.15}, one can show that  $\|\eta\|_{H(\widetilde{m})}=\mO(\e^{-C/h})$; since $b^w$ and $(1-\psi)^w$ are bounded on $H(\widetilde{m})$, it follows that $\|((1-\psi)\#b)^w\eta\|_{H(\widetilde{m})} =\mO(\e^{-C/h})$. Adding the two contributions, we get the announced estimate $\|(1-\psi)^w e_+\|_{H(\widetilde{m})} =\mO(h^{\infty}) $. 
\end{proof}
Combining Proposition \ref{prop4.1} with the Lemmata \ref{lem4.1}-\ref{lem4.3}, 
we get that for any $\psi\in\mathcal{C}^{\infty}_0(\R^2)$, $\psi = 1$ near $\rho_+(z)$
\begin{equation}\label{eq4.17}
\begin{split}
 (P_h-z)e_+ &= (P_h-z)\psi^w e_+ +(P_h-z)(1-\psi)^w e_+\\ 
 	& = (\psi\#q^+\#(\xi+g^+))^w e_+ + [P_h,\psi^w]e_+ +\mO(h^{\infty}) \\
	& = \mO(h^{\infty}).
 \end{split}
\end{equation}
Here the first term in the second line is $\mO(h^{\infty})$ by 
construction of $e_+$. The second term is $\mO(h^{\infty})$ by Lemma \ref{lem4.3}, 
since the support 
of the symbol of the commutator is (up to an $\mO(h^{\infty})$ term) disjoint 
from a fixed neighbourhood of $\rho_+(z)$. 
\\
\par
This concludes the proof of Proposition \ref{prop4.2} in the "+" case. The "$-$" case can 
be obtained following the same steps to construct a quasimode for $(P_h-z)^*$, using the microlocal factorization of this operator into  $(q^{-,w})^* (hD_x+\overline{g^-})$ by 
Proposition \ref{prop4.1}. 
\end{proof}
\subsection{Quasimodes for symmetric symbols}\label{sec:QuasModSym}
Let us now assume \eqref{eq.15}, i.e. that the symbol is symmetric:
\begin{equation*}
	p(x,\xi;h) = p(x,-\xi;h).
\end{equation*}
Then the formal adjoint $(P_h-z)^*$ satisfies 
\begin{equation}\label{eq4.3.2}
	(P_h-z)^* = \Gamma ( P_h-z )\Gamma, 
	\quad \Gamma u \defeq\overline{u}.
\end{equation}
Moreover, \eqref{eq.15} implies that if $\rho=(x,\xi) \in p_0^{-1}(z)$ with 
$\{\Rea p,\Ima p\}(\rho)<0$, as in \eqref{eq1.8.2}, \eqref{eq1.8.1}, then 
$(x,-\xi) \in p_0^{-1}(z)$ with $\{\Rea p,\Ima p\}(x,-\xi)>0$. Thus, the hypotheses \eqref{eq1.8.1}, \eqref{eq1.8.1a} write as 
\begin{equation}\label{eq4.3.3}
	\begin{split}
	&p_0^{-1}(z)=\{\rho_{\pm}^j(z)=(x^j_{\pm}(z), \xi^j_{\pm}(z))
		=(x^j(z),\pm \xi^j(z)); j=1,\dots,J\},\text{ with } \\
	&\pm \{\Rea p,\Ima p\}(\rho_{\pm}^j(z)) <0, 
	 \text{ and } x^i\neq x^j, \text{if}\  i\neq j.
	\end{split}
\end{equation}
For  $j=1,\dots, J$ let $e_+^{j,hol}$, $\chi_+^j$, $\varphi_+^j$, $W(z_0)$, and $\Phi_+^j$ 
be as in Proposition \ref{prop4.2}. Let us define the "-" quasimode as:
\begin{equation}\label{eq4.3.4}
	e_-^{j,hol}(z;h) = \Gamma e_+^{j,hol}(z;h). %
\end{equation}
It is then clear form Proposition \ref{prop4.2} that for all $z\in W(z_0)$
\begin{equation}\label{eq4.3.5}
	\|e_{-}^{j,hol}(z;h)\| =
	\e^{\frac{1}{h}\Phi_+^j(z;h)}.
\end{equation}
Moreover, by \eqref{eq4.3.2}, \eqref{eq4.5}, we see that 
since  
\begin{equation}\label{eq4.3.5b}
	e_-^j(z;h)=e_{-}^{j,hol}(z;h)\,\e^{-\frac{1}{h}\Phi_+^j(z;h)} = \Gamma e_+^j(z;h)\,,
\end{equation}
$e_-^j(z;h)$ is indeed a quasimode:
\begin{equation*}
 \|(P_h-z)^*e_-^j(z;h)\|=\mO(h^\infty).
\end{equation*}
Transposing the proof of Proposition~\ref{prop4.2}, we also obtain that 
for all $\psi\in\mathcal{C}^{\infty}_0(\R^2)$, $\psi =1$ near $\rho^j_{-}(z)$, 
\begin{equation}\label{eq4.3.6}
	 \|(1-\psi^w)e_-^j(z;h)\|_{H(m')} = \mO(h^\infty),
\end{equation}
and for all $\psi\in\mathcal{C}^{\infty}_0(\R^2,[0,1])$, $\chi =1$ near the point $(x_-^j(z),2\xi_-^j(z))$, 
\begin{equation}\label{eq4.3.7}
	 \|(1-\psi^w)(e_-^j(z;h))^2 \|_{H(m')} = \mO(h^\infty).
\end{equation}
In the sequel we shall keep that $\pm$-notation to distinguish the quasimodes, even for 
symmetric symbols \eqref{eq.15}.
\subsection{Relation with the symplectic volume}
In this section we will study the functions $\Phi_{\pm}^j(z;h)$ governing the $L^2$ norm of the holomorphic quasimodes (see Prop.~\ref{prop4.2}). 
We will strongly make use of the almost holomorphicity of Lemma~\ref{lem:AH}. From the expression \eqref{eq4.2.1}, let us write
\begin{equation}\label{eq4.4.1}
	\Phi_{\pm}^j(z;h) =\Phi_{\pm,0}^j(z) + h\log \big(h^{\frac{1}{4}}A{_\pm}^j(z;h)\big), \quad
	\Phi_{\pm,0}^j(z)\defeq-\Ima \varphi_{\pm}^j(x_{\pm}^j(z),z).
\end{equation}
From now on, let us focus on the function $\Phi_{+}^j(z;h) $, and omit to indicate the super/subscript $_+^j$ on all quantities. 
\par
By \eqref{eq4.7} and the ensuing discussion, $x(z)=x_{+}^j(z)$ is the unique 
minimum of $x\mapsto \Ima \varphi(x,z)$ on the support of $\chi_{+}^j$. 
Using the equations~\eqref{eq4.1b}, \eqref{eq4.2.0} and a Taylor expansion of $x(z)$ at $z_0$, we easily obtain
\begin{equation}\label{eq4.4.2}
(\partial_{\bar{z}}\varphi)(x(z),z) = \mO(|x(z) - x(z_0)|^{\infty}) =  \mO(|z - z_0|^{\infty}),
\end{equation}
which leads to
\begin{equation}\label{eq4.4.5}
	\frac{2}{i} \partial_z\Phi_{0}^j(z)= (\partial_z\varphi)(x(z),z) +
	 \mO(|z - z_0|^{\infty})\,.
\end{equation}
Next, recall that by Proposition \ref{prop4.1} and \eqref{eq4.2.0}, we have 
$$
\partial_x\varphi(x,z)= - g_0(x,z),\quad\text{in particular}\quad  \partial_x\varphi(x(z),z)= \xi(z).
$$
Differentiating w.r.t. $z$ and $\bar{z}$, we get 
\begin{equation}\label{dxxphi}
(\partial_{xx}^2\varphi)(x(z),z)\partial_z x(z) + (\partial_{xz}^2\varphi)(x(z),z)
	= \partial_z \xi (z)
\end{equation}	
Differentiating w.r.t. $\bar{z}$ and using Lemma~\ref{lem:AH}, we obtain
$$
(\partial_{xx}^2\varphi)(x(z),z)\partial_{\bar{z}}x(z) +  \mO(|z - z_0|^{\infty}) = \partial_{\bar{z}}\xi(z).
$$
Eq.~\eqref{eq4.4.2} is a form of almost holomorphy at the point $(x(z_0),z_0)$. Using Eq.~\eqref{eq4.1ddd} one finds a natural extension of this property: 
$$
(\partial^2_{z\bar{z}}\varphi)(x(z),z)=\mO(|z - z_0|^{\infty})\,.
$$
Using this expression, we find 
$$
\partial^2_{z\bar z}\big(\varphi(x(z),z) \big) = \partial_{z}\xi(z) \partial_{\bar z} x(z) + \xi(z)\partial^2_{z\bar z}x(z) + \mO(|z - z_0|^{\infty})\,.
$$
Taking the imaginary part of this equation produces 
\begin{equation}\label{eq4.4.2a}
	\frac{2}{i} \partial^2_{z\bar{z}}\Phi_{0}(z)=
	 \partial_z\xi (z)\partial_{\bar{z}} x(z) - \partial_{\bar{z}}\xi(z)\partial_z x(z)
		+ \mO(|z - z_0|^{\infty}).
\end{equation}
We now restore the $+,j$ notations, and write this expression using 2-forms:
\begin{equation}\label{eq4.4.3.1}
	\begin{split}
	-d\xi_+^j\wedge dx_+^j(z) &= 
	\big(\partial_z\xi_+^j(z)\partial_{\bar{z}}x_+^j(z) - \partial_{\bar{z}}\xi_+^j(z)\partial_zx_+^j(z) \big)\,
	d\overline{z}\wedge dz \\
	& = \left( 4\partial^2_{z\bar{z}}\Phi_{+,0}^j(z)+ \mO(|z - z_0|^{\infty})\right) \frac{d\overline{z}\wedge dz}{2i}.  \\ 
	\end{split}
\end{equation}
This expressions provides the connection between the volume form in phase space $d\xi\wedge dx$, and the volume form in energy space $\frac{d\bar z \wedge dz}{2i}$.

One can perform the symmetric computations for the functions $\Phi_{-,0}^j(z)$, and obtain
\begin{equation}\label{eq4.4.3}
	\begin{split}
	d\xi_-^j\wedge dx_-^j(z) &= 
	\big( \partial_{\bar{z}}\xi_-^j(z)\partial_zx_-^j(z)-\partial_z\xi_-^j(z)\partial_{\bar{z}}x_-^j(z) \big)
	\,d\overline{z}\wedge dz \\
	& = \left(4 \partial^2_{z\bar{z}}\Phi_{-,0}^j(z)+ \mO(|z - z_0|^{\infty})\right) \frac{d\overline{z}\wedge dz}{2i}\,.
	\end{split}
\end{equation}
Let us now express the factor $4\partial^2_{z\bar{z}}\Phi_{+,0}^j(z)$ in terms of the symbol $p_0$. 
Differentiating the identity $p_0(\rho_{\pm}^j(z))=z$, cf. \eqref{eq1.8.1}, we obtain the linear 
system
\begin{equation*}
	\begin{cases}
	& \partial_{\xi}p_0(\rho_{\pm}^j)\partial_z \xi_{\pm}^j + \partial_{x}p_0(\rho_{\pm}^j)\partial_z x_{\pm}^j =1 \\
	&\partial_{\xi}\overline{p}_0(\rho_{\pm}^j)\partial_z \xi_{\pm}^j
	 + \partial_{x}\overline{p}_0(\rho_{\pm}^j)\partial_z x_{\pm}^j =0,
	\end{cases}
\end{equation*}
which can be solved by
\begin{equation*}
 \partial_z\xi_{\pm}^j(z) = \frac{-\partial_x\overline{p}_0}{\{\overline{p}_0,p_0\}}(\rho_{\pm}^j(z)), 
 \quad 
  \partial_z x_{\pm}^j(z) = \frac{\partial_{\xi}\overline{p}_0}{\{\overline{p}_0,p_0\}}(\rho_{\pm}^j(z)).
\end{equation*}
Using the fact that $x_{\pm}^j$, $\xi_{\pm}^j$ are real, we get by \eqref{eq4.4.2a}, and 
by a similar computation for $\Phi_{-,0}^j$, that  
\begin{equation*}
	4 \partial^2_{z\bar{z}}\Phi_{\pm,0}^j(z) = \left(\frac{\mp 1}{\frac{1}{2i}\{\overline{p}_0,p_0\}(\rho^j_{\pm}(z))}  
	+ \mO(|z - z_0|^{\infty})\right).
\end{equation*}
It follows from \eqref{eq1.8.1} that the first term on the RHS is positive, hence that the functions $\Phi_{\pm}^j$ are strictly subharmonic near $z_0$. 

Identifying the Lebesgue measure $L(dz)$ on the energy plane with the volume form $ \frac{d\overline{z}\wedge dz}{2i}$, and denoting by $|d\xi\wedge dx|$ the symplectic volume on $T^*\R$, the expression \eqref{eq4.4.3.1} (resp. \eqref{eq4.4.3})
relates an infinitesimal volume in the energy space, with an infinitesimal volume near the phase space point $\rho^j_+$ (resp. $\rho^j_-$). 
Adding the contributions of all the $2J$ points in $p_0^{-1}(z)$, we obtain the following expression for the push-forward of the symplectic volume $|d\xi\wedge dx|$ through the principal symbol $p_0$:
\begin{equation}\label{eq4.4.4}
\begin{split}
	&(p_0)_*(|d\xi\wedge dx|) =  \sum_{j=1}^J \big( \sigma_{+}^j(z)+ \sigma_{-}^j(z) \big) L(dz), \\ 
	& \text{with } \sigma_{\pm}^j(z) \defeq  4 \partial^2_{z\bar{z}}\Phi_{\pm,0}^j(z) + \mO(|z - z_0|^{\infty})\,.
\end{split}
\end{equation}
Observe that in case we assume the symmetry \eqref{eq.15}, we get  $\sigma_{+}^j(z)=\sigma_{-}^j(z)$.
\begin{rem}
The error term $\mO(|z - z_0|^{\infty})$ will be very small in the following, because we will investigate values of $z$ in some $h^{1/2}$-neighbourhood of $z_0$. 
\end{rem}
%
\section{Interaction between the quasimodes}\label{sec:QuasInter}
In this section we study the overlaps (scalar products) between nearby quasimodes. In our notations, the $L^2$ scalar product $(u|v)$ is linear w.r.t. $u$ and antilinear w.r.t. $v$.

From the assumption \eqref{eq1.8.1}, %
we may shrink implies that the neighbourhood $W(z_0)$, 
(see Proposition~\ref{prop4.2}) to ensure that 
\begin{equation}\label{eq4.17.1}
\rho_{\pm}^j(\overline{W(z_0)})\cap \rho_{\pm}^k(\overline{W(z_0)})=\emptyset\quad \text{for any } j\neq k. 
\end{equation}
Under the further assumption \eqref{eq1.8.1a}, we may even assume (up to shrinking $W(z_0)$ that
the cutoffs functions $\chi_{\pm}^j$ used to construct our quasimodes have disjoint supports:%
\begin{equation}\label{eq4.17.1b}
	\supp\chi_\pm^j \cap \supp\chi_\pm^k  %
	=\emptyset,\quad \text{for any } j\neq k. 
\end{equation}
Thus, from \eqref{eq4.17.1} and the microlocalization property \eqref{eq4.6},  we have for all $z\in W(z_0)$,%
\begin{equation}\label{eq4.6.1}
	(e_{+}^j(z)|e_{+}^k(z)) = \delta_{jk} + \mO(h^{\infty}), \quad 
	(e_{-}^j(z)|e_{-}^k(z)) = \delta_{jk} + \mO(h^{\infty}),
\end{equation}
uniformly for all $z\in W(z_0)$ (here $\delta_{jk}$ denotes the Kronecker symbol). %
\par
Similarly, one obtains by \eqref{eq4.6}, \eqref{eq4.17.1} that for all $z,w\in W(z_0)$ 
\begin{equation}\label{eq4.6.1b}
	(e_{\pm}^j(z)|e_{\pm}^k(w)) =  \mO(h^{\infty}),  \quad \text{for } j\neq k
\end{equation}
This quasi-orthogonality reflects the fact that for $j\neq k$, the points $\rho_{\pm}^j(z)$ and $\rho_{\pm}^k(w)$ remain at positive distance from each other, uniformly when $z,w\in W(z_0)$.
%
\subsection{Overlaps between nearby quasimodes}
The scalar product between quasimodes localized on nearby points is given in the following Propositions, which will be crucial for us later. The first one deals with the overlaps between the "+" quasimodes.
\begin{prop}\label{prop4.3}
Let $(e_{+}^{j,hol}(z)_{z\in W(z_0)}$ be the quasimodes constructed in Proposition \ref{prop4.2}. Recall the notation $\Phi_{+,0}^j(z;h)$ 
from \eqref{eq4.4.1}. Then, for $|z-w|\leq c$, with $c>0$ sufficiently small, 
\begin{equation}\label{eq4.18}
	( e_{+}^{j,hol}(z)|e_{+}^{j,hol}(w))
	=\e^{\frac{2}{h}\Psi^j_{+}(z,w;h)} + \mO(h^{\infty})
	\e^{\frac{1}{h}\Phi_{+,0}^j(z)+\frac{1}{h}\Phi_{+,0}^j(w)}
\end{equation}
with 
\begin{equation}\label{eq4.18a}
	\Psi^j_{+}(z,w;h) = \Psi^j_{+,0}(z,w)+\frac{h}{2}\log \big( h^{\frac{1}{2}} b^j_+(z,w;h) \big)\,.
\end{equation}
Here $\Psi^j_{+,0}(z,w)$ is almost $z$-holomorphic and almost $w$-antiholomorphic 
at $\{(z_0,z_0)\}$,  and $b_+^j(z,w;h) \sim b^{+,j}_0(z,w) + h b^{+,j}_1(z,w) + \dots $ 
is smooth in $z$ and $w$, with any derivative with respect to $z,\bar{z},w,\bar{w}$ uniformly 
bounded as $h\to 0$. Moreover, 
\begin{itemize}
	\item $\overline{\Psi_{+}^j(z,w;h)} = \Psi_{+}^j(w,z;h)$,
	\item  $\Psi^j_{+,0}(z,z)= \Phi_{+,0}^j(z)$,
	\item $b_+^j(z,z;h) = A_+^j(z,h)^2$, %
		 with $A_+^j(z,h)$ given in Proposition \ref{prop4.2},
	\item for $|\zeta_i|\leq c$, $i=1,2$, $c>0$ sufficiently small, and for any $N\in\N$, 
\begin{equation*}
	\Psi_{+,0}(z_0+\zeta_1, z_0+\zeta_2) = 
	\sum_{|\alpha|\leq N} (\partial^{\alpha_1}_z\partial^{\alpha_2}_{\bar{z}}\Phi_{+,0})(z_0)
		\frac{\zeta_1^{\alpha_1}\overline{\zeta}_2^{\alpha_2}}{\alpha !} 
		+ \mO(|\zeta |^{N+1}).
\end{equation*}
\end{itemize}
\end{prop}
The second proposition, symmetric to the previous one, deals with the overlaps between nearby "$-$" quasimodes.
\begin{prop}\label{prop4.3b}
Let $(e_{-}^{j,hol}(z)_{z\in W(z_0)}$ be the quasimodes constructed in Proposition \ref{prop4.2}. 
Then, for $|z-w|\leq c$, with $c>0$ sufficiently small, 
\begin{equation}\label{eq4.19}
	(e_{-}^{j,hol}(z)| e_{-}^{j,hol}(w))
	=\e^{\frac{2}{h}\Psi^j_{-}(z,w;h)}
	+ \mO(h^{\infty})
	\e^{\frac{1}{h}\Phi_{-,0}^j(z)+\frac{1}{h}\Phi_{-,0}^j(w)},
\end{equation}
with 
\begin{equation}\label{eq4.18b}
	\Psi^j_{-}(z,w;h) = \Psi^j_{-,0}(z,w)+\frac{h}{2}\log\big( h^{\frac{1}{2}}b^j_-(z,w;h)\big), 
\end{equation}
Here $\Psi^j_{-,0}(z,w)$ is almost $z$-antiholomorphic 
and almost $w$-holomorphic at $\{(z_0,z_0)\}$, and $\Phi_{-,0}^j(z;h)$  is the function defined in \eqref{eq4.4.1}. The symbol $b_-^j(z,w;h) \sim b^{-,j}_0(z,w) + h b^{-,j}_1(z,w) + \dots $ 
is smooth in $z$ and $w$, with any derivative with respect to $z,\bar{z},w,\bar{w}$ 
uniformly bounded as $h\to 0$. Moreover, 
\begin{itemize}
	\item $\overline{\Psi_{-}^j(z,w;h)} = \Psi_{-}^j(w,z;h)$,
	\item $\Psi^j_{-,0}(z,z)= \Phi_{-,0}^j(z)$, 
	\item  $b_-^j(z,z;h)= A_-^j(z,h)^2$  
		with $A_-^j(z,h)$ as in Proposition \ref{prop4.2}, 
	\item for $|\zeta_i|\leq c$, $i=1,2$, $c>0$ small enough, and for any $N\in\N$,
	\begin{equation*}
	\Psi_{-,0}(z_0+\zeta_1,z_0+\zeta_2) = 
	\sum_{|\alpha|\leq N} ( \partial^{\alpha_1}_{\bar z} \partial^{\alpha_2}_{{z}}\Phi_{-,0})(z_0)
		\frac{\bar\zeta_1^{\alpha_1}\zeta_2^{\alpha_2}}{\alpha !} 
		+ \mO(|\zeta |^{N+1}).
\end{equation*}
\end{itemize}
\end{prop}
\begin{proof}[Proof of Proposition \ref{prop4.3}]
Until further notice we will suppress the superscript $j$. By Proposition 
\ref{prop4.2} and \eqref{eq4.4.1} we have that for $z,w \in W(z_0)$ 
\begin{equation}\label{eq4.n0}
	\begin{split}
	&(e_{+}^{hol}(z)|e_{+}^{hol}(w))
	= \e^{\frac{1}{h}\Phi_{+,0}(z)+\frac{1}{h}\Phi_{+,0}(w)}I(z,w), \\
	&
	I(z,w) = \int \chi_+(x)\,a_+(x,z;h)\,\overline{a_+(x,w;h)}\, \e^{\frac{i}{h}\Theta(x,z,w)} dx,
	\end{split}
\end{equation}
with 
\begin{equation}\label{eq4.n1}
	\begin{split}
	 \Theta (x,z,w) &= \varphi_+(x,z) - \overline{\varphi_+(x,w)} 
	  + i \Phi_{+,0}(z) + i \Phi_{+,0}(w)  \\
	  & \stackrel{\mathrm{def}}{=}  2\psi(x,z,w)  
	  + i \Phi_{+,0}(z) + i \Phi_{+,0}(w)  
	  \end{split}
\end{equation}
By \eqref{eq4.2.0}, \eqref{eq4.4.1} and the ensuing discussion, we have that 
for all $x\in\supp\chi_+$ and $z,w\in W(z_0)$
\begin{equation}\label{eq4.n1b}
	 \Ima  \Theta(x,z,w)  \geq 0,
\end{equation}
with equality on the submanifold  $\{(x_+(z),z,z)\}$. Here, we used \eqref{eq4.7} and the discussion 
afterwards, which implies that 
\begin{equation}\label{eq4.n2}
	\partial_x\psi(x_+(z),z,z) =0, \quad   \Ima \partial^2_{xx} \psi(x_+(z),z,z) > 0. 
\end{equation}
We have in fact $\Ima (\partial^2_{xx} \psi)(x,z,z)>0$ for all $x\in\supp \chi_+$. Hence, 
by the method of stationary phase for complex-valued phase functions \cite[Theorem 2.3]{MelSj74}, 
we see that for $|z-w|< c$, $c>0$ sufficiently small, the integral $I(z,w)$ can be written as
\begin{equation}\label{e:I(z,w)}
	I(z,w) = h^{\frac{1}{2}}b_+(z,w;h) 
	\e^{\frac{1}{h}(2\Psi_{+,0}(z,w) - \Phi_{+,0}(z) - \Phi_{+,0}(w) )} + \mO(h^\infty),
\end{equation}
where $b_+(z,w;h) \sim b_{+,0}(z,w) + h b_{+,1}(z,w) + \dots $ is smooth in $z$ and $w$,
and all derivatives with respect to $z,\bar{z},w,\bar{w}$ are bounded as $h\to 0$. 
To compute the "off-diagonal" phase function $\Psi_{+,0}(z,w)$, we extend the function $\psi(\cdot,z,w)$ to an almost holomorphic extension $\widetilde{\psi}(\cdot,z,w)$ defined in a small convex complex neighbourhood $\widetilde{X}$ of $\supp\chi_+$,
such that 
\begin{equation*}
	\partial^2_{xx} \widetilde{\psi}(x,z,w) \neq 0, \quad \forall x\in\widetilde{X},\ \forall z,w\in W(z_0)\,.
\end{equation*}
We can indeed find such a neighbourhood of $\supp\chi_+$, since $\partial^2_{xx} \psi(x,z,z)\neq 0$ for 
all $x\in\supp \chi_+$. 
From  \eqref{eq4.n1}, \eqref{eq4.2.0}, we see that $\widetilde{\psi}$ 
can be obtained by using an almost $x$-holomorphic extension $\widetilde{g}_0^+$ of $g_0^+$ in $\widetilde{X}$, and then defining $\widetilde{\varphi}_+(x,z)=-\int_{x_0}^x \tilde{g}_0(y,z)\,dy$ by taking as contour of integration the straight line connecting $x_0$ to $x\in \widetilde{X}$.  The exact relation $\partial_x\varphi_+(x,z)=-g_0^+(x,z)$ is then replaced by the following approximate relations: for any $\alpha,\beta\in\N$, 
\begin{equation}\label{e:approx-deriv}
\partial^\alpha_z\partial^\beta_{\bar z}\partial_x \widetilde{\varphi}_+(x,z) = - \partial^\alpha_z\partial^\beta_{\bar z}\widetilde{g}_0^+(x,z) + \mO_{\alpha,\beta}(|\Ima x|^\infty),\quad z\in W(z_0),\ x\in\widetilde{X}\,,
\end{equation}
with the implied constants uniform w.r.t. $z\in W(z_0)$.

The extension $\widetilde{\psi}$ can then be defined through
\begin{equation}\label{e:tilde-psi}
\widetilde{\psi}(x,z,w)=\widetilde{\varphi}_+(x,z)-\overline{\widetilde{\varphi}_+(\bar{x},w)}\,.
\end{equation}
Now, fixing $z,w\in W(z_0)$, the function $\widetilde{\psi}(\cdot,z,w)$ admits a unique, nondegenerate critical point $x=x^c(z,w)\in \widetilde{X}$, namely a point satisfying
\begin{equation}\label{eq4.n3}
	\partial_{x} \widetilde{\psi}(x^c(z,w),z,w) =0, \quad 
	\partial^2_{xx}\widetilde{\psi}(x^c(z,w),z,w) \neq 0.
\end{equation}
In particular, on the diagonal $(z,w)=(z,z)$ on finds that $x^c(z,z)=x(z)$ is real valued.
The phase function appearing in \eqref{e:I(z,w)} is now defined by
\begin{equation}\label{eq4.n5}
	\Psi_{+,0}(z,w) = i\widetilde{\psi}(x^c(z,w),z,w).
\end{equation}
Since there is an arbitrariness in the extension $\widetilde{\psi}$ of $\psi$, there is also an arbitrariness in the function $\Psi_{+,0}(z,w)$, especially in the regime $|z-w|\asymp 1$; however, in this regime the first term in \eqref{e:I(z,w)} is exponentially small, and therefore absorbed by the $\mO(h^\infty)$ remainder.
\par
Next, let us check that $x^c(z,w)$ is an almost $z$-holomorphic and almost 
$w$-antiholomorphic function at the diagonal $\{(z,z),\ z\in W(z_0)\}$. 
Differentiating the first equation in \eqref{eq4.n3} with respect to $z$ and $\bar{z}$, 
one obtains by \eqref{e:tilde-psi} that 
\begin{equation}\label{eq4.n4}
	\begin{split}
	&\partial_z x^c(z,w) = \frac{ -(\partial^2_{xz} \widetilde{\varphi}_+)(x^c(z,w),z) 
	+ \mO(|\Ima x^c(z,w) |^{\infty})}{2\partial^2_{xx}\widetilde{\psi}(x^c(z,w),z,w)},\\
	&\partial_{\bar{z}} x^c(z,w) = \frac{ - (\partial^2_{x\bar{z}}\widetilde{\varphi}_+)(x^c(z,w),z) 
	+ \mO(|\Ima x^c(z,w) |^{\infty})}{2\partial^2_{xx}\widetilde{\psi}(x^c(z,w),z,w)}, 
	\end{split}
\end{equation}
where we used that $\widetilde{\psi}$ is almost $x$-holomorphic. 
Similar expressions can be obtained for the $w$ and $\bar{w}$ derivatives of $x^c$. 
Lemma \ref{lem:AH} and \eqref{e:approx-deriv} imply that 
 \begin{equation*}
 \partial_{x\bar{z}} \widetilde{\varphi}_+(x^c(z,w),z) 
 = \mO(|x^c(z,w)-x_+(z)|^{\infty}).
 \end{equation*} 
A similar computation as for \eqref{eq4.n4} in the case of \eqref{eq4.n2} shows 
that $(\partial_z x^c)(z,z) = \partial_z x_+(z) = - (\partial_{\bar{w}}x^c)(z,z)$. 
Since $x^c$ depends smoothly on $z$ and $w$, and since $x^c(z,z) = x_+(z)$, it 
follows by Taylor expansion that $|x^c(z,w) -x_+(z)|= \mO(|z-w|) $ and by \eqref{eq4.n3} that 
\begin{equation*}
	\partial_{\bar{z}} x^c(z,w) = \mO(|z-w|^{\infty}) , \quad \partial_{w} x^c(z,w) = \mO(|z-w|^{\infty}). 
\end{equation*}
By \eqref{eq4.n5}, \eqref{e:tilde-psi} and \eqref{eq4.1a} in Lemma \ref{lem:AH}, we then see that 
$\Psi_{+,0}(z,w)$ is almost $z$-holomorphic at the point $z=w=z_0$:
 \begin{equation*}
	\begin{split}
	-i \partial_{\bar{z}} \Psi_{+,0}(z,w) &=
	(\partial_x\widetilde{\psi})(x^c(z,w),z,w)\partial_{\bar{z}}x^c(z,w) + \\
	&\quad + (\partial_{\bar{x}}{\widetilde{\psi}})(x^c(z,w),z,w)\overline{\partial_{z}x^c(z,w) } 
	+ \frac12\partial_{\bar{z}}\widetilde{\varphi}_+(x^c(z,w),z) \\ 
	& = \mO(|z-w|^{\infty}+|z-z_0|^{\infty}).
	\end{split}
\end{equation*}
\par
By \eqref{eq4.n0}, \eqref{eq4.2.1},  we have the normalization $I(z,z)=1$. Hence 
$\Psi_{+,0}(z,z)= \Phi_{+,0}(z)$ and $b_+(z,z;h)\sim A_+(z;h)^2$ 
in $\mathcal{C}^{\infty}(W(z_0))$, with 
$A_+(z;h)$ as in Proposition \ref{prop4.2}. Thus, by Taylor expansion 
at $(z_0,z_0)$ we get for any $N\in\N$ and $|\zeta| \ll 1$ 
\begin{equation*}
	\Psi_{+,0}(z_0+\zeta_1,z_0+\zeta_2) = 
	\sum_{|\alpha|\leq N} (\partial^{\alpha_1}_z\partial^{\alpha_2}_{\bar{z}}\Phi_{+,0})(z_0)
		\frac{\zeta_1^{\alpha_1}\overline{\zeta}_2^{\alpha_2}}{\alpha !} 
		+ \mO(|\zeta |^{N+1}).
\end{equation*}
In particular
\begin{equation*}
	2\Rea \Psi_{+,0}(z,w) - \Phi_{+,0}(z) - \Phi_{+,0}(w)  
	= -\partial^2_{z\bar z} \Phi_{+,0}(z_0) |z-w|^2
		+ \mO((|z-z_0|,|w-z_0|)^3).
\end{equation*}
Finally, \eqref{e:tilde-psi} and \eqref{eq4.n5} show that $x^c(z,w)=\overline{x^c(w,z)}$, and hence that $\Psi_{+,0}(w,z)=\overline{\Psi_{+,0}(z,w)}$, from which we draw that $\Psi_{+,0}(z,w)$ is almost 
$w$-antiholomorphic at the point $\{z=w=z_0\}$. 

Actually, setting
$\Psi_{+}^j(z,w;h) = \Psi_{+,0}^j(z,w) + \frac{h}{2}\log \big( h^{\frac{1}{2}}b_+^j(z,w;h)\big)$,  
we directly see by \eqref{eq4.n0} and \eqref{e:I(z,w)} that 
$\overline{\Psi_{+}^j(z,w;h)} = \Psi_{+}^j(w,z;h)$. \qedhere
\end{proof}
The proof of Proposition \ref{prop4.3b} is totally symmetric to the one of  Proposition \ref{prop4.3} we have just given.

\subsection{Symmetric symbols}\label{s:symmetric-symbols}
Assume \eqref{eq.15}, the additional symmetry of the symbol 
$p(x,\xi;h)$, we will also need to compute the interaction between the squares of the quasimodes 
$e_-^j$, when considering perturbations by a random potential. The construction of these quasimodes was discussed in Section 
\ref{sec:QuasModSym}.

\par
Notice that by \eqref{eq4.3.3}, \eqref{eq4.3.7} and the assumption \eqref{eq4.17.1b}, we have for 
all $z,w\in W(z_0)$ 
\begin{equation}\label{eq4.21.0}
	((e_{-}^j(z))^2|(e_{-}^k(w))^2) =  0, \quad  \text{for } j\neq k.
\end{equation}
Similarly as in the proof of Proposition \ref{prop4.2}, one obtains by the 
method of stationary phase that 
\begin{equation}\label{eq4.21.1}
\begin{split}
	&\|(e_{-}^{j,hol}(z))^2\| = \e^{\frac{1}{h}\Phi^j_s(z;h)}, \\
	&\Phi^j_s(z;h)\defeq \Phi^j_{s,0}(z)  + h\log \big( h^{\frac{1}{4}}A^j_s(z;h)\big), 
	\quad \Phi^j_{s,0}(z)\defeq 2 \Phi_{+,0}^j(z)
\end{split}
\end{equation}
where $\Phi_{+,0}^j(z)$ is as in \eqref{eq4.4.1} and, in view of \eqref{eq4.3.3}, 
$x_+^j = x_-^j=x^j$. 
We remind that  $\Phi_{+,0}^j(z)\geq 0 $, with equality if and only if $z=z_0$, 
and $A^j_s(z;h)\sim A_0^{j,s}(z)+hA_1^{j,s}(z)+\dots$ depends smoothly on $z$ 
such that all derivatives with respect to $z$ and $\bar{z}$ are bounded when $h\to 0$. 
Moreover, 
\begin{equation}\label{eq4.21.1b}
	A_0^{j,s}(z)=
	 \left(\frac{\pi |a_0^+(x_+(z);z)|^4}{2\Ima \partial_{xx}^2\varphi_+(x_+(z),z)}\right)^{1/4}>0,
\end{equation}
where $a_0^+$ is as in \eqref{eq4.9.5}.
\begin{prop}\label{prop4.4}
Suppose that \eqref{eq.15} holds and that $W(z_0)\Subset\C$, as in Proposition \ref{prop4.2} 
satisfies \eqref{eq4.17.1}, \eqref{eq4.17.1b}. Let $\Phi^j_s(z;h)$, $z\in W(z_0)$, be as in 
\eqref{eq4.21.1}. Then, for $|z-w|\leq c$, with $c>0$ sufficiently small, 
\begin{equation}\label{eq4.22}
	((e_{-}^{j,hol}(z))^2|( e_{-}^{j,hol}(w))^2)
	=\e^{\frac{2}{h}\Psi^j_s(z,w;h)}+ \mO(h^{\infty})
	\e^{\frac{1}{h}\Phi^j_{s,0}(z)+\frac{1}{h}\Phi^j_{s,0}(w)},
\end{equation}
with 
\begin{equation}\label{eq4.22a}
	\Psi^j_{s}(z,w;h) = \Psi^j_{s,0}(z,w)+\frac{h}{2}\log\big( h^{\frac{1}{2}} b^j_s(z,w;h)\big), 
\end{equation}
where $\Psi^j_{s,0}(z,w)$ is almost $z$-antiholomorphic  and almost $w$-holomorphic
at $\{(z_0,z_0)\}$ and $b_s^j(z,w;h) \sim b^{s,j}_0(z,w) + h b^{s,j}_1(z,w) + \dots $ 
is smooth in $z$ and $w$ and all derivatives with respect to $z,\bar{z},w,\bar{w}$ 
are bounded as $h\to 0$. Moreover, 
\begin{itemize}
	\item $\overline{\Psi_{s}^j(z,w;h)} = \Psi_{s}^j(w,z;h)$, 
	\item $\Psi^j_{s,0}(z,z)= \Phi^j_{s,0}(z)=2\Phi_{+,0}^j(z)$,
	\item $b_s^j(z,z;h)\sim A_s^j(z,h)^2$ with $A_s^j(z,h)$ as in \eqref{eq4.21.1}, 
	\item  for $|\zeta_i|\leq c$, $i=1,2$, $c>0$ small enough, and for any $N\in\N$,
\begin{equation*}
	\Psi_{s,0}^j(z_0+\zeta_1,z_0+\zeta_2) = 
	2\sum_{|\alpha|\leq N} (\partial^{\alpha_1}_{\bar z}\partial^{\alpha_2}_{z}\Phi_{+,0})(z_0)
		\frac{\overline{\zeta}_1^{\alpha_1}\zeta_2^{\alpha_2}}{\alpha !} 
		+ \mO(|\zeta |^{N+1}).
\end{equation*}
\end{itemize}
\end{prop}
\begin{proof}
	Similar to the proof of Proposition \ref{prop4.3}.
\end{proof}
\subsection{Finite rank truncation of the quasimodes}
In this section we show that the quasimodes $\{e_{\pm}^j(z)\}$ essentially live in a finite-dimensional subspace of $L^2(\R)$, which we build using the the orthonormal eigenbasis $\{e_m\}_{m\in\N}$ 
of the harmonic oscillator $H=-\partial_x^2+x^2$, corresponding to the eigenvalues $\{\lambda_{m}=2m+1\}$. 

For any $N\in \N$, let us call $\Pi_N$ the orthogonal projector on the subspace of $L^2$ spanned by the states $\{e_m\}_{0\leq m\leq N}$. In the following Lemma we show that if $N$ is chosen large enough, this projection does essentially not modify our quasimodes. 
\begin{lem}\label{lem7.1}
Let $W(z_0)\Subset\C$ be as in Proposition \ref{prop4.2}, and satisfying \eqref{eq4.17.1b}. 
Let $\{e_{\pm}^{j}(z),z\in W(z_0)\}$ be the normalized quasimodes constructed in Proposition \ref{prop4.2}. Then, if $C_1>0$ is chosen sufficiently large, taking $N(h)=C_1/h^2$, we have
\begin{equation}\label{e:Pi_N}
	\forall z\in W(z_0),\ \forall j=1,\ldots,J,\quad	\| (1-\Pi_{N(h)})\,e_{\pm}^{j}(z)\| =   \mO(h^{\infty}) .
\end{equation}
In the case of a symmetric symbol \eqref{eq.15}, we will also need the estimate
\begin{equation}
		\forall z\in W(z_0),\ \forall j=1,\ldots,J,\quad	\| (1-\Pi_{N(h)})\,(e_{\pm}^{j}(w))^2 \| = \mO(h^{\infty}) .
\end{equation}
In both cases the $\mO(h^\infty)$ remainder is uniform w.r.t. $z\in W(z_0)$ and $h\in (0,1]$.
\end{lem}
\begin{proof}
We will focus on proving the first estimate, the estimate for squared quasimodes being very similar. 
\par
Let $U\Subset\R^2$ be a bounded open set, so that 
	 \begin{equation}\label{eq4.n13.5}
	 	\bigcup_{j=1}^J \rho_{\pm}^j\left(\overline{W(z_0)}\right) \subset U.
	 \end{equation}
Our proof will amount to show that the projector $\Pi_{N(h)}$ is microlocally equal to the identity in a neighbourhood of $U$.

Let $\psi\in\mathcal{C}^{\infty}_c(\R^2,[0,1])$ such that $\psi \equiv 1$ in a small 
neighbourhood of $\overline{U}$ and $\supp \psi \subset B(0,R)\subset \R^2$ the ball of radius $R$ centered at $0$, for some sufficiently large $R>0$ . Then, by the microlocalization property \eqref{eq4.6}, we have
\begin{equation}\label{eq_OL7}
	\|(1 - \psi^w)e_{\pm}^j(z)\| =\mO(h^{\infty}),\quad \text{uniformly in $z\in W(z_0)$.}
\end{equation}
It suffices to prove that there exists a constant $C_1>0$ such that
\begin{equation}\label{eq_OL3}
\forall z\in W(z_0), \forall m>N(h)=C_1/h^2,\qquad	\|\psi^w(x,hD_x) e_m\|=\mO( \lambda_m^{-\infty})\,,
\end{equation}
where the implied constant is uniform w.r.t. $m>N(h)$, $h\in (0,1]$ and $z\in W(z_0)$. 
Indeed, the above bounds imply that 
$$
\|(1-\Pi_{N(h)})\psi^w e_{\pm}^{i}(z)\|^2 = \sum_{m>N(h)} |(\psi^w e_{\pm}^{i}(z)| e_m)|^2=\mO(h^\infty)\,.
$$
Together with  \eqref{eq_OL7}, this shows the requested estimate \eqref{e:Pi_N}. 

\medskip

Let us then prove \eqref{eq_OL3}. We recall the definition of "exotic" symbol classes.
Let $\widetilde{m}$ be an order function as in \eqref{eq1.4}. For any  and let $0\leq\delta\leq 1/2$, we define a symbol class, generalizing the classes in \eqref{eq1.5}:
\begin{equation}\label{eq_OL1}
	S_{\delta}(\widetilde{m}) = S_{\delta}(\R^2,\widetilde{m})\stackrel{\mathrm{def}}{=} \{a\in\mathcal{C}^{\infty}(\R^2); ~
	|\partial^{\alpha}_{\rho}a(\rho)| \leq C_{\alpha} h^{-\delta|\alpha|} \widetilde{m}(\rho), 
	~~ \forall \rho \in\R^2 \}.
\end{equation}
For more details on the semiclassical calculus in such a class of symbols we refer the 
reader to \cite[Chapter 4]{Zw12}. 
\par
For $\lambda>0$, let $U_{\lambda}$ be the unitary transformation on $L^2(\R)$ given by 
$U_{\lambda} u(x) = \lambda^{1/4}u(\lambda^{1/2}x)$. The harmonic oscillator can be written $H=q^w(x,D_x)$, where $q(x,\xi) = \xi^2+x^2$. A direct computation shows that
\begin{equation}\label{e:HO}
	U_{\lambda_m} (H - \lambda_m )U_{\lambda_m}^{-1} =\lambda_m ( q^w(x,\lambda_m^{-1} D_x) - 1)\,,\quad \tilde h=\lambda_m^{-1}\,.
\end{equation}
Let us insert the dilation $U_{\lambda_m}$ in the expression we want to estimate:
$$
\|\psi^w(x,hD_x)  e_m = \| U_{\lambda_m}\psi^w(x,hD_x) U_{\lambda_m}^{-1}\, \tilde e_m) \|,\quad \tilde e_m = U_{\lambda_m} e_m.
$$
We notice that 
\begin{equation}\label{e:tilde-e_m}
(q^w(x,\tilde h D_x) - 1)\tilde e_m = 0\,.
\end{equation}
We want to write $U_{\lambda_m}\psi^w(x,hD_x)U_{\lambda_m}^{-1}$ as a 
pseudodifferential operator for the $\tilde h$-calculus. 
An easy computation shows that
$$
U_{\lambda_m}\psi^w(x,hD_x) U_{\lambda_m}^{-1} = \tilde\chi^w(x, \tilde h D_x),\quad \text{for the symbol }
\tilde\chi(x,\xi)=\psi(\lambda_m^{1/2}x,\lambda_m^{1/2}h\xi)\,.
$$
The symbol $\tilde\chi$ belongs to the exotic class $\widetilde{S}_{1/2}(1)$, the analogue of $S_{1/2}(1)$ (see \eqref{eq_OL1}) for the semiclassical parameter $\tilde h$. Our task is thus to show that the norm
\begin{equation}\label{e:element2}
\| \psi^w(x,hD_x)  e_m \|= \| \tilde\chi^w(x,\tilde h D_x)  \tilde e_m \|
\end{equation}
is very small when $\tilde h\to 0$.
Since we are in the parameter range $m>N(h)$, we have $\lambda_m^{1/2}> \sqrt{2C_1} h^{-1}$, hence $\lambda_m^{1/2}h>\sqrt{2C_1} $. As a result, since $\psi$ was assumed to be supported in some compact neighbourhood of $\bar{U}$, by taking $C_1$ large enough we can ensure that $\tilde\chi$ is supported inside the ball $B(0,1/3)\subset \R^2$. 

On the other hand, \eqref{e:tilde-e_m} implies that the  $\tilde h$-wavefront set of $\tilde e_m$ is contained in the unit circle $q^{-1}(1)$. As we show below\footnote{This consequence would be standard if $\tilde\chi\in S(1)$; here we will carefully take into account the fact that $\tilde\chi\in \tilde S_{1/2}(1)$.}, these two facts imply that $\tilde\chi(x,\tilde hD_x) \tilde e_m=\mO_{L^2}(\tilde h^\infty)$.
\par
Let $\eta \in \mathcal{C}^{\infty}_0(\R,[0,1])$ such that $\eta \equiv 1$ on 
$q(B(0,1/2))=[0,1/4]$, and $\supp \eta \subset [-2/3,2/3]$. By the functional calculus for 
semiclassical pseudodifferential operators \cite[Chapter 8]{DiSj99},
\begin{equation}\label{eq4.n10}
	\eta(q^w(x,\widetilde{h}D_x)) = g^w(x,\widetilde{h}D_x),\qquad \text{for a symbol }
	g \in S(\langle \rho \rangle^{-\infty})).
 \end{equation}
Here $S(\langle \rho \rangle^{-\infty})$ is the intersection of all symbol classes $S(\langle \rho \rangle^{-n})$, $n\geq 0$.
Moreover, for any $\chi_1\in C^\infty_c(B(0,1/2))$, we have
\begin{equation}\label{eq_OL6}
	g\chi_1 = 1 +b,\quad\text{for some }b \in\widetilde{h}^\infty  S(\langle \rho \rangle^{-\infty}).
\end{equation}
Next, we recall from \cite[Chapter 7]{DiSj99} a basic fact about the composition of two pseudo 
differential operators: Let $a_j\in \tilde S_{\delta}(\R^2,\widetilde{m}_j)$, for $j=1,2$, then 
 \begin{equation*}
 	a_1^w(x,\widetilde{h}D_x)\circ a_2^w(x,\widetilde{h}D_x)
	=(a_1\# a_2)^w(x,\widetilde{h}D_x),
 \end{equation*}
where the Moyal product $\#: \tilde S_{\delta}(\widetilde{m}_1)\times \tilde S_{\delta}(\widetilde{m}_2) 
\to \tilde S_{\delta}(\widetilde{m}_1\widetilde{m}_2)$ is  bilinear and continuous. 
\par
Let us choose the cutoff $\chi_1\in\mathcal{C}_c^{\infty}(B(0,1/2),[0,1])$ such that 
$\chi\equiv 1$ on $B(0,5/12)$, and set $\chi_2 =1 -\chi_1$. 
By \eqref{eq_OL6}, 
 \begin{equation*}
 	\begin{split}
 	\widetilde{\chi} \# g&=  \widetilde{\chi} \# g\chi_1+ \widetilde{\chi}\# g\chi_2 \\ 
	&=  \widetilde{\chi} \#\chi_1+ \widetilde{\chi} \# b+ \widetilde{\chi} \# g\chi_2\\ 
	&=  \widetilde{\chi} - \widetilde{\chi}\#  \chi_2+ \widetilde{\chi}\#b + \widetilde{\chi}\# g\chi_2.
	\end{split}
 \end{equation*}
Since $\dist(\supp\chi_2,\supp \widetilde{\chi}) \geq 1/12$, it follows \cite[Theorem 4.25]{Zw12} that 
 \begin{equation*}
 	(\widetilde{\chi}\#g\chi_2)^w(x,\widetilde{h}D_x)= 
	\mO_{L^2\to L^2}(\widetilde{h}^{\infty}),\quad 
	(\widetilde{\chi}\# \chi_2)^w(x,\widetilde{h}D_x) = 
	\mO_{L^2\to L^2}(\widetilde{h}^{\infty}).
 \end{equation*}
The bilinearity of the Moyal product shows that $\widetilde{\chi}\# b \in \widetilde{h}^\infty \tilde S_{1/2}(\langle \rho \rangle^{-\infty})$, and from the Calder\'on-Vaillancourt theorem that
 $( \widetilde{\chi} \# b)^w(x,\widetilde{h}D_x)
	= \mO_{L^2\to L^2}(\widetilde{h}^{\infty}).
$
We have obtained
 \begin{equation*}
 	\widetilde{\chi}^w(x,\widetilde{h}D_x) = 
	\widetilde{\chi}^w(x,\widetilde{h}D_x)\# \eta(q^w(x,\widetilde{h}D_x))  + \mO_{L^2\to L^2}(\widetilde{h}^{\infty}).
 \end{equation*}
Inserting this identify in \eqref{e:element2} and using the fact that $\tilde e_m$ is a normalized null eigenstate of $\eta(q^w(x,\tilde h D_x))$ (since $\eta(1)=0$), we get the required estimate:
\begin{equation*}
	\|\psi^w e_m\|= \|\widetilde{\chi}^w(x,\widetilde{h}D_x) \eta(q^w(x,\widetilde{h}D_x)) \tilde e_m\| + \mO(\widetilde{h}^{\infty})= \mO(\widetilde{h}^{\infty}) = \mO(\lambda_m^{-\infty}). \qedhere
\end{equation*}
\end{proof}
The following result is not of immediate importance, but will be relevant later on in 
the text. 
\begin{lem}\label{lem7.2}
Let $e_{\pm}^{j,hol}(z)$ be the quasimodes constructed in Proposition \ref{prop4.2}. 
Take as before $\{e_k\}_{k\in\N}$ the orthonormal eigenbasis of the harmonic oscillator $H$, associated with the eigenvalues $\{\lambda_k=2k+1\}$. 
Then, 
\begin{equation}\label{e:bound-overlap}
\forall m\geq 0,\qquad	|(h^{-1/4}e_{\pm}^{j,hol}(z)| e_m)|
	=\mO(h^{1/4}) \,\lambda_m^{-1/6} \,\e^{\frac{1}{h}\Phi_{\pm,0}^j(z)}\,,
\end{equation}
where $\Phi_{\pm,0}^j(z)$ is given in \eqref{eq4.4.1}.
Similarly, the squared quasimodes satisfy
\begin{equation*}
\forall m\geq 0,\qquad	|(h^{-1/4}(e_{-}^{j,hol}(z))^2| e_m)|
	=\mO(h^{1/4})\, \lambda_m^{-1/6} \,\e^{\frac{2}{h}\Phi_{+,0}^j(z)}.
\end{equation*}
In both estimates the implied constant is uniform w.r.to $m\in\N$, $h\in(0,1]$ and $z\in W(z_0)$.
\end{lem}
\begin{proof}
 We only show the case of the nonsquared quasimodes, the other case being similar. 
 By the H\"older inequality,
 \begin{equation*}
	|(h^{-1/4}e_{\pm}^{j,hol}(z)| e_m)| 
	\leq \| (h^{-1/4}e_{\pm}^{j,hol}(z)\|_{L^1} \, \|e_m\|_{L^\infty}.
\end{equation*}
Using \eqref{eq4.11}, the method of stationary phase yields
 \begin{equation*}
	\| h^{-1/4}e_{\pm}^{j,hol}(z)\|_{L^1}
	=\mO(h^{\frac{1}{4}})\, \e^{\frac{1}{h}\Phi_{+,0}^j(z)},
\end{equation*}
where $\Phi_{+,0}^j(z)$ is as in \eqref{eq4.4.1}. By \cite[Corollary 3.2]{KoTa05}, 
there exists a constant $C>0$, such that for all $m\in\N$,
 \begin{equation}\label{eqKoTa}
 	\| e_m\|_{\infty} \leq C\lambda_m^{-1/6}\|e_m\|_{L^2}. 
 \end{equation}
\end{proof}
\begin{rem}
The choice of the orthonormal basis $(e_m)_{m\geq 0}$ used to define the truncation operator is rather arbitrary, as explained in \cite{Ha06b,HaSj08}. What is needed is the fact that for $N=N(h)$ large enough, the projection $\Pi_{N(h)}$ on the subspace spanned by a collection of $N(h)$ of those states is microlocally equivalent to the identity in some given bounded region $U\subset \R^2$. Any orthonormal basis of $L^2(\R)$ will have this property, but the number $N(h)$ of necessary state will depend on the choice of basis.

In \cite{Ha06b} the author used the eigenbasis of the semiclassical harmoric oscillator $q^w(x,hD_x)$, in which case it was sufficient to include only $\mO(h^{-1})$ eigenstates. Our choice to use the nonsemiclassical harmonic oscillator $H$ requires to include a larger number $N(h)=\mO(h^{-2})$ of states. This choice was guided by the extra requirement that each quasimode $e^j_{\pm}(z)$ should decomposes into {\it many} basis states $e_m$ (as shown by Lemma~\ref{lem7.2}), a fact which will be important in Section~\ref{s:weak_convergence} when applying the Central Limit Theorem.
\end{rem}

\section{Grushin Problem}\label{sec:GP}
We begin by giving a short refresher on Grushin problems. They 
have become an important tool in microlocal analysis and are 
employed in a vast number of works, especially when dealing with spectral studies of nonselfadjoint operators. 
As reviewed in \cite{SjZw07}, the central idea 
is to analyze the operator $P(z)=(P-z):\cH_1\to \cH_2$ by extending this operator into a larger operator of the form
\begin{equation*}
 \begin{pmatrix}
  P(z) & R_- \\ 
  R_+ & 0 \\
 \end{pmatrix}
 :
 \mathcal{H}_1\oplus \mathcal{H}_- 
 \longrightarrow \mathcal{H}_2\oplus \mathcal{H}_+,
\end{equation*}
where $\cH_{\pm}$ (resp. $R_{\pm}$) are well-chosen auxiliary spaces (resp. operators). The Grushin problem is 
said to be {\it well-posed}
if this matrix of operators is bijective for the range of $z$ under study, 
with a good control on its inverse. In the cases where
$\dim\mathcal{H}_-  = \dim\mathcal{H}_+ < \infty$, on writes this inverse blockwise as
\begin{equation*}
 \begin{pmatrix}
  P(z) & R_- \\ 
  R_+ & 0 \\
 \end{pmatrix}^{-1}
 =
 \begin{pmatrix}
  E(z) & E_{+}(z) \\ 
  E_{-}(z) & E_{-+}(z) \\
 \end{pmatrix}.
\end{equation*}
The key observation, going back to Schur's complement formula or, 
equivalently, the Lyapunov-Schmidt bifurcation method, is the following: 
the operator $P(z): \mathcal{H}_1 \rightarrow \mathcal{H}_2$ 
is invertible if and only if the finite rank operator
$E_{-+}(z):\cH_{+}\to \cH_-$ is invertible, in which case both inverses are related by:
\begin{equation*}
  P(z)^{-1}= E(z) - E_{+}(z) E_{-+}^{-1}(z) E_{-}(z).
\end{equation*}
The finite rank operator $E_{-+}(z)$ is often called an {\it effective Hamiltonian} for the original problem $P(z)$. As opposed to $P(z)$, it depends in a nonlinear way of the variable $z$, but it has the advantage to be finite dimensional. In a sense, $E_{-+}(z)$ encapsulates, in a compact way, the spectral properties of $P$.

\subsection{Grushin problem for our unperturbed nonselfadjoint operator}
Hager \cite{Ha06b} showed how to construct an efficient Grushin problem for our nonselfadjoint operator $P_h$, using the quasimodes constructed in Proposition \ref{prop4.2}.
\begin{prop}[Unperturbed Grushin]\label{prop5.1}
Let $p(\cdot;h)$ in $S(\R^2,m)$ be as in \eqref{eq1.6} and satisfy \eqref{eq1.7}, 
\eqref{eq1.8.1}. Let $W(z_0)\Subset \C$ and $\{e_{\pm}^j(z),z\in W(z_0)\}$, $j=1,\dots,J$, be as
in Proposition~\ref{prop4.2}, with $W(z_0)$ satisfying \eqref{eq4.17.1}.

For $z\in W(z_0)$ let 
\begin{equation*}
 \mathcal{P}(z)=\begin{pmatrix}
  P_h-z & R_-(z) \\ 
  R_+(z) & 0 \\
 \end{pmatrix}
 : H(m)\times \C^J \longrightarrow L^2 \times \C^J\,,
\end{equation*}
with 
\begin{equation}\label{eq5.2}
 (R_+(z)u)_j\defeq (u|e_+^j(z)), \ \  u\in H(m), 
 \qquad 
 R_-(z) u_- \defeq \sum_{j=1}^J u_-^je_-^j(z), ~~ u_-\in\C^J.
\end{equation}
Then, $\mathcal{P}(z)$ is bijective,  with bounded inverse: %
\begin{equation*}
 \mathcal{E}(z)=\begin{pmatrix}
  E(z) & E_{+}(z) \\ 
  E_{-}(z) & E_{-+}(z) \\
 \end{pmatrix}
 : L^2\times \C^J \longrightarrow H(m) \times \C^J.
\end{equation*}
The components $E_{\pm}(z)$ are given by
\begin{equation}\label{eq5.3}
 (E_-(z)v)_j=(v|e_-^j(z)) + \mO(h^{\infty}), ~~ v\in L^2, 
 \qquad 
 E_+(z)v_+ = \sum_{j=1}^J v_+^je_+^j(z) + \mO(h^{\infty}), ~~ v_+\in\C^J\,.
\end{equation}
Finally, the blocks of $\mathcal{E}(z)$ admit the following bounds in the semiclassical limit:
\begin{equation*}
\begin{split}
\|E(z)\|_{L^2\to H(m)}=\mO(h^{-1/2}),&\quad \|E_{+}(z)\|_{\C^J\to L^2}=\mO(1),\\
\|E(z)\|_{L^2\to\C^J}=\mO(1),&\quad \|E_{-+}(z)\|_{\C^J\to \C^J}=\mO(h^{\infty}),
\end{split}
\end{equation*}
uniformly for $z\in W(z_0)$. 
\end{prop}
\begin{proof}
One can follow line by line, with the obvious changes, the proof of \cite[Proposition 4.1]{Ha06b}.
\end{proof}
\subsection{Grushin problem for the perturbed operator}
We wish to study the eigenvalues of 
\begin{equation}\label{eq5.2.0}
	P^{\delta} = P_h + \delta Q,
\end{equation}
where $\delta>0$ satisfies \eqref{e:delta} and $Q$ is given by a random 
matrix $M_{\omega}$, as in \eqref{eq1.11}, or by a random potential 
$V_{\omega}$,  as in \eqref{eq1.14}. Recall from \eqref{eq1.15.2} that we have restricted 
the random variables used to construct $M_{\omega}$ and $V_{\omega}$,  
to large discs of radius $C/h$. By the tail estimate \eqref{eq1.9.1} and by the fact that 
$N(h)=C_1/h^2$, this restriction holds with probability close to one, more precisely
there exists a constant $C_2>0$ such that 
\begin{equation}\label{eq5.2.1.1}
	\begin{split}
	&\prob\big[|q_{i,j}| \leq C/h, ~\forall  i,j <N(h)\big] \geq 1 - \kappa N(h)^2 h^{4+\varepsilon_0} = 1 - C_2 h^{\varepsilon_0}, \\
	&\prob\big[|v_{i}| \leq C/h, ~\forall  i<N(h)\big] \geq 1 - \kappa N(h) h^{4+\varepsilon_0} = 1- C_2 h^{2+\varepsilon_0}.
	\end{split}
\end{equation}
Thus, by \eqref{eq5.2.1.1}, we have with probability 
$\geq 1 - C_2 h^{\varepsilon_0}$, that 
\begin{equation}\label{eq5.2.1.5}
	\|M_{\omega}\|_{HS} = N(h)^{-1}\Big(\sum_{i,j <N(h)}|q_{ i,j}|^2\Big)^{1/2} \leq C h^{-1}. 
	\end{equation}
Similarly, using \eqref{eq5.2.1.1}, \eqref{eqKoTa}, we have with probability 
$\geq 1 - C_2 h^{2+\varepsilon_0}$, for $h$ small enough,
\begin{equation}\label{eq5.2.1.6}
	\|V_{\omega}\|_{\infty} \leq N(h)^{-1}\sum_{n<N(h)} |v_n| \|e_n\|_{\infty}
		\leq Ch^{-1}.
\end{equation}
This proves the estimates \eqref{eq1.11.1} and \eqref{eq1.15.1} on the size of the perturbations.
\par
Until further notice we will work in the restricted probability space, where 
all $|q_{i,j}|\leq C/h$ respectively $|v_{j}|\leq C/h$. Then, \eqref{eq5.2.1.5} 
respectively \eqref{eq5.2.1.6} holds.
Using Proposition \ref{prop5.1}, we obtain a well-posed Grushin 
Problem for the perturbed operator $P^{\delta}$.
\begin{prop}[Pertubed Grushin]\label{prop5.2.1}
Let $p(\cdot;h)$ in $S(\R^2,m)$ be as in \eqref{eq1.6} and satisfy \eqref{eq1.7}, 
\eqref{eq1.8.1}. Let $P^{\delta}$ be as in \eqref{eq5.2.0} with $\delta>0$ satisfying 
\eqref{e:delta}. 
For $z\in W(z_0)$ let 
\begin{equation*}
 \mathcal{P}^{\delta}(z)
 =\begin{pmatrix}
  P^{\delta}-z & R_-(z) \\ 
  R_+(z) & 0 \\
 \end{pmatrix}
 : H(m)\times \C^J \longrightarrow L^2 \times \C^J
\end{equation*}
with 
\begin{equation}\label{eq5.2.1}
 (R_+(z)u)_j\defeq (u|e_+^j(z)), ~~ u\in H(m), 
 \qquad 
 R_-(z)u_- \defeq \sum_{j=1}^J u_-^je_-^j(z), ~~ u_-\in\C^J.
\end{equation}
Then, $P^{\delta}(z)$ is bijective with bounded inverse %
\begin{equation*}
 \mathcal{E}^{\delta}(z)=\begin{pmatrix}
  E^{\delta}(z) & E_{+}^{\delta}(z) \\ 
  E_{-}^{\delta}(z) & E_{-+}^{\delta}(z) \\
 \end{pmatrix}
 : L^2\times \C^J \longrightarrow H(m) \times \C^J,
\end{equation*}
such that $\mathcal{P}^{\delta} \mathcal{E}^{\delta} = 1_{L^2\times \C^J}$ and 
$ \mathcal{E}^{\delta} \mathcal{P}^{\delta} = 1_{H(m)\times \C^J}$. 

Moreover, the perturbed blocks  $E^\delta_{\pm}$, $E^\delta_{-+}$ are related as follows with the unperturbed blocks of Proposition \ref{prop5.1}:
\begin{equation}\label{eq5.2.2}
 E^{\delta}_- = E_- +\mO(\delta h^{-3/2}), \quad E^{\delta}_+ = E_+ +\mO(\delta h^{-3/2}),
\end{equation}
and 
\begin{equation}\label{eq5.2.3}
 E^{\delta}_{-+} = E_{-+} - \delta E_- Q E_+ + \mO(\delta^2 h^{-5/2}).
\end{equation}
\end{prop}
%
\begin{proof}
The result follows from an application of the Neumann series. Let $\mathcal{E}$ 
be as in Proposition \ref{prop5.1}, then 
\begin{equation*}
	\mathcal{P}^{\delta}\mathcal{E} = \mathcal{P}\mathcal{E} + 
	\begin{pmatrix}
	\delta Q & 0 \\ 
	0 & 0 \\
	\end{pmatrix}
	\mathcal{E}
	= 1 + 
	\begin{pmatrix}
	\delta QE & \delta QE_+ \\ 
	0 & 0 \\
	\end{pmatrix}
	= 1+ K.
\end{equation*}
By Proposition \ref{prop5.1} and \eqref{eq5.2.1.5}, \eqref{e:delta}, 
$\|K\| \leq \delta\|Q\| \|E\| = \mO(\delta h^{-3/2}) \ll1$. Thus, $\mathcal{P}^{\delta}$ 
has the inverse $\mathcal{E}(1+K)^{-1}$ and, by a Neumann series expansion, 
we get 
\begin{equation}\label{eq5.2.3.1}
	\begin{split}
	\mathcal{E}^{\delta}
	&=\begin{pmatrix}
 	 E^{\delta}(z) & E_{+}^{\delta}(z) \\ 
 	 E_{-}^{\delta}(z) & E_{-+}^{\delta}(z) \\
 	\end{pmatrix}\\
	&=
	\begin{pmatrix}
	\sum_{n=0}^{\infty}(-1)^n E(\delta QE)^n& \sum_{n=0}^{\infty}(-1)^n (\delta QE)^nE_+ \\ 
	& \\
	\sum_{n=0}^{\infty}(-1)^n E_-(\delta QE)^n & 
	E_{-+}+\delta\sum_{n=1}^{\infty}(-1)^n E_-(\delta QE)^{n-1}QE_+ \\
	\end{pmatrix}.
	\end{split}
\end{equation}
This, yields the claimed estimates and concludes the proof. 
\end{proof}
Using Propositions \ref{prop4.2}, \ref{prop5.1} and \eqref{eq5.2.3}, we get 
\begin{equation}\label{eq5.2.4}
	(E_{-+}^{\delta}(z))_{i,j}= 
	-\delta (Qe_+^j |e_-^i ) +  \mO(\delta h^{\infty}) + \mO(h^{\infty})
	+\mO(\delta^2 h^{-5/2}).
\end{equation}
\begin{rem}We will show below that, with high probability, the first term in the right hand side of \eqref{eq5.2.4} dominates the following ones: the matrix $E_{-+}^\delta$ is thus dominated by the perturbation term, eventhough the latter has a small norm. 
This effect comes from the fact that unperturbed matrix $E_{-+}$
is extremely small, a consequence of the strong pseudospectral effect. 
\end{rem}%

Using the assumption \eqref{e:delta} on the size of $\delta$ and taking the determinant, we get
\begin{equation}\label{eq5.2.4.1}
	\det\left[  \delta^{-1} E_{-+}^{\delta}(z)\right] =
	(-1)^J \det \left[ (Qe_+^j(z) |e_-^i(z) )_{i,j\leq J} 
	+\mO(\delta h^{-5/2})\right].
\end{equation}
Notice that $\det E_{-+}^{\delta}(z)$ depends smoothly on $z\in W(z_0)$ since we 
have used the normalised quasimodes. However, we can make it 
holomorphic in $z$ since it satisfies a $\overline{\partial}$-equation in $z$. 
To see this, we take the derivative with respect to $\bar{z}$ of 
$ \mathcal{P}^{\delta} \mathcal{E}^{\delta} = 1_{L^2\times \C^J}$, and we get 
\begin{equation*}
     \partial_{\bar{z}}\mathcal{E}^{\delta} = - 
     \mathcal{E}^{\delta}\partial_{\bar{z}}\mathcal{P}^{\delta}\mathcal{E}^{\delta}.
\end{equation*}
Hence, 
\begin{equation*}
     \partial_{\bar{z}}{E}_{-+}^{\delta} = 
     - {E}_{-+}^{\delta}(\partial_{\bar{z}}R_+ )E_+^{\delta} 
     -E_-^{\delta}(\partial_{\bar{z}}R_-){E}_{-+}^{\delta}.
\end{equation*}
This, together with the identity $ \partial_{\bar{z}}\log\det{E}_{-+}^{\delta} = 
\tr({E}_{-+}^{\delta} )^{-1} \partial_{\bar{z}}{E}_{-+}^{\delta} $ then yields 
\begin{equation}\label{eq5.2.5}
	\begin{split}
     \partial_{\bar{z}} \det{E}_{-+}^{\delta} &= 
     -\tr [
     (\partial_{\bar{z}}R_+ )E_+^{\delta} +
     E_-^{\delta}(\partial_{\bar{z}}R_-)
     ]
     \det{E}_{-+}^{\delta} \\
     &\defeq -k^{\delta}  \det{E}_{-+}^{\delta} \,.
     \end{split}
\end{equation}
Let us study the factor $k^\delta(z)$. 
Using the expressions \eqref{eq5.2}, 
\eqref{eq5.3}, \eqref{eq5.2.2}, we find 
$$
\big((\partial_{\bar{z}}R_+ )E_+^{\delta}\big)_{jj} = (e_+^j(z)+\mO(\delta h^{-3/2})|\partial_z e_+^j(z)),\quad
\big(E_-^{\delta} \partial_{\bar{z}}R_- )\big)_{jj} = (\partial_{\bar z}e_-^j(z)| e_-^j(z)+\mO(\delta h^{-3/2})).
$$
The expressions for the quasimodes in Prop.~\ref{prop4.2} show that $\|\partial_ze_+^j (z)\|=\mO(h^{-1})$, $\|\partial_{\bar{z}}e_-^j(z)\|=\mO(h^{-1})$. For instance,
\begin{equation}\label{eq5.2.6}
\partial_ze_+^j = \e^{-\frac{1}{h}\Phi_+^j} \Big(\partial_ze_+^{j,hol} -h^{-1}(\partial_z\Phi_+^j) e_+^{j,hol}\Big)
=\mO(h^{-1})_{L^2}.
\end{equation}
so we have
\begin{equation}\label{eq5.2.8}
     k^{\delta} = \sum_{j=1}^J
     [(e_+^j|\partial_{z}e_+^j)+(\partial_{\bar{z}}e_-^j|e_-^j)]
     +\mO(\delta h^{-5/2}),
\end{equation}
uniformly in $z\in W(z_0)$. 
Taking the $\partial_z$ derivative of $\|e_+^{j,hol}\|=\e^{\Phi_+^j/h}$, 
cf. \eqref{eq4.2.1}, we get 
\begin{equation*}
    h^{-1}\partial_z\Phi_+^j = \frac{1}{2}\e^{-\frac{2}{h}\Phi_+^j} \left(
    (\partial_{z}e_+^{j,hol}|e_+^{j,hol})
    + (e_+^{j,hol}|\partial_{\bar{z}}e_+^{j,hol})\right).
\end{equation*}
Using these identities together with \eqref{eq4.2.1} and  
\eqref{eq5.2.6}, we find 
\begin{equation}\label{eq5.2.9}
\begin{split}
    (e_+^j|\partial_{z}e_+^j) &= \frac{1}{2}\e^{-\frac{2}{h}\Phi_+^j} 
   (e_+^{j,hol}|\partial_{z}e_+^{j,hol})+ \mO(h^{-1}|z-z_0|^{\infty}+h^{\infty})\\
    &=h^{-1}\partial_{\bar{z}}\Phi_+^{j} +\mO(h^{-1}|z-z_0|^{\infty}+h^{\infty}).
 \end{split}
\end{equation}
Similar computations show that
\begin{equation}\label{eq5.2.9a}
\begin{split}
    (\partial_{\bar{z}}e_-^j|e_-^j) &= 
    \frac{1}{2}\e^{-\frac{2}{h}\Phi_-^j} 
   (\partial_{\bar{z}}e_-^{j,hol}|e_-^{j,hol})+ \mO(h^{-1}|z-z_0|^{\infty}+h^{\infty})\\   
&=h^{-1}\partial_{\bar{z}}\Phi_-^{j} +\mO(h^{-1}|z-z_0|^{\infty}+h^{\infty}),
 \end{split}
\end{equation}
which finally results in
$$
k^\delta = h^{-1}\sum_{j=1}^J \big(\partial_{\bar{z}}\Phi_+^{j}(z) + \partial_{\bar{z}}\Phi_-^{j}(z)\big)+\mO(h^{-1}|z-z_0|^{\infty} 
+\delta h^{-5/2}).
$$
Since $k^{\delta}(z)$ depends smoothly on $z\in W(z_0)$, the equation $\partial_{\bar{z}} l^{\delta}=hk^{\delta}$ can be solved in 
$W(z_0)$ (see e.g. \cite[Theorem 1.4.4]{Ho66} or \cite[page 6]{GrHa78}) with a solution of the form
\begin{equation}\label{eq5.2.10}
  l^{\delta}(z) = -2J h\log (h^{1/4})+ \sum_{j=1}^J \big(\Phi_+^{j}(z;h) + \Phi_-^{j}(z;h)\big)+\mO(|z-z_0|^{\infty}
  +\delta h^{-3/2}).
\end{equation}
Here we added the constant term $-2J h\log (h^{1/4})$ in order to balance the behaviour of
$\Phi_{\pm}^{j}(z_0;h) = h\log (h^{1/4})+\mO(1) $.
From \eqref{eq5.2.5}, we conclude that the following function is holomorphic in $z\in W(z_0)$:
\begin{equation*}
\begin{split}
   G^{\delta}(z;h) & \defeq (-\delta)^{-J}\e^{\frac{1}{h}l^{\delta}(z)}\det E_{-+}^{\delta}(z)\\
	&= 
	\e^{\frac{1}{h}l^{\delta}(z)} \det \left[ (Q e_+^j(z) |  e_-^i(z) )_{i,j\leq J}  	
	+\mO(\delta h^{-5/2})\right].
\end{split}
\end{equation*}
Using \eqref{eq5.2.10} and \eqref{eq4.2.2}, this holomorphic function can be written as
\begin{equation}\label{eq5.2.11a}
	\begin{split}
   	G^{\delta}(z;h)=\big(1+ R_1\big)
	\det \left[ (Q h^{-1/4} e_+^{j,hol}(z) | h^{-1/4} e_-^{i,hol}(z) )_{i,j\leq J}  + R_2\right],
	\end{split}
\end{equation}
where (using \eqref{eq4.2.1})
\begin{equation}\label{eq5.2.12}
\begin{split}
 & R_1 \defeq R_1(z;h)  = \mO(|z-z_0|^{\infty}+\delta h^{-3/2}), \\ 
 & R_2 \defeq R_2(z;h) \defeq \Lambda_+ \mO(\delta h^{-3}) \Lambda_- ,
\end{split}
\end{equation}
where $\Lambda_{\pm}\defeq \mathrm{diag}(\e^{\frac{1}{h}( \Phi_{\pm}^{j}(z)+\mO(|z-z_0|^{\infty} 
  +\delta h^{-3/2}))})_{j=1,\dots, J}$. 
The important fact is that the eigenvalues of $P^{\delta}$ in $W(z_0)$ are given by the zeros of the 
holomorphic function $G^{\delta}(z;h)$.
\section{Random analytic function}\label{sec:RAF}
In this section we provide background material and references concerning the 
theory of random analytic functions, which are needed for the proofs in Sections \ref{sec:LS_M} 
and \ref{sec:LS_V}. We begin by recalling some standard notions and facts about random 
analytic functions and stochastic processes, as discussed for instance 
in \cite{HoKrPeVi09,Kal97}.
\par
Let $O\subset \C$ be an open, simply connected domain, and let $\Hi(O)$ denote 
the space of holomorphic function on $O$. Giving ourselves an exhaustion by compact subsets $ K_j\Subset O $ of the domain $O$, we endow $\Hi(O)$ with the metric 
\begin{equation}\label{eq6.0.1}
	d(f,g)=\sum_{j=1}^{\infty}\frac{1}{2^j}\frac{\|f-g\|_{K_j}}{1+\|f-g\|_{K_j}},
\end{equation}
where $\|f\|_{K_j}\defeq \max_{z\in K_j}|f(z)|$. This metric induces the topology of uniform 
convergence of analytic functions on compact set. This makes $\Hi(O)$ a complete 
separable metric space, and we may equip 
it with its Borel $\sigma$-algebra  $\mathcal{B}(\Hi(O))$. This makes $(\Hi(O),\mathcal{B}(\Hi(O)))$ a measurable space.
\begin{defn}
Let $(\mathcal{M},\mathcal{A},\nu)$ be a probability space. Then, any measurable map 
	\begin{equation*}
	 f: (\mathcal{M},\mathcal{A})\longrightarrow (\Hi(O),\mathcal{B}(\Hi(O)))
	\end{equation*}
is called a $\C$-valued stochastic process on $O$ with paths in $\Hi(O)$.
\end{defn}
The Borel $\sigma$-algebra $\mathcal{B}(\Hi(O))$ is equal to the 
$\sigma$-algebra generated by the evolution maps 
\begin{equation*}
	\pi_z:~ \Hi(O)\longrightarrow \C, \quad \pi_z g \defeq g(z), \quad z\in O,
\end{equation*}
namely it is the smallest $\sigma$-algebra in $\Hi(O)$ such that $\pi_z$ is 
measurable for every $z\in O$. It is then a standard fact that 
$f:(\mathcal{M},\mathcal{A})\to (\Hi(O),\mathcal{B}(\Hi(O)))$ is measurable if and only if 
$\pi_z f:(\mathcal{M},\mathcal{A})\to (\C,\mathcal{B}(\C))$ is measurable for all $z\in O$, 
where $\mathcal{B}(\C)$ denotes the Borel $\sigma$-algebra in $\C$. 
\par
Hence, for any $ f: \mathcal{M}\longrightarrow \Hi(O)$ a $\C$-valued stochastic 
process on $O$ with paths in $\Hi(O)$, we can regard $f$ as well as a
function from $O\times \mathcal{M}$ to $\C$:
\begin{equation*}
	f(z,\omega) = \pi_z f(\omega), \quad (z,\omega)\in O\times \mathcal{M}.
\end{equation*}
From now on we will call such an $f$ simply a \textit{random analytic function on $O$}. 
Due to the above measurability property it is equivalent to consider $f$ as a 
collection of random elements $f(z)$ in $\C$. The associated finite-dimensional 
distributions are given by direct images of the probability measure $\nu$ by 
the random vector $(f(z_1),\dots,f(z_k))\in\C^k$, i.e. 
\begin{equation*}
	\mu_{z_1,\dots,z_k} = (f(z_1),\dots,f(z_k))_{*}(\nu), 
	\quad z_1,\dots, z_k \in O, ~ k\in \N^*. 
\end{equation*}
For all $(z_1,\dots, z_k)\in O^k$ we have that $\mu_{z_1,\dots,z_k}$ is a probability 
measure on $\C^k$ , called the joint distribution of $(f(z_1),\dots,f(z_k))$.
\\
\par
The next result tells us that the distribution of a random analytic function, i.e. the 
direct image measure $f_*\nu$, is determined by its finite-dimensional distributions, 
see e.g. \cite[Proposition 2.2]{Kal97}.
\begin{thm}\label{thm:finite-dim-dist}
Let $f$ and $g$ be two random analytic	 functions, then 
\begin{equation*}
	f \stackrel{d}{=} g \quad \Longleftrightarrow \quad 
	(f(z_1),\dots,f(z_n))\stackrel{d}{=}(g(z_1),\dots,g(z_n)), ~
	\forall z_1,\dots,z_n\in O,~ \forall n\in\N,
\end{equation*}
where the symbol $\stackrel{d}{=}$ means equality in distribution, i.e. that the 
respective direct image measures are equal. 
\end{thm}
Next, let us recall that a $\C^k$-valued random variable $X$ is said to have 
a centred symmetric complex Gaussian distribution with a covariance matrix 
$\Sigma_k\in \mathrm{GL}_k(\C)$, in short $X\sim \mathcal{N}_{\C}(0,\Sigma_k)$, if 
its distribution is given by
\begin{equation*}
	X_*\nu= (\det \pi \Sigma_k)^{-1}\e^{-X^*\Sigma_k^{-1}X} L(dX),
\end{equation*}
where $L(dX)$ denotes the Lebesgue measure on $\C^k$. The covariance matrix $\Sigma_k$ is Hermitian definite positive. As usual, this Gaussian distribution is characterized by its variances:
\begin{equation*}
	\erw[ X_i X_j ] = 0, \quad
	\erw[ X_i \overline{X}_j ] = (\Sigma_k)_{ij},\qquad 1\leq i,j\leq k.
\end{equation*}
\begin{defn}(Gaussian analytic function - GAF)
 Let $O\subset\C$ be an open simply connected complex domain. 
 A random analytic function $f$ on $O$ is called a Gaussian analytic function 
 on $O$ if its finite-dimensional 
 distributions are centred symmetric complex Gaussian, namely if for all $k\in\N^*$ and all $z_1,\dots,z_k\in O$, the 
 random vectors $(f(z_1),\dots,f(z_k)) \sim \mathcal{N}_{\C}(0,\Sigma_k)$ for some covariance matrix $\Sigma_k$.
\end{defn}
The matrix $\Sigma_k$ depends on $(z_1,\dots,z_k)$. Each of its entries $(\Sigma_k)_{ij}$ is
given by the covariance kernel
\begin{equation*}
 K(z_i,\bar{z}_j)\stackrel{\mathrm{def}}{=}\erw[ f(z_i) \overline{f(z_j)} ], 
\end{equation*}
which is a $z_i$-holomorphic and $z_j$-anti-holomorphic function on $O\times O$. Hence, the complete distribution of the GAF is fully characterized by the covariance kernel. 

For more details on the theory of Gaussian analytic functions we refer the reader to 
\cite{HoKrPeVi09}. 
\subsection{Sequences of random analytic functions}\label{sec:Conv}
For the reader's convenience, we recall in this section some of the notions and results 
concerning the convergence of sequences of random analytic functions which we shall 
need in the sequel. 
\begin{defn}\label{def_CD}
Let $X_n$, $n\in\N$, and $X$ be random variables taking values in a complete separable metric 
space $(S,d)$, these variables being defined on probability spaces $(\mathcal{M}_n,\mathcal{F}_n,\nu_n)$, 
$n\in\N$, respectively $(\mathcal{M},\mathcal{F},\nu)$. We say that 
$(X_n)_n$ {\em converges in distribution} to $X$ if the induced probability measures on $S$ 
converge:
\begin{equation*}
	(X_n)_*(\nu_n) \stackrel{w^*}{\longrightarrow} X_*(\nu), ~ n\to\infty,
\end{equation*}
or equivalently if for all bounded continuous functions on $S$, $\varphi\in\mathcal{C}_b(S,\R)$, 
\begin{equation*}
	\int \phi(X_n) d\nu_n \longrightarrow \int \phi(X)d\nu,~ n\to\infty.
\end{equation*}
When this is the case we write $X_n\stackrel{d}{\to} X$. 
\end{defn}
\begin{defn}\label{def_CFidi}
Let $O\subset\C$ be an open and simply connected domain. Let $f_n$, $n\in\N$, and $f$ 
be random analytic functions on $O$ (not necessarily defined on the same probability space). 
We say that $f_n$ {\em converges in the sense of finite dimensional distributions} to $f$ if 
for all $k\geq 1$ and all $z_1,\dots,z_k\in O$,
\begin{equation*}
	(f_n(z_1),\dots,f_n(z_k))\stackrel{d}{\longrightarrow} 
	(f(z_1),\dots,f(z_k)), ~ n\to\infty.
\end{equation*}
When this is the case we write $f_n\stackrel{fd}{\to} f$. 
\end{defn}
As discussed in Thm~\ref{thm:finite-dim-dist}, the distribution of each random 
analytic function is uniquely determined by its finite dimensional distributions. However, the
convergence of a sequence of random analytic 
functions  in the sense of finite dimensional distributions does in general {\em not} imply the convergence in distribution. 
To achieve this implication, one needs 
to add a {\em tightness condition}, providing the relative compactness of the sequence. 
\begin{defn}
Let $(\mu_n)_n$ be a sequence of probability measures on some complete separable 
metric space $(S,\mathcal{B}(S))$. The sequence $(\mu_n)_n$ is said to be {\em tight} if 
\begin{equation*}
	\sup\limits_{K\Subset S} \liminf\limits_{n\to\infty} \mu_n( K )=1,
\end{equation*}
where the supremum is taken over all compact sets $K\Subset S$. \\
Similarly, a sequence of random variables $(X_n)_n$ taking values in $S$ is called tight if 
\begin{equation*}
	\sup\limits_{K\Subset S} \liminf\limits_{n\to\infty} \mathds{P}[X_n \in K ]=1,
\end{equation*}
where we use the standard notation that $\mathds{P}[ X_n \in K ] = \mu_n(K)$ 
is the probability measure induced by $X_n$ on $S$. 
\end{defn}
An important result due to Prohorov, see for instance \cite[Theorem 14.3]{Kal97}, is the following. 
\begin{thm}[Prohorov]\label{thm:Pro}
For any sequence of random variables $(X_n)_n$ taking values in a complete separable metric space, 
tightness is equivalent to relative compactness in distribution, i.e. the sequence of probability measures 
$(\mu_n)_n$ induced by $(X_n)_n$ is relatively compact in the weak-* topology.
\end{thm}
\begin{rem}\label{Rem6.1}
As a consequence of the above proposition, for a tight sequence of probability 
measures on a  complete separable metric space $S$, convergence with respect to the $w^*$-topology of 
$\mathcal{C}_b(S)'$, where $\mathcal{C}_b(S)$ denotes the space of bounded continuous functions on $S$, is equivalent to convergence with respect to the $w^*$-topology of $\mathcal{C}_c(S)'$ where $\mathcal{C}_c(S)$ denotes the space of continuous functions on $S$ with compact support. The latter  topology is sometimes referred to as the {\em vague topology}. 
\end{rem}
Shirai \cite[Proposition 2.5]{Sh12} provides a useful criterion for tightness of sequences 
of random analytic functions: 
\begin{prop}\label{prop6.1} \cite{Sh12}
Let $f_n$, $n\in \N$, and $f$ be random analytic functions on an open simply connected 
set $O\subset\C$. Suppose that for all 
compact sets $K\Subset O$, the sequence of random real variables $(\|f_n\|_{L^{\infty}(K)})_n$ is tight , i.e. 
\begin{equation*}
	\lim\limits_{r\to\infty} \limsup\limits_{n\to\infty} \mathds{P}[\|f_n\|_{L^{\infty}(K)} > r] = 0.
\end{equation*}
Then, the sequence $(f_n)_n$ is tight in the space of random analytic functions on $O$, and 
\begin{equation*}
	f_n\stackrel{fd}{\longrightarrow} f ~\text{as}~ n\to\infty \quad 
	\implies \quad f_n\stackrel{d}{\longrightarrow} f ~\text{as}~ n\to\infty.
\end{equation*}
These properties naturally extend to a family of functions $(f_h)_{h}$ depending on a continuous parameter $h\in (0,h_0]$: it holds for any sequence $(f_{h_n})_n$ such that $h_n\to 0$ as $n\to\infty$.
\end{prop}
%
%
\subsection{Point processes given by the zeros of a random analytic function} 
Let $O\subset\C$ be an open connected domain, and let $M(O)$ denote 
the space of complex valued, locally finite Borel measures on $O$, which we endow with the 
vague topology of $\mathcal{C}_c(O)'$. This topology is metrisable, and it makes $M(O)$ a 
complete, separable metric space which we can equip with its Borel $\sigma$-algebra. 
The random elements in $M(O)$ are called random measures on $O$. 
\par
Inside $M(O)$ we distinguish the space $PM(O)$ of integer-valued measures in $M(O)$. It forms  a closed 
subspace of $M(O)$. Any element $\mu\in PM(O)$ is a point measure: it can be 
expressed as 
\begin{equation*}
	\mu = \sum_{z_i} \delta_{z_i}
\end{equation*}
where $\delta_{z_i}$ denotes the Dirac measure on the point $z_i\in O$. The points $\{z_i\}_i$ form a finite or countable set, which has no accumulation point (yet  each $z_i$ can be repeated 
finitely many times). 
A $PM(O)$-valued random variable is called a
{\em point process} on $O$.
\\
\par
Let $f\not\equiv 0$ be a nontrivial analytic function on $O$. Then, we call 
\begin{equation}\label{eq:PpP}
	\cZ_f \defeq \sum_{\lambda\in f^{-1}(0)}\delta_\lambda
\end{equation}
the point measure defined by its set of zeros (counted with multiplicities). If $f$ is a random analytic function (with $f\not\equiv 0$ a.s.), then 
$\cZ_f$ defines a point process on $O$, which we will call a zero point process. 
This is indeed a point process: $f$ is measurable and 
every functional $\pi_{\varphi}:\mathcal{H}(O)\setminus\{0\}\to\R$, 
$\pi_{\varphi}(f)=\langle \cZ_f,\varphi\rangle$, $\varphi\in\mathcal{C}_c(O,\R)$, 
is continuous on $(\mathcal{H}(O)\backslash\{0\},d)$, with $d$ the metric as in \eqref{eq6.0.1}. 
This was observed for instance 
by Shirai \cite{Sh12}.
\begin{rem}\label{rem6.1}
	An easy extension is that the mapping 
	$\pi_{\varphi^{\otimes M}}:  (\mathcal{H}(O)\backslash\{0\})^M
	\to \C$, 
	\begin{equation*}
	\pi_{\varphi^{\otimes M}}(f_1,\dots,f_M) = 
	\langle \cZ_{f_1},\varphi_1\rangle\cdots \langle \cZ_{f_M},\varphi_M\rangle,
	\quad \varphi^{\otimes M}\in\bigotimes_{j=1}^M \mathcal{C}_c(O,\R), 
	\end{equation*}
	 is continuous 
	with respect to the metric $\widetilde{d}(f,g)=\sum_{i=1}^M d(f^i,g^i)$ on 
	$\mathcal{H}(O)^M$.
\end{rem}
The following result is essentially due to Shirai \cite[Proposition 2.3]{Sh12}.
\begin{prop}\label{prop6.6}
Let $O\subset\C$ be an open, simply connected domain. Let $f_n$, $n\in\N$, and $f$ 
be random analytic functions on $O$, not necessarily defined on the same probability 
space. Suppose that $f_n,f \not\equiv 0$ almost surely, and suppose that $f_n$ converges 
in distribution to $f$. 

Then the zero point processes  $\cZ_{f_n}$ converge in distribution 
to $\cZ_f$. Moreover, for any  $\varphi^{\otimes M}\in\bigotimes_{j=1}^M \mathcal{C}_c(O,\R)$, we have the following 
convergence of real random variables:
\begin{equation*}
	\langle \cZ_{f_n},\varphi_1\rangle\cdots \langle \cZ_{f_n},\varphi_M\rangle 
	\stackrel{d}{\longrightarrow} 
	\langle \cZ_{f},\varphi_1\rangle\cdots \langle \cZ_{f},\varphi_M\rangle.
\end{equation*}
\end{prop}
\begin{proof}
The convergence in distribution of the point processes $\cZ_{f_n}$ to $\cZ_f$ is 
equivalent to 
\begin{equation*}
	\langle \cZ_{f_n},\varphi\rangle \stackrel{d}{\longrightarrow} \langle \cZ_{f},\varphi\rangle.
\end{equation*}
for all $\varphi\in\mathcal{C}_c(O,\R)$. As discussed in Remark \ref{rem6.1}, the mapping 
$\pi_{\varphi^{\otimes M}}$ is continuous. Hence, the continuous mapping theorem, stating that the convergence in distribution of random variables is preserved through continuous mappings between metric spaces, see e.g. \cite[Theorem 3.27]{Kal97}, implies the claimed results.
\end{proof}
\subsection{A central limit theorem for complex valued random variables}
We prove the following version of a central limit theorem for complex valued random variables. 
\begin{thm}\label{thmCLT}
Let $\sigma>0$ and let $\xi\sim \mathcal{N}_{\C}(0,\sigma)$ be a complex Gaussian 
random variable with mean $0$ and variance $\sigma$. Let $\{\xi_{nj}\}_{n\in \N, 1\leq j \leq N(n)}$ 
be a triangular array, with $N(n)\to +\infty$, as $n\to +\infty$, of row-wise independent complex-valued random variables satisfying 
\begin{enumerate}[(i)]
	\item $\sum_{j=1}^{N(n)} |\erw [\xi_{nj} ]|\longrightarrow 0, \text{ as } n\to +\infty$, 
	\item $\sum_{j=1}^{N(n)} \erw [\xi_{nj}^2 ]\longrightarrow 0, \text{ as } n\to +\infty$, 
	\item $\sum_{j=1}^{N(n)} \erw [|\xi_{nj}|^2 ]\longrightarrow \sigma, 	\text{ as } n\to +\infty$,
	\item $\sum_{j=1}^{N(n)} \erw [|\xi_{nj}|^2\, \mathbf{1}_{\{|\xi_{nj}|>\varepsilon\}} ]\longrightarrow 0, 	\text{ as } n\to +\infty$, for any $\varepsilon >0$. 
	\end{enumerate}
Then, $\sum_{j=1}^{N(n)} \xi_{nj} \stackrel{d}{\longrightarrow} \xi$ as $n\to +\infty$.
\end{thm}
\begin{rem}
	Condition $(iv)$ is also known as the Lindeberg condition for a central limit theorem. 
	A simpler version of the above Theorem was presented in \cite[Proposition 4.2]{Sh12} supposing 
	that all random variables $\xi_{nj}$ have expectation $0$ and satisfy that for all $n$ and $j$
	the random variables $\Rea \xi_{nj}$ and $\Ima \xi_{nj}$ are independent and have the same 
	variance. Notice that these assumptions imply conditions $(i)$ and $(ii)$. 
\end{rem}
\begin{proof}[Proof of Theorem \ref{thmCLT}]
The proof we present here is a modification of the proof of the well-known central limit theorem 
under the Lindeberg condition, see \cite[Theorem 4.12]{Kal97}. 
\\
\par
1. Identifying $\C \cong \R^2$, it follows that $\sum_{j=1}^{N(n)} \xi_{nj} \stackrel{d}{\longrightarrow} \xi$,  as $n\to +\infty$, is equivalent to the following convergence of random vectors in $\R^2$:
\begin{equation}\label{clt1}
	 \sum_{j=1}^{N(n)} \begin{pmatrix} \Rea \xi_{nj} \\ \Ima \xi_{nj} \\ \end{pmatrix}
	 \stackrel{d}{\longrightarrow} \Xi \sim \mathcal{N}(0,\Sigma) ,  \text{ as } n\to +\infty,
\end{equation}
where $\Xi$ is a random vector in $\R^2$ with $2$-dimensional real Gaussian distribution 
$\mathcal{N}(0,\Sigma) $ with mean $0$ and covariance matrix 
$\Sigma = \mathrm{diag}(\sigma/2,\sigma/2) $. More precisely, 
\begin{equation*}
	 \Xi_*(d\prob) = (\det 2\pi \Sigma )^{-1/2} \e^{- \frac{1}{2} \Xi^t \Sigma^{-1} \Xi} d\Xi_1 d\Xi_2. 
\end{equation*}
By the Cram\'er-Wold Theorem \cite[Corollary 4.5]{Kal97} it follows that \eqref{clt1} is equivalent 
to 
\begin{equation}\label{clt2}
	  \widetilde{\xi}_n(s,t)\defeq \sum_{j=1}^{N(n)} (s\,\Rea \xi_{nj} +t \,\Ima \xi_{nj} )
	 \stackrel{d}{\longrightarrow}( s\,\Xi_1 + t\,\Xi_2) \defeq \widetilde{\xi}(s,t) ,  \text{ as } n\to +\infty, 
\end{equation}
for any $s,t\in \R$. Since $\Xi_1$ and $\Xi_2$ are independent and identically distributed 
real Gaussian random variables $\sim \mathcal{N}(0,\sigma/2)$, it follows that 
$\widetilde{\xi}(s,t) \sim \mathcal{N}(0,(s^2+t^2)\sigma/2)$. 
\par
Write $\widetilde{\xi}_{nj}(s,t) = s\,\Rea \xi_{nj} +t \,\Ima \xi_{nj} $ and set 
$c_{nj} = \erw [ \widetilde{\xi}_{nj}(s,t)^2]$.
Let $\{\zeta_{nj}\}_{n\in\N}$ be a triangular array of row-wise independent (real) 
Gaussian random variables with distribution 
$\zeta_{nj} \sim \mathcal{N}(0,c_{nj})$. Let $c_n = \sum_{j=1}^{N(n)}c_{nj}$, 
then 
\begin{equation*}
	 \zeta_n \defeq \sum_{j=1}^{N(n)}\zeta_{nj} \sim \mathcal{N}(0,c_n).
\end{equation*}
Using that $(\Rea z)^2 = \frac{1}{4}(z^2+\bar{z}^2 + 2z\bar{z})$ and 
$(\Ima z)^2 = -\frac{1}{4}(z^2+\bar{z}^2 - 2z\bar{z})$, it follows by  $(ii)$ that 
\begin{equation}\label{clt5}
	\sum_{j=1}^{N(n)}\erw[ (\Rea \xi_{nj})^2 ] = \sum_{j=1}^{N(n)}\erw[ (\Ima \xi_{nj})^2 ] 
	+ o_n(1),
\end{equation}
where $o_n(1)\to 0$ as $n\to\infty$. Similarly, we see by using \eqref{clt5}, the identitiy 
$z^2 = (\Rea z)^2 - (\Ima z)^2 + 2i\Rea z \Ima z$ and $(ii)$, that 
\begin{equation}\label{clt6}
	\sum_{j=1}^{N(n)}\erw[ \Rea  \xi_{nj} \Ima \xi_{nj}] = o_n(1) .
\end{equation}
Then,
\begin{equation}\label{clt6.1}
\begin{split}
	c_n &= \sum_{j=1}^{N(n)}\erw [ \widetilde{\xi}_{nj}(s,t)^2] \\
	& = \frac{s^2+t^2}{2}  \sum_{j=1}^{N(n)}\erw[ (\Rea \xi_{nj})^2 + (\Ima \xi_{nj})^2] 
	      +o_n(1) \\ 
	     & \to \frac{(s^2+t^2)\sigma}{2} , \quad \text{as } n\to\infty.
\end{split}
\end{equation}
Here, in the last line we used as well \textit{(iii)}. 
This implies that $\zeta_n \stackrel{d}{\to}\widetilde{\xi}(s,t)$, as $n\to \infty$.
\\
\\
2. We show that $\widetilde{\xi}_n(s,t)$ and $\zeta_n$ converge to the same limit in 
distribution. Recall Definition \ref{def_CD}. 
\par
Let $\mu_n = (\widetilde{\xi}_n(s,t))_*(d\prob)$, $\mu = (\widetilde{\xi}(s,t))_*(d\prob)$ 
and $\nu_n = (\zeta_n)_*(d\prob)$. We use a similar notation for the probability measures induced 
by the random variables $ \widetilde{\xi}_{nj}(s,t)$ and $\zeta_{nj}$. Since all induced measures 
are probability measures it follows that $\mu_n \stackrel{w^*}{\to} \mu$, as $n\to \infty$, if and 
only if the Fourier transforms of these measures converges pointwise for any point, i.e. 
$\widehat{\mu}_n(\tau) \to\widehat{\mu}(\tau)$, as $n\to \infty$, for any $\tau\in\R$. Notice that  
$\widehat{\mu}_n(\tau)=\erw [ \e^{i\tau\widetilde{\xi}_n} ]$ and 
$\widehat{\nu}_n(\tau)=\erw [ \e^{i\tau\zeta_n(s,t)} ]$ . Similar statments hold 
for $\nu_n$. Hence, to prove the statment of the Theorem it is sufficient to show that  
$|\widehat{\mu}_n(\tau) - \widehat{\nu}_n(\tau) |\to 0$ for all $\tau\in\R$. 
\\
\par
We begin with two helpful identities: Let $z_{i},w_{i}\in \C$, for $i=1,\dots, n$, with $|z_i|, |w_i| =1$. 
Then, $|z_1z_2 - w_1w_2| \leq |z_1-w_1| + |z_2-w_2|$, and by induction 
\begin{equation}\label{clt3}
	 \left| \prod_{i=1}^n z_i - \prod_{i=1}^n w_i\right| \leq \sum_{i=1}^n |z_i - w_i|. 
\end{equation}
Moreover, by Taylor expansion we see that for all $\tau\in\R$
\begin{equation}\label{clt4}	
	 \left| \e^{i\tau} - \sum_{k=1}^n \frac{(i\tau)^k}{k!} \right| \leq \min\left(\frac{2|\tau|^n}{n!},\frac{|\tau|^{n+1}}{(n+1)!} \right).
\end{equation}
Write $a_{nj} = \erw[\widetilde{xi}_{nj}]$. Fix $\tau\in \R$. 
Since $\xi_{nj}$ and $\zeta_{nj}$ are row-wise independent, it follows by \eqref{clt3}, \eqref{clt4}, 
\begin{equation*}
\begin{split}
	&|\widehat{\mu}_n(\tau) - \widehat{\nu}_n(\tau)  | = 
	 \left| \prod_{j=1}^{N(n)} \widehat{\mu}_{nj}(\tau)  - \prod_{j=1}^{N(n)} \widehat{\nu}_{nj}(\tau)\right| \\
	 &\leq \sum_{j=1}^{N(n)}\left|  \widehat{\mu}_{nj}(\tau)  -  \widehat{\nu}_{nj}(\tau)\right| \\
&\leq \sum_{j=1}^{N(n)}\left| \erw[  \e^{i\tau\widetilde{\xi}_{nj}(s,t)}] - 1 - i\tau a_{nj} + \frac{1}{2}\tau^2c_{nj} \right| 
+    \sum_{j=1}^{N(n)}\left|  \erw[  \e^{i\tau\zeta_{nj}}] - 1 - i\tau a_{nj} + \frac{1}{2}\tau^2c_{nj} \right|\\
&\leq \mO(1)
 \sum_{j=1}^{N(n)} \erw\left[ \widetilde{\xi}_{nj}(s,t)^2 \min(1, |\widetilde{\xi}_{nj}(s,t)|) \right]
+    \mO(1)\sum_{j=1}^{N(n)} \left( \erw[\zeta_{nj}^2 \min(1, |\zeta_{nj}|)] + |a_{nj}| \right) \\ 
& \defeq I_1(n) + I_2(n).
\end{split}
\end{equation*}
We are going to show that both terms in the last line converge to $0$ as $n\to\infty$. 
Notice that for any $\varepsilon >0$
\begin{equation*}
\begin{split}
\sum_{j=1}^{N(n)}\erw\left[ \widetilde{\xi}_{nj}(s,t)^2 \min(1, |\widetilde{\xi}_{nj}(s,t)|) \right] 
&\leq \mO(1) \sum_{j=1}^{N(n)}\erw\left[|\xi_{nj}|^2 \min(1, |\xi_{nj}|)\right] \\
&\leq \mO(1) \varepsilon \sum_{j=1}^{N(n)}\erw\left[|\xi_{nj}|^2\right] + 
\mO(1) \sum_{j=1}^{N(n)} \erw\left[|\xi_{nj}|^2\mathbf{1}_{|\xi_{nj}|> \varepsilon}\right] \\
&\leq \mO(1) \varepsilon c_n + 
\mO(1) \sum_{j=1}^{N(n)} \erw\left[|\xi_{nj}|^2\mathbf{1}_{|\xi_{nj}|> \varepsilon}\right].
\end{split}
\end{equation*}
Taking the limit $n\to \infty$ and using \eqref{clt6.1} and assumption $(iv)$, it follows that 
$\lim_{n\to\infty}I_1(n) = \mO(\varepsilon)$ for any $\varepsilon$. Since $I_1(n)$ does not 
depend on $\varepsilon$ it follows that $I_1(n)\to 0$, as $n\to \infty$. 
\\
\par
Next, let us treat $I_2(n)$. First notice by assumption $(i)$ we have that 
$\sum_{j=1}^{N(n)}|a_{nj}| \to 0$, as $n\to \infty$.
Since $\zeta_{nj}\sim \mathcal{N}(0,c_{nj})$, we have that 
\begin{equation*}
	 \erw[|\zeta_{nj}|^4 ] = \int x^4 \e^{-\frac{x^2}{2c_{nj}}}\frac{dx}{\sqrt{2c_{nj}}} 
	 			      = \frac{3c_{nj}^2}{2}.
\end{equation*}
Hence, 
\begin{equation*}
\begin{split}
\sum_{j=1}^{N(n)}\erw[\zeta_{nj}^2 \min(1, |\zeta_{nj}|)] 
&\leq \sum_{j=1}^{N(n)} \erw[|\zeta_{nj}|^3] 
\leq  \sum_{j=1}^{N(n)}\erw[|\zeta_{nj}|^4]^{1/2}  \erw[|\zeta_{nj}|^2]^{1/2}  \\
 &\leq \frac{\sqrt{3}}{\sqrt{2}} \sup\limits_{j} c_{nj}^{1/2}\sum_{j=1}^{N(n)}c_{nj} 
 = \frac{\sqrt{3}}{\sqrt{2}}c_n \sup\limits_{j} c_{nj}^{1/2}.
\end{split}
\end{equation*}
Finally, notice that for any $\varepsilon >0$ 
\begin{equation*}
\begin{split}
	\sup_j c_{nj} &=\sup_j\erw[ |\widetilde{\xi}_{nj}(s,t)|^2] 
			   \leq \mO(1) \sup_j\erw[ |\xi_{nj}|^2] \\
			   &\leq \mO(1) \sup_j\erw[ |\xi_{nj}|^2\mathbf{1}_{|\xi_{nj}|\leq  \varepsilon}] + 
			   \mO(1) \sup_j\erw[ |\xi_{nj}|^2\mathbf{1}_{|\xi_{nj}|> \varepsilon}] \\
			    &\leq \mO(\varepsilon^2 )+ 
			   \mO(1) \sum_{j=1}^{N(n)} \erw[ |\xi_{nj}|^2\mathbf{1}_{|\xi_{nj}|> \varepsilon}]. 
\end{split}
\end{equation*}
In view of assumption $(iv)$, we see that $\sup_j c_{nj}\to 0$ as $n\to +\infty$. In conclusion 
it follows that  $|\widehat{\mu}_n(\tau) - \widehat{\nu}_n(\tau)  | \to 0$, as $n\to \infty$, which concludes 
the proof of the Theorem.
\end{proof}

%
\section{Local Statistics of the eigenvalues of $P_h$ perturbed by a random matrix $M_{\omega}$}
\label{sec:LS_M}
In this section we are interested in the local statistics of the eigenvalues of the perturbed 
operator
\begin{equation}\label{eq7.1}
	P^{\delta}=P_h + \delta M_{\omega}, 
\end{equation}
see also \eqref{eq1.11}, where the perturbation $M_{\omega}$ is the random matrix defined in \eqref{eq1.10}. Let 
$\mathrm{PD}_{N}(0,C/h)=D(0,C/h)^{N}\subset \C^{N}$, $N\in\N^*$, denote 
the polydisc of radius $C/h$ centred at $0$. Recall from \eqref{eq1.15.2} 
that we have restricted the random variables to large polydiscs:
\begin{equation}\label{eq6.3.0}
q=(q_{j,k})_{j,k<N(h)}\in \mathrm{PD}_{N(h)^2}(0,C/h),
\end{equation}
with as consequence the Hilbert-Schmidt estimate \eqref{eq5.2.1.5}. Recall that the 
random variables $q_{j,k}$, $j,k<N(h)$, are independent and identically 
distributed. We denote by $(\mathcal{M}_h,\mathcal{F}_h, \prob_h)$, the 
probability space $(\mathcal{M},\mathcal{F}, \prob)$ restricted to 
$\bigcap_{j,k<N(h)} q_{j,k}^{-1}(D(0,C/h))$. The random variables restricted to this space will be denoted by 
\begin{equation}\label{eq7.2}
	q_{j,k}^h, \quad j,k\leq N(h).
\end{equation}
The restricted probability measure $\prob_h$ is given by the conditional probability 
\begin{equation}\label{eq7.2.5}
	\prob_h[~\cdot~] = \prob[~\cdot~ | ~|q_{j,k}|\leq C/h].
\end{equation}
Like their unrestricted versions, the $q_{j,k}^h$ are independent and identically distributed. Moreover, by \eqref{eq1.9.1}, 
\eqref{eq1.9}, and by Fubini's Theorem, we obtain the following estimates of their average and variance of these restricted variables:
\begin{equation}\label{eq7.3}
	\begin{split}
	\big | \erw[ q_{j,k}^h ] \big |& \leq \left(1 + \mO(h^{4+\varepsilon_0})\right)
	\erw\left[ \int_0^{\infty} 1_{\{|q_{j,k}|>t\}}dt\, 1_{\{|q_{j,k}|\geq C/h\}}\right]\\
	& =\left(1 + \mO(h^{4+\varepsilon_0})\right)
	\int_0^{\infty} \prob\left [\{|q_{j,k}| > t\}\cap\{|q_{j,k}| \geq C/h\}\right] dt \\
	& = \mO(h^{3+\varepsilon_0}).
	\end{split}
\end{equation}
A similar computation shows that 
\begin{equation}\label{e:E(q2)}
\big | \erw[ (q_{j,k}^h)^2 ] \big | = \mO(h^{2+\vareps_0})\,.
\end{equation}
Finally,
\begin{equation}\label{eq7.4}
	\begin{split}
	\erw[ |q_{j,k}^h|^2 ] &\leq \left(1 + \mO(h^{4+\varepsilon_0})\right)\left(\erw[ |q_{j,k}|^2 ] + 
	\int_0^{\infty}\prob\left [\{|q_{j,k}|^2 > t\}\cap\{|q_{j,k}| \geq C/h\}\right] dt \right)\\
	& = 1+ \mO(h^{2+\vareps_0}).
	\end{split}
\end{equation}
Next, let $\Omega\Subset\mathring{\Sigma}$ be as in \eqref{eq1.8.2} and pick a $z_0\in\Omega$. 
By the Grushin problem constructed in Proposition \ref{prop5.2.1} and by \eqref{eq5.2.11a}, 
the eigenvalues  of the perturbed operator $P^{\delta}$, in a small neighborhood $W(z_0)$ 
of $z_0$, are precisely the zeros of the function
\begin{equation}\label{n1}
   	G^{\delta}(z;h)=\big(1+ R_1(z;h)\big)
	\det \left[ (M_{\omega}h^{-\frac{1}{4}}e_+^{j,hol}(z) |h^{-\frac{1}{4}}e_-^{i,hol}(z) )_{i,j\leq J} 
	+ R_2(z;h) \right].
\end{equation}
Since the spectrum of $P^{\delta}$ in $W(z_0)$ is discrete, it follows that 
$G^{\delta}(\cdot;h)\not\equiv 0$ in $W(z_0)$. To compute the local spectral statistics near the point $z_0$, we rescale the 
spectral parameter $z$ by the average distance between nearest neighbours, which is 
of order $\asymp h^{1/2}$, namely we write
\begin{equation}\label{e:rescaling}
z=z_w=z_0+h^{1/2}w\,,\quad w\in\C.
\end{equation}
For $z$ an eigenvalue of $P^\delta$ in $W(z_0)$, we will call the corresponding $w$ a {\em rescaled eigenvalue}, which forms the rescaled spectrum. We will focus on the eigenvalues in a {\em microscopic} neighbourhood of $z_0$: we will consider an arbitrary open, simply connected and relatively compact set $ O\Subset\C$, and only consider the points such that $w\in O$. For $h>0$ small enough, the rescaled eigenvalues 
are precisely the zeros of the $w$-holomorphic function
\begin{equation}\label{eq6.3.1}
\begin{split}
F^{\delta}_h(w)&\stackrel{\mathrm{def}}{=} G^{\delta}(z_w;h) \\
&=(1+ \widetilde{R}_1(w;h))\det \left[ (f^{\delta,h}_{i,j}(w))_{i,j\leq J} + \widetilde{R}_2(w;h)\right] , 
\quad w \in  O.
\end{split}
\end{equation}
with 
\begin{equation}\label{eq6.3.1b}
\begin{split}
&f^{\delta,h}_{i,j}(w)\stackrel{\mathrm{def}}{=} 
	(M_{\omega}h^{-\frac{1}{4}} e_+^{j,hol}(z_w) | h^{-\frac{1}{4}}e_-^{i,hol}(z_w)), 
	\quad 0\leq i,j\leq J, \\ 
& \widetilde{R}_1(w;h)\stackrel{\mathrm{def}}{=} R_1(z_w;h), \\
& \widetilde{R}_2(w;h)\stackrel{\mathrm{def}}{=} R_2(z_w;h).
\end{split}
\end{equation}
From now on we let 
\begin{equation}\label{eq6.3.1a}
	O\Subset\C \text{ be an open, simply connected and relatively compact set } 
\end{equation}
and $h>0$ small enough such that \eqref{eq6.3.1} is well defined. 
\par
We have $F^{\delta}_h\not\equiv 0$ in 
$O$ hence, by the discussion in Section \ref{sec:Conv}, 
\begin{equation}
	\cZ_h \defeq \sum_{w\in (F^{\delta}_h)^{-1}(0)}\delta_w
\end{equation}
is a well-defined point process on $O$ (the zeros are repeated according to 
their multiplicities). This is the rescaled spectral point process $\cZ_{h,z_0}^M$ of \eqref{eq:PPM}, restricted to the set $O$.
%
\subsection{Covariance}\label{sec:Cov1}
In this section we study the covariance of the random variables 
$f^{\delta,h}_{i,j}(w)$ defined over $O\Subset \C$. This will be of fundamental 
importance to the main result of this paper. 
\begin{prop}\label{prop7.1}
Let $\sigma_{\pm}^j(z_0)$ be as in \eqref{eq4.4.4} and let
$f^{\delta,h}_{i,j}(\bullet)$, for $0\leq i,j \leq J$, be the rescaled functions as in \eqref{eq6.3.1b}. 

Then, we have the following expression for the covariance of those functions, uniformly for $v,w\in  O$:
\begin{equation}\label{e:covar}
\begin{split}
\erw\big[f^{\delta,h}_{i,j}(v) &\overline{f^{\delta,h}_{l,k}(w)}\big] \e^{-F_{i,j}(v;h)-\overline{F_{l,k}(w;h)}} \\
 &= \delta_{i,l}\delta_{j,k} K^{i,j}(v,\overline{w})(1 + \mO(h^{1/2})) 
 	+\mO(h^2).%
 \end{split}
\end{equation}
Here $\delta_{i,j}$ denotes the Kronecker-delta and the most important part of the above formula is the kernel
\begin{equation}\label{eq7.9.1}
K^{i,j}(v,\overline{w}) = 
\exp\Big(\frac{1}{2}\big(\sigma_+^j(z_0)+\sigma_-^i(z_0)\big)\,v\overline{w}\Big).
\end{equation}
The other ingredients are the "gauge functions"
\begin{equation}\label{e:F_ij}
	F_{i,j}(v;h) = \phi_{+}^j(v;h) +\phi_{-}^i(v;h), 
\end{equation}
where
\begin{equation}\label{e:phi^j}
	\phi_{\pm}^j(v;h) = \frac{1}{2}\left[ \log  A_{\pm}^j(z_0;h)+ 
				(\partial^2_{zz}\Phi_{{\pm},0}^j)(z_0)v^2
	\right],
\end{equation}
with $A_{\pm}^j(z_0;h)\sim A_0^{j,\pm}(z_0)+hA_1^{j,\pm}(z_0)+\dots $ as 
in Proposition \ref{prop4.2} and $\Phi_{{\pm},0}^j(z)$ as in \eqref{eq4.4.1}. 
\end{prop}
We already notice that the kernel $K^{i,j}(v,\overline{w})$ corresponds to the GAF in Theorem~\ref{thm_m1} (see \eqref{e:Gaf_ij}).
We recall from~\eqref{eq4.13.2} that $A_0^{j,\pm}(z_0)>0$. The above proposition implies the limit
\begin{equation}\label{eq7.10}
\erw\big[f^{\delta,h}_{i,j}(v) \overline{f^{\delta,h}_{l,k}(w)}\big]\e^{-F_{i,j}(v;h)-\overline{F_{l,k}(w;h)}} 
 \xrightarrow{h\to 0}
 \delta_{i,l}\delta_{j,k} \,K^{i,j}(v,\overline{w}),
\end{equation}
uniformly for $v,w\in O$. The exponentials $\e^{F_{i,j}(\bullet;h)}$ 
have no influence on the zeros of the random function $f^{\delta,h}_{i,j}(\bullet)$, 
we can therefore see them as irrelevant gauge functions, which we shall get rid of later. 
 \begin{proof}[Proof of Proposition \ref{prop7.1}]
For $v,w\in O$, we have by \eqref{eq6.3.1b}, \eqref{eq1.10}, 
\begin{equation}\label{eq7.5}
\begin{split}
	hK^{i,j,l,k}_h(v,\overline{w})&\stackrel{\mathrm{def}}{=}
	h\erw \left[f^{\delta,h}_{i,j}(v)\overline{f^{\delta,h}_{l,k}(w)} \right] \\
	&= 
	\sum_{n,m,n',m'} \erw\left[q_{n,m}^h\,\overline{q_{n',m'}^h}\right] 
	\zeta_{n,m}^{i,j}(v)\,\overline{\zeta_{n',m'}^{l,k}(w)},
\end{split}
\end{equation}
where all indices are summed in the range $[0,N(h))$, and we use the notation
\begin{equation}\label{eq7.5.1}
	\zeta_{n,m}^{i,j}(w) = (e_+^{j,hol}(z_w)|e_m)
					 (e_n| e_-^{i,hol}(z_w)),
\end{equation}
where we the notation $z_w=z_0+h^{1/2}w$, $z_v=z_0+h^{1/2}v$ as in \eqref{e:rescaling}. By \eqref{eq7.3}, \eqref{eq7.4} and the independence of the $q_{n,m}$,  the expression \eqref{eq7.5} is equal to 
\begin{equation*}
\begin{split}
	&(1+\mO(h^{2+\varepsilon_0}))\left(\sum_{m}(e_+^{j,hol}(z_v)|e_m)
			 (e_m| e_+^{k,hol}(z_w))   \right) 
	\left(\sum_{n} ( e_-^{l,hol}(z_w)|e_n)
			      (e_n| e_-^{i,hol}(z_v))  \right) \\
	& + \mO(h^{6+2\varepsilon_0}) \sum_{n,m,n',m'} |(e_+^{j,hol}(z_v)|e_m)|\,
		| (e_n|e_-^{i,hol}(z_v))| \,
	|( e_+^{k,hol}(z_w)|e_{m'})|\,
	| (e_{n'} | e_-^{l,hol}(z_w))|.
	\end{split}
\end{equation*}
From Lemma \ref{lem7.1},  
\begin{equation*}
	 \sum_{n<N(h)} |(e_+^{j,hol}(z_v)|e_n)| \leq 
	 N(h)^{1/2}\| e_+^{j,hol}(z_v)\| \left(1+ \mO(h^{\infty})\right).
\end{equation*}
Since $N(h)=\mO(h^{-2})$, we obtain by another application of Lemma \ref{lem7.1}, that 
\begin{equation}\label{eq7.0.6}
\begin{split}
	hK^{i,j,l,k}_h(v,\overline{w})= 
	&
	(e_+^{j,hol}(z_v)| e_+^{k,hol}(z_w)) \,
	(e_-^{l,hol}(z_w)| e_-^{i,hol}(z_v)) \\
	& + \mO(h^{2+\varepsilon_0}) 
	\| e_+^{j,hol}(z_v)\|\,\|  e_-^{i,hol}(z_v)\| \,
	\| e_+^{k,hol}(z_w)\|\,\|e_-^{l,hol}(z_w)\|.
\end{split}
\end{equation}
Using Proposition \ref{prop4.2}, Proposition \ref{prop4.3}, \eqref{eq7.0.6} and 
\eqref{eq4.6.1b}, we get the following expression:%
\begin{equation}\label{eq7.6}
\begin{split}
hK^{i,j,l,k}_h(v,\overline{w}) = 
& \delta_{i,l}\delta_{j,k} \,
 \e^{\frac{2}{h} \big(\Psi^j_+(z_v,z_w;h) + 
\Psi^i_-(z_w, z_v;h) \big)} \\
&+\mO(h^{2+\varepsilon_0}) \e^{\frac{1}{h}(\Phi^j_+(z_v;h) + 
\Phi^k_+(z_w;h) +
\Phi^i_-(z_v;h) +
\Phi^l_-(z_w;h))}.
 \end{split}
\end{equation}
\\
\par
Now, recall from \eqref{eq4.2.0}, \eqref{eq4.4.1} that the phase functions $\Phi^j_{\pm,0}$ had been "centered" at the point $z_0$, so that
$$
\Phi_{\pm,0}^j(z_0)=0,\quad \partial_z\Phi_{\pm,0}^j(z_0)=0,\quad
(\partial_{\bar{z}}\Phi_{\pm,0}^j)(z_0) = 0.
$$
Taking into account that $\partial^\alpha \log A^j_{\pm}(\bullet;h)=\mO(1)$ in $W(z_0)$, the Taylor expansion of $\Phi_{\pm}^j(\bullet;h)$ around $z_0$ gives
\begin{equation}\label{eq7.7}
	\frac{1}{h}\Phi_{\pm}^j(z_v;h) = \log (h^{1/4}) + 
	(\partial^2_{z\bar{z}}\Phi_{\pm,0}^j)(z_0)v\overline{v} 
	+ {\phi}_{\pm}^j(v;h) + \overline{{\phi}_{\pm}^j(v;h)}
	+\mO(h^{1/2}),
\end{equation}
where we used the notation \eqref{e:phi^j}.
\par
Similarly, by Proposition \eqref{prop4.3}, we have:%
\begin{align}\label{eq7.9}
	\frac{1}{h}\Psi_{+}^j(z_v, z_w;h) 
	&= \log (h^{1/4}) + (\partial^2_{z\bar{z}}\Phi_{+,0}^j)(z_0)v\overline{w} 
	+ {\phi}_{\pm}^j(v;h) + \overline{{\phi}_{+}^j(w;h)}
	+\mO(h^{1/2})\nonumber\\ 
	&= \log (h^{1/4}) +\frac{1}{4}\sigma_+^j(z_0)v\overline{w} 
	+{\phi}_{+}^j(v;h) + \overline{{\phi}_{+}^j(w;h)}
	+\mO(h^{1/2}),\\
\intertext{where in the second line we used \eqref{eq4.4.4}. We also have}
\frac{1}{h}\Psi_{-}^j(z_w, z_v;h) &= \log (h^{1/4}) + \frac{1}{4}\sigma_-^j(z_0) \overline{w} v 
+ \overline{{\phi}_{-}^j(w;h)}
	+{\phi}_{-}^j(v;h) 
	+\mO(h^{1/2}).
\end{align}
In all estimates the error terms are uniform in $z,w \in O$. 
Thus, combining \eqref{eq7.6} with (\ref{eq7.7}-\ref{eq7.9}) and using the notation \eqref{e:F_ij}, we obtain
$$
K^{i,j,l,k}_h(v,\overline{w}) = \delta_{i,l}\delta_{j,k} e^{\frac12( \sigma_+^j(z_0) + \sigma_-^i(z_0)  ) v\overline{w}}\ e^{F_{i,j}(v;h) + \overline{F_{l,k}(w;h)} + \mO(h^{1/2})} + \mO(h^{2+\varepsilon_0}).
$$
Notice that we used the factor $h$ on the left hand side of \eqref{eq7.6} was facing four factors $h^{1/4}$ on the right hand side, so removed them all. 
This estimate is equivalent with the equation \eqref{e:covar} of the proposition.
\end{proof}

\subsection{Tightness}\label{sec:Tightness1}
In view of Proposition \ref{prop6.1}, we will now show that the family 
of random analytic functions $(F^{\delta}_h(\bullet))_{h\to 0}$ on $O$, defined in \eqref{eq6.3.1}, is tight, namely the function $F^{\delta}_h$ has a small probability to be large on $O$, uniformly w.r.t. $h\to 0$. 
\begin{prop}
There exists $h_0>0$ such that the family of random analytic functions $(F^\delta_h(\bullet))_{0<h\leq h_0}$ on the open set $O$, defined in \eqref{eq6.3.1}, is tight, in the sense used in Prop.~\ref{prop6.1}. 
\end{prop}
\begin{proof}
Recall the estimates \eqref{eq5.2.12} for the remainders $R_1$, $R_2$; for $z\in O$, we get $|z-z_0|=\mO(h^{1/2})$, so all terms $|z-z_0|^\infty$ are negligible. Besides, the expansion \eqref{eq7.7} for $\frac{1}{h}\Phi_{\pm}^j(z_v;h)$ implies that 
  \begin{equation}\label{eqNL1}
	\widetilde{R}_1(v;h)= \mO(\delta h^{-3/2}), \qquad
  	\widetilde{R}_2(v;h) =\mO(\delta h^{-5/2}),
  \end{equation}
uniformly in $v\in O$ and in $q\in\mathrm{PD}_{N(h)^2}(0,C/h)$. %
\\
\par
Let $K\Subset O$ be some compact subset. For $\varepsilon>0$, let 
$K_{\varepsilon} = K + \overline{D(0,\varepsilon)}$ be the closure of an 
$\varepsilon$-neighbohood of $K$. Pick $\varepsilon>0$ small enough 
so that $K_{\varepsilon}\Subset O$. By Proposition \ref{prop7.1}, we 
have for $h_0>0$ small enough:
\begin{equation}\label{eq7.11}
\sup\limits_{0<h\leq h_0} \erw \left[ \| f_{i,j}^{h,\delta} \|^2_{L^2(K_{\varepsilon})} \right]  
\leq C(K_{\varepsilon}) < +\infty.
\end{equation}
Since $F_h^{\delta}$ is holomorphic, Hardy's convexity theorem \cite[Lemma 2.6]{Sh12} 
implies that there exists a positive constant $C_{K_{\varepsilon}}>0$ depending only on
 $K_{\varepsilon}$ so that for any $p>0$
\begin{equation}\label{eq7.13b}
\| F_h^{\delta} \|^p_{L^{\infty}(K)}   \leq   
C_{K_{\varepsilon}}  \int_{K_{\varepsilon}}  |F_h^{\delta}(w) |^p L(dw). 
\end{equation}
Notice that 
\begin{equation}\label{e:det-HS}
	|\det ( (f_{i,j}^{h,\delta}(w))_{i,j}+\widetilde{R}_2(w;h))|^2 
	\leq  \| (f_{i,j}^{h,\delta}(w))_{i,j}+\widetilde{R}_2(w;h)\|_{HS}^{2J}.
\end{equation}
Markov's inequality, together with \eqref{eqNL1} and \eqref{eq6.3.1}, then shows that 
for $h_0>$ small enough 
\begin{equation}\label{eq7.13c}
\begin{split}
 \sup\limits_{0<h<h_0} &\prob \left[\| F_h^{\delta} \|^2_{L^{\infty}(K)} >r \right]  \\
 &=\sup\limits_{0<h<h_0} \prob\left[\| F_h^{\delta} \|^{2/J}_{L^{\infty}(K)} >r^{1/J} \right] \\
 & \leq \sup\limits_{0<h<h_0} r^{-1/J} \erw \left[\| F_h^{\delta} \|^{2/J}_{L^{\infty}(K)} \right]  \\
&\leq \frac{2 C_{K_{\varepsilon}}}{ r^{1/J} }\sup\limits_{0<h<h_0}  
\erw \left[ \|1+\widetilde{R}_1\|^{2/J}_{L^{\infty}(K_{\varepsilon})}
 \int_{K_{\varepsilon}} \big( \| (f_{i,j}^{h,\delta}(w))_{i,j\leq J}\|_{HS}^{2}+\|\widetilde{R}_2(w;h)\|_{HS}^{2} \big)L(dw) \right ] \\
&\leq \mO(r^{-1/J}).
\end{split}
\end{equation}
This concludes the proof that the family $(F_h^{\delta})_{0<h\leq h_0}$ is tight.%
\end{proof}

For future use, let us remark that similarly to the above, we obtain that 
for any $w_1, \dots, w_k \in O$, for any $t >0$ and $h>0$ small enough:
\begin{equation}\label{eq7.13d}
\begin{split}
 \prob \bigg[ \sum_{r=1}^k|\det \big( (f_{i,j}^{h,\delta}(w_r))_{i,j} +\widetilde{R}_2(w_r;h) \big)|^2 >t \bigg]  
 &\leq \mO(t^{-1/J}k^{1/J})  \erw \big[\| F_h^{\delta} \|^{2/J}_{L^{\infty}(O)}  \big] \\
&\leq \mO(t^{-1/J}k^{1/J}) .
\end{split}
\end{equation}
Taking the $L^{\infty}$ norm over $O$ is justified since we can arrange that, 
for $h>0$ small enough, $F_h^{\delta}$ is well-defined and holomorphic in a slightly larger set. 

\subsection{Weak convergence to a Gaussian analytic function}
\label{s:weak_convergence}
Next we will show that the random analytic function $F_h^{\delta}$ converges in distribution to a {\em Gaussian} analytic 
function when $h\to 0$. By Prop.~\ref{prop6.1} and Section~\ref{sec:Tightness1}, 
it is sufficient to prove the convergence of the finite-dimensional distributions of $F_h^{\delta}$. 

We begin with the following result.
\begin{prop}\label{lem6.3.3}
Let $f_{i,j}^{h,\delta}$ be as in \eqref{eq6.3.1b}, and let 
$K^{i,j}(z,\overline{w})$ be as in \eqref{eq7.9.1}. Then 
\begin{equation}\label{eq7.14}
	(f_{i,j}^{h,\delta}; 1\leq i,j \leq J) \stackrel{fd}{\longrightarrow} 
	(f_{i,j}^{GAF}; 1\leq i,j \leq J), \quad h\to 0,
\end{equation}
where $f_{i,j}^{GAF}$ are independent Gaussian analytic functions on $ O$ 
with covariance kernel 
\begin{equation}\label{eq7.14.1}
	K^{i,j}(v,\overline{w})\,\e^{F_{i,j}(v;0)+\overline{F_{i,j}(w;0)}}, \quad z,w\in O.
\end{equation}
Here $K^{i,j}$ was defined in \eqref{eq7.9.1}, and the function $F_{i,j}(w;0)$ is the limit as $h\to 0$ of the function defined in \eqref{e:F_ij}, \eqref{e:phi^j}.
\end{prop}
Before we give the proof of this Proposition, we draw an immediate consequence: 
\begin{cor}\label{cor7.3}
Let $\widetilde{R}_2$ be as in \eqref{eq6.3.1b}. Then, under the assumptions of 
Proposition \ref{lem6.3.3}, we have that 
\begin{equation}\label{eq7.14.b}
	\det \left( (f_{i,j}^{h,\delta})_{i,j}+\widetilde{R}_2\right)%
	\stackrel{fd}{\longrightarrow} 
	\det \left( (f_{i,j}^{GAF})_{i,j}\right)%
	, \quad h\to 0.
\end{equation}
%
\end{cor}
\begin{proof}
From Definition \ref{def_CD} it is immediately clear that the notion of convergence 
in distribution is preserved by composition with a continuous function between 
complete separable metric spaces. This is otherwise known as the continuous 
mapping theorem, see e.g. \cite{Kal97}. The first statement is then an immediate 
consequence of 
\begin{equation}\label{cor7.3e_1}
	(f_{i,j}^{h,\delta}+(\widetilde{R}_2)_{i,j}; 1\leq i,j \leq J) \stackrel{fd}{\longrightarrow} 
	(f_{i,j}^{GAF}; 1\leq i,j \leq J), \quad h\to 0.
\end{equation}
This follows immediately from Proposition~\ref{lem6.3.3} and \eqref{eqNL1}, and 
the following easy Lemma, which states that adding decaying perturbations to a converging sequence of random variables does not change the limit.
\begin{lem}
Let $(X_n)_{n\in\N}$ and $X$ be random vectors in $\C^N$ such that $X_n\stackrel{d}{\to} X$, 
and let $(R_n)_{n\in\N}$ be a sequence 
of random vectors in $\C^N$ such that
\begin{equation}\label{q1}
	\forall t>0\qquad \lim\limits_{n\to+\infty} \prob [ |R_n| > t ] = 0. 
\end{equation}
Then $X_n+R_n\stackrel{d}{\to} X$.
\end{lem}
\begin{proof}
We first check that $\{X_n\}_{n\in\N}$ is a tight sequence of random vectors. Indeed, for $r>1$ consider the continuous function  
$\C^N \ni x\mapsto g(x)\defeq (1-(r - |x|)_+)_+$. We then have 
\begin{equation*}
\begin{split}
	\limsup\limits_{n\to+\infty} \prob [ |X_n| > r ] 
	&=\limsup\limits_{n\to+\infty} \erw [ \mathbf{1}_{\{|X_n| > r\}} ]  \\
	&\leq \limsup\limits_{n\to+\infty} \erw [ g(|X_n|) ] 
	= \erw [ g(|X|) ]  
	\leq \prob [ |X| > r-1 ] .
\end{split}
\end{equation*}
Since $\prob [ |X| > r ] \to 0 $ as $r\to +\infty$ it follows that $\{X_n\}_{n\in\N}$ is  tight. 
Next, notice that 
\begin{equation*}
\begin{split}
	 \prob [ |X_n+R_n| > r ] &\leq \prob [ |X_n|+|R_n| > r ] \\
	 & \leq  \prob [ |X_n|+|R_n| > r \big| |R_n| > 1] \times  \prob [ |R_n| > 1  ]  \\
	 & \ \ \ + \prob [ |X_n|+|R_n| > r \big| |R_n| \leq 1  ] ]  \times \prob [ |R_n| \leq 1  ] \\
	 &\leq   \prob [ |R_n| > 1  ]  + \prob [ |X_n| > r -1 ]
\end{split}
\end{equation*}
where $\prob[A \big| B]$ denotes the conditional probability of $A$ conditioned on $B$. 
Since $(X_n)_n$ is tight, it follows by \eqref{q1} that $(X_n + R_n)_n$ is tight as well. 

Let us show that $X_n+R_n \stackrel{d}{\rightarrow} X$ as $n\to +\infty$. 
By Definition \eqref{def_CD} and Remark \eqref{Rem6.1} it is sufficient to show convergence 
of the induced probability measures in the $w^*$ topology of $\mathcal{C}_c(\C^N)'$. 
So take $\phi \in \mathcal{C}_c(\C^N;\R)$, and choose $\varepsilon >0$. Since $\phi$ is 
uniformly continuous there exists a $t >0$ such that for any $|X-Y|< t$ we have 
$|\phi(X)-\phi(Y)|< \varepsilon/2$. Let $G_n$ denote the event that $|R_n|>t$ and 
let $G_n^c$ denote its complement. By 
\eqref{q1} there exists $n_\varepsilon>0$ such that for any $n\geq n_{\varepsilon}$,  
$\prob [G_n] <\varepsilon/(2\|\phi\|_{\infty})$.  
Then, for $n\geq n_{\varepsilon}$,
\begin{equation*}
\begin{split}
	\erw[ |\phi (X_n + R_n) - \phi(X_n)|] & \leq 
	\erw[ |\phi (X_n + R_n) - \phi(X_n)| \mathbf{1}_{G_n}]  \\
	&~+ \erw[ |\phi (X_n + R_n) - \phi(X_n)|  \mathbf{1}_{G_n^c}] \\ 
	&\leq \varepsilon
\end{split}
\end{equation*}
This shows that $X_n+R_n \stackrel{d}{\rightarrow} X$ as $n\to +\infty$. 
\end{proof}

If we consider $N$-tuples of random elements of the form $f_{i,j}^{h,\delta}(w_k)+\widetilde{R}_2(w_k)$,   Proposition~\ref{lem6.3.3} and estimate \eqref{eqNL1}, proves  \eqref{cor7.3e_1}, hence the Corollary. \end{proof}
\begin{proof}[Proof of Proposition \ref{lem6.3.3}]
The proof is an adaptation of the proof of \cite[Theorem 4.4]{Sh12}. We want to show 
that for any $M\in\N^*$ and all $w_1,\dots,w_M \in  O$ 
\begin{equation}\label{eq7.14.0}
	(f_{i,j}^{h,\delta}(w_l); 1\leq i,j \leq J,~ 1\leq l \leq M) \stackrel{d}{\longrightarrow} 
	(f_{i,j}^{GAF}(w_l); 1\leq i,j \leq J,~ 1\leq l \leq M), \quad h\to 0,
\end{equation}
By the Cram\'er-Wold Theorem \cite[Corollary 4.5]{Kal97}, to prove \eqref{eq7.14.0} it suffices 
to show that for all $\lambda=(\lambda^{i,j}_l; 1\leq i,j \leq J, 1\leq l \leq M) \in \C^{J^2\cdot M}$,
the complex valued random variable 
$$
 S(\lambda) \defeq \sum_{i,j,l} \lambda^{i,j}_lf_{i,j}^{h,\delta}(w_l) 
 $$
converges in distribution, as $h\to 0$, to the complex Gaussian random variable
\begin{equation}\label{eq7.14.00}
 	S^{GAF}(\lambda)\defeq \sum_{i,j,l} \lambda^{i,j}_l f_{i,j}^{GAF}(w_l).
\end{equation}
Let us write $S(\lambda)$ in terms of our restricted complex random variables $q^h_{n,m}$:
\begin{equation*}
 S(\lambda) = \sum_{n,m<N(h)} q^h_{n,m} G_{n,m},  \quad 
		   G_{n,m} = \sum_{i,j,l}\lambda^{i,j}_l h^{-1/2} \zeta_{n,m}^{i,j}(w_l),
\end{equation*}
where we used the notation of \eqref{eq7.5.1}. To prove the limit $S(\lambda)\stackrel{d}{\to} S^{GAF}(\lambda)$, we will use the central limit theorem expressed in Theorem~\ref{thmCLT}. 
We thus need to check that the family of random variables $(S(\lambda))_{0<h\leq h_0}$ satisfies the four conditions stated in the theorem.
\par
Let us first estimate the average of $S(\lambda)$: from \eqref{eq7.3}, we get
\begin{align}\label{eee1}
\sum_{n,m<N(h)} &|\erw [q^h_{n,m} G_{n,m}]| 
\leq \mO(h^{3+\vareps_0}) \sum_{n,m<N(h)} |G_{n,m}|\\
&\leq \mO(h^{3+\vareps_0}) \sum_{i,j,l}\sum_{m<N(h)} \left|(h^{-1/4}e_+^{j,hol}(z_{w_l})|e_m)\right| 
	\sum_{n<N(h)} \left|( h^{-1/4} e_-^{i,hol}(z_{w_l})|e_n)\right|\\
&\leq  \mO(h^{3+\vareps_0}) N(h) \sum_{i,j,l}\| h^{-1/4}e_+^{j,hol}(z_{w_l}) \| \,\|( h^{-1/4} e_-^{i,hol}(z_{w_l})\|\\
&\leq  \mO(h^{1+\vareps_0})\,.
\end{align}
In the third line we used two Cauchy-Schwarz inequalities, and in the last one the fact that the states $\| h^{-1/4} e_{\pm}^{j,hol}(z_{w})\|=\mO(1)$ uniformly when $w\in O$. This proves the point $(i)$ in Theorem~\ref{thmCLT}.
Let us now check condition $(ii)$ of Theorem \ref{thmCLT}. 
Using \eqref{e:E(q2)}, we draw the bound
\begin{equation}\label{eee2}
\begin{split}
\left| \sum_{n,m<N(h) }\erw[(q_{n,m}^h)^2G_{n,m}^2] \right|
&=\sum_{n,m<N(h)}  \left|\erw[(q_{n,m}^h)^2]G_{n,m}^2\right|\\
&\leq \mO(h^{2+\vareps_0}) \sum_{m,n} |G_{n,m}|^2 \\
& \leq  \mO(h^{2+\vareps_0}) \sum_{i,j,l}\sum_{m<N(h)} \left|(h^{-1/4}e_+^{j,hol}(z_{w_l})|e_m)\right|^2
	\sum_{n<N(h)} \left|( h^{-1/4} e_-^{i,hol}(z_{w_l})|e_n)\right|^2\\
& \leq \mO(h^{2+\vareps_0}) \,.
\end{split}
\end{equation}
The variance of $S(\lambda)$ needs more care. By \eqref{eq7.3} and \eqref{eq7.4},
\begin{equation*}
	\erw \bigg[\sum_{n,m<N(h)} |q^h_{n,m}G_{n,m}|^2 \bigg]
	= (1+\mO(h^{2+\varepsilon_0}))\sum_{i,j,l,r,s,t }\lambda^{i,j}_l\overline{\lambda^{r,s}_t}
	\sum_{n,m<N(h)} h^{-1}\zeta_{n,m}^{i,j}(w_l)\overline{\zeta_{n,m}^{r,s}(w_t)}
\end{equation*}
which, by \eqref{eq7.5.1} and Lemma \ref{lem7.1} is equal to 
\begin{equation}\label{eq7.15}
	\begin{split}
	(1+\mO(h^{2+\varepsilon_0}))\sum_{i,j,l,r,s,t }\lambda^{i,j}_l\overline{\lambda^{r,s}_t}
	h^{-1}&\left(e_+^{j,hol}(z_{w_l})|
			  e_+^{s,hol}(z_{w_t})  \right)\\
	&\times\left(e_-^{r,hol}(z_{w_t})|
			   e_-^{i,hol}(z_{w_l})  \right) + \mO_\lambda(h^\infty).
       \end{split}
\end{equation}
Here, we used as well $ \| e_{\pm}^{k,hol}(z_{w})\|= \mO(h^{1/4})$ uniformly for $w\in O$.

By \eqref{eq4.6.1b} we see that the terms in \eqref{eq7.15} with $i\neq r$ and $j\neq s$ 
are $\mO(h^\infty)$. Similar to the proof of Proposition \ref{prop7.1}, we then 
obtain as $h\to 0$:
\begin{equation}\label{eq7.16}
	\erw \left[\sum_{n,m<N(h)} |q^h_{n,m}G_{n,m}|^2 \right]
	\longrightarrow \sum_{i,j,l,t }\lambda^{i,j}_l\overline{\lambda^{i,j}_t} 
	K^{i,j}(w_l,w_t)\,\e^{F_{i,j}(w_l;0)+\overline{F_{i,j}(w_t;0)}}
	\defeq\sigma(\lambda,w).
\end{equation}
\par
In order to check the condition $(iv)$ of Thm~\ref{thmCLT}, we compute the $(4+\vareps_0)$-moment of $S(\lambda)$.  
From \eqref{eq1.9} and the H\"older inequality, we get
\begin{equation*}
\begin{split}
	&\sum_{n,m<N(h)} \erw \left[|q^h_{n,m}G_{n,m}|^{4+\varepsilon_0} \right]
	\leq C_{\lambda}  \sum_{i,j,l}
	\sum_{n,m<N(h)} | h^{-1/2} \zeta_{n,m}^{i,j}(w_l)|^{4+\varepsilon_0} \\ 
	&= C_{\lambda}  \sum_{i,j,l}
	\left(\sum_{m<N(h)}\left|(h^{-1/4}e_+^{j,hol}(z_{w_l})|e_m)\right|^{4+\varepsilon_0}  \right) 
	\left(\sum_{n<N(h)} \left|( h^{-1/4} e_-^{i,hol}(z_{w_l})|e_n)\right|^{4+\varepsilon_0} \right).
\end{split}
\end{equation*}
Splitting $|\bullet|^{4+\vareps_0}=|\bullet|^2|\bullet|^{2+\vareps_0}$, using the fact that $\{e_m\}$ forms an orthonormal basis of $L^2(\R)$ and the bound \eqref{e:bound-overlap} of Lemma \ref{lem7.2}, one obtains
\begin{equation*}
\begin{split}
	\sum_{n,m<N(h)}&\erw \big[  |q^h_{n,m}G_{n,m}|^{4+\vareps_0} \big] \\ 
	&= \mO(h^{1+\varepsilon_0/2}) \sum_{i,j,l} 
	\left\|h^{-1/4} e_+^{j,hol}(z_{w_l}))\right\|^2
	\left\|h^{-1/4}e_-^{i,hol}(z_{w_l})\right\|^2
	\e^{\frac{2}{h}\Phi_{-,0}^i(z_{w_l})}\e^{\frac{2}{h}
	\Phi_{+,0}^j(z_{w_l})}\\
	&= \mO(h^{1+\varepsilon_0/2}). 
\end{split}
\end{equation*}
In the last line we used the fact that both the norms and the exponentials are $\mO(1)$ uniformly for $w\in O$. 
\begin{rem} It is especially for the smallness of this $(4+\vareps_0)$-moment that we needed the overlaps $(h^{-1/4}e_{\pm}^{j,hol}(z_{w})|e_m)$ to all be small, and for this very reason that we chose our auxiliary basis $(e_m)_{m\in\N}$ to be composed of nonsemiclassical states.
\end{rem}
From this $(4+\vareps_0)$-moment, we verify the condition $(iv)$ of Thm~\ref{thmCLT}: for any $\varepsilon >0$
\begin{equation}\label{eq7.17}
 \sum_{n,m<N(h)} \erw \big[ |q^h_{n,m}G_{n,m}|^2 
 	\mathds{1}_{\{|q^h_{n,m}G_{n,m}|>\varepsilon\}} \big]
 < \varepsilon^{-(2+\varepsilon_0)}\sum_{n,m<N(h)} \erw \big[ |q^h_{n,m}G_{n,m}|^{4+\varepsilon_0} \big] 
 \longrightarrow 0,
\end{equation}
as $h\to 0$. Equipped with the estimates \eqref{eee1}, \eqref{eee2},\eqref{eq7.15}, \eqref{eq7.16}, \eqref{eq7.17}, we may apply the  version of the CLT given in Thm~\ref{thmCLT}, to show that 
$S(\lambda)$ converges in distribution to the complex Gaussian random variable $\mathcal{N}_{\C}(0,\sigma(\lambda,w)^2)$, with variance given in \eqref{eq7.16}. 
\\
\par
On the other hand, since the $(f_{i,j}^{GAF})_{i,j\leq J}$ are independent Gaussian analytic functions 
with covariance kernel \eqref{eq7.14.1}, it follows that 
\begin{equation}
	 \sum_{i,j,l} \lambda^{i,j}_l f_{i,j}^{GAF}(w_l) \sim \mathcal{N}_{\C}(0,\widetilde{\sigma}(\lambda,w))
\end{equation}
with
\begin{equation}
\begin{split}
\widetilde{\sigma}(\lambda,w)  &=\erw\left[\big| \sum_{i,j,l} \lambda^{i,j}_l f_{i,j}^{GAF}(w_l)\big|^2\right] 
= 
\sum_{i,j,l,r,s,t} \lambda^{i,j}_l \overline{\lambda^{r,s}_t}\,\erw\left[ f_{i,j}^{GAF}(w_l) \overline{f_{r,s}^{GAF}(w_t)} \right] \\
	& = \sum_{i,j,l,t }\lambda^{i,j}_l\overline{\lambda^{i,j}_t} 
	K^{i,j}(w_l,w_t)\,\e^{F_{i,j}(w_l;0)+\overline{F_{i,j}(w_t;0)}} \\
	& = \sigma(\lambda,w),
\end{split}
\end{equation}
where the second line uses the independence of the different $f_{i,j}^{GAF}$. Since a complex Gaussian random variable is uniquely determined by its expectation and its variance, we conclude 
that $S(\lambda)\stackrel{d}{\to} S^{GAF}(\lambda)$, hence the Proposition.
\end{proof}
Next, we show that the finite dimensional distributions of the random functions 
$\det ( (f_{i,j}^{h,\delta})_{i,j}+\widetilde{R}_2)$ and $F^{\delta}_h$ (see \eqref{eq6.3.1}) converge in distribution to the same limit. 
\par
From \eqref{eq7.13b}, \eqref{eq7.13c}, it is clear that 
for any $n\in \N^*$ and any $ w = (w_1,\dots,w_n) \in O^n$, the family of random vectors in $\C^n$:
\begin{equation*}
	\big(\widetilde{F}^{\delta}_h(w) \big)_{h\in (0,h_0]}\stackrel{\mathrm{def}}{=} \big(F^{\delta}_h(w_1),\dots,F^{\delta}_h(w_n)\big)_{h\in (0,h_0]}, 
\end{equation*}
is tight. 
\\
\par
Next, let $n\in\N^*$ and $\phi\in\mathcal{C}_c(\C^n,\R)$. Set for $w = (w_1,\dots,w_n) \in O^n$:
\begin{equation*}
	\widetilde{g}^{\delta}_h(w)\stackrel{\mathrm{def}}{=}
	\big(g^{\delta}_h(w_1),\cdots ,{g}^{\delta}_h(w_n)\big),\qquad g^{\delta}_h(w_l)=\det \big((f_{i,j}^{h,\delta}(w_l))_{i,j}+\widetilde{R}_2(w_l;h)\big).\end{equation*}
Let $G_h$ denote the event that $\| \widetilde{g}^{\delta}_h(w)\|_{\ell^2([1,n])}^2 \leq Ch^{-1}$ 
and let $G_h^c$ denote its complement. Then, 
\begin{equation*}
\begin{split}
	\big|\erw [ \phi(\widetilde{F}^{\delta}_h(w))] - \erw [ \phi(\widetilde{g}^{\delta}_h(w))] \big| 
	&\leq \erw [ | \phi(\widetilde{F}^{\delta}_h(w)) -\phi(\widetilde{g}^{\delta}_h(w)) | 
	{G_h} ] \times \prob[G_h]\\
	&
	\phantom{\leq}
	+  \erw [ | \phi(\widetilde{F}^{\delta}_h(w)) -\phi(\widetilde{g}^{\delta}_h(w)) | 
	{G_h^c} ]\times \prob[G_h^c].
\end{split}
\end{equation*}
Since $\phi$ is bounded, it follows by \eqref{eq7.13d} that the second term is 
of order $\mO(h^{1/J})$. Suppose that $G_h$ holds. Then, by \eqref{eqNL1}, 
\begin{equation*}
\begin{split}
	 \|\widetilde{F}^{\delta}_h(w)- \widetilde{g}^{\delta}_h(w)\|_{\ell^2([1,n])} &\leq \left(\mO(h^{\infty}) + \mO(\delta h^{-5/2})\right)
	   	  \| \widetilde{g}^{\delta}_h(w)\|_{\ell^2([1,n])}  \\ 
          &\leq \mO(\delta h^{-3})
\end{split}
\end{equation*}
Since $\phi$ is uniformly continuous, we conclude that for any $\varepsilon>0$ 
there exists $h_0>0$ such that for all $0<h<h_0$ 
\begin{equation}\label{eq7.18}
	\big| \erw [ \phi(\widetilde{F}^{\delta}_h(z))] - \erw [ \phi(\widetilde{g}^{\delta}_h(z))] \big| \leq \varepsilon.
\end{equation}
Since $(\widetilde{F}^{\delta}_h(z))_{h\leq h_0}$ is a tight family of random vectors, it follows from
Remark \ref{Rem6.1}, \eqref{eq7.18} and Corollary \ref{cor7.3} that the random function
\begin{equation}\label{eq7.18.1}
	F^{\delta}_h(\bullet) \stackrel{fd}{\longrightarrow} \det \big((f_{i,j}^{GAF}(\bullet))_{1\leq i,j \leq J}\big) 
	\quad 
	\text{when }h\to 0. 
\end{equation}
Moreover, since $F^{\delta}_h(\bullet)$ is a tight sequence of random analytic functions, see Section 
\ref{sec:Tightness1}, Prop.~\ref{prop6.1} and \eqref{eq7.18.1} imply that 
\begin{equation}\label{eq7.19}
	F^{\delta}_h(\bullet) \stackrel{d}{\longrightarrow} \det \big((f_{i,j}^{GAF}(\bullet))_{1\leq i,j \leq J}\big) 
	\quad 
	\text{when }h\to 0.
\end{equation}
We recall that the GAFs $ f_{i,j}^{GAF}$ were defined in Proposition \ref{lem6.3.3}, their covariance kernel is
given by 
\begin{equation*}
K^{i,j}(v,\overline{w})\,\e^{F_{i,j}(v;0)}\e^{\overline{F_{i,j}(w;0)}}, 
\qquad 
F_{i,j}(v;0) = \phi_-^i(v;0) + \phi_+^j(v;0),
\end{equation*}
cf. \eqref{eq7.14.1}. Since $\e^{F_{i,j}(\bullet;0)}$ is a nonvanishing  
deterministic holomorphic function on $O$, we define the random analytic function
 \begin{equation*}
 	g_{i,j}(\bullet)\defeq f_{i,j}^{GAF}(\bullet)\,\e^{-F_{i,j}(\bullet;0)},\qquad i,j=1,\ldots,J.
 \end{equation*}
Then $\{g_{i,j}(\bullet);i,j=1,\ldots, J\}$ are independent GAFs on $O$, with covariance 
kernels $K^{i,j}(v,\overline{w})$. If we define the diagonal matrices
\begin{equation*}
	\Lambda_{\pm} =
	\mathrm{diag}( \e^{-\phi_{\pm}^i(z;0)})_{1\leq i\leq J },\qquad \phi_{\pm}^i\quad\text{as in Prop. \ref{prop7.1}},
\end{equation*}
we get the equality between random analytic functions:
$$
T\defeq \det \big( ( f_{i,j}^{GAF})_{i,j}\big) = \det( \Lambda_- \Lambda_+) \det \big( (g_{i,j} )_{i,j}\big) \defeq  \det( \Lambda_- \Lambda_+) G.
$$
Since, $ \det( \Lambda_- \Lambda_+)(\bullet)$ never vanishes and 
is deterministic, it follows by the continuous mapping theorem \cite{Kal97} and the discussion after 
\eqref{eq:PpP}, that the zero point processes
 \begin{equation*}
 	\cZ_{T}\defeq \sum_{\lambda\in T^{-1}(0)}  \delta_{\lambda} =
	\cZ_G.
 \end{equation*}
Hence, by \eqref{eq7.19} and Proposition \ref{prop6.6}, we see 
that
 \begin{equation}\label{eq7.19.1}
  	\cZ_{F^{\delta}_h}
	\stackrel{d}{\longrightarrow} 
	\cZ_{T}
	=\cZ_{G} 
	\text{ when } h\to 0.
 \end{equation}
Together with the discussion at the beginning Section \ref{sec:LS_M}
and the fact that $\cZ_{F^{\delta}_h}=\cZ^M_{h,z_0}$ represents the set of rescaled eigenvalues of $P^\delta$, this concludes the proof of Theorem~\ref{thm_m1}. 
%
\subsection{$k$-point measures}\label{sec:kptmes}
In this subsection we will show that the $k$-point measures $\mu_k$ of the point process $\mathcal{Z}_{h,z_0}^M$, 
defined in \eqref{eq2.11}, converge to the $k$-point measures $\mu_k$ of the 
point process $\mathcal{Z}_{G_{z_0}}$ defined in Theorem \ref{thm_m1}. We begin 
with a technical
\begin{lem}\label{lem7.5}
	Let $F_h^{\delta}(w)$ be as in \eqref{eq6.3.1}, and let $G_{z_0}(w)$ be as in 
	Theorem \ref{thm_m1}, with $w\in O$. Then, for any compact 
	$K\Subset O$ the numbers $n_F^h(K)$ (resp. $n_{G_{z_0}}(K)$) of zeros of $F_h^{\delta}$ (resp. $G_{z_0}$)
	in $K$ have exponential tail: there exist constants $C_1,C_2>0$ so that for any $h\leq h_0$,
	\begin{equation}\begin{split}
	\forall h\leq h_0,\qquad	\prob[ n_F^h(K) > \lambda ] &\leq C_1\e^{-C_2\lambda},\\
		\prob[ n_{G_{z_0}}(K) > \lambda ] &\leq C_1\e^{-C_2\lambda}.
	\end{split}\end{equation}
\end{lem}
\begin{proof} 
From  \cite[Theorem 3.2.1]{HoKrPeVi09}, to prove the first part of the Lemma it suffices to show that, for some $c>0,\ b>0$, the random analytic function $F_h^{\delta}$ satisfies the bounds
\begin{equation}\label{e:criterion}
\erw[|F_h^{\delta}(w)|^{\pm c}]  \leq b,\qquad \text{uniformly for $w\in O$ and $h\leq h_0$.}\,,
\end{equation}
and the second part requires a similar estimate for the random function $G_{z_0}(w)$.

Fix $w\in O$. Recall the bounds \eqref{eqNL1} on the matrices $\widetilde{R}_1(w)$, $\widetilde{R}_2(w)$. We start with the following uniform bound on the random variable $F^\delta_h(w)$:
\begin{equation}\label{eq7.30.1}
\begin{split}
	\erw[|F_h^{\delta}(w)|^{2/J}] 
	&\leq C_1 \erw[|\det \big((f_{i,j}^{\delta,h}(w))_{i,j}+\widetilde{R}_2(w)\big)|^{2/J}] \\
	& \leq C_1\erw[ \| (f_{i,j}^{\delta,h}(w))_{i,j} + \widetilde{R}_2(w)\|_{HS}^{2} ]  \\ 
	& \leq 2C_1\erw[ \| (f_{i,j}^{\delta,h}(w))_{i,j}\|_{HS}^{2} + \|\widetilde{R}_2(w)\|_{HS}^{2} ]  \\ 
	& \leq 2C_1 \sum_{i,j=1}^J \big[K^{i,j}(w,\bar{w})(1+\mO(\sqrt{h})\e^{2\Rea F_{i,j}(w;h)}  
		+ \mO(h^2)%
		\big] + C_2\delta^2 h^{-5}\\
	& \leq \mO(1),
\end{split}
	\end{equation}
where on the second line we used the inequality \eqref{e:det-HS}, and in the two last inequalities we used Proposition \ref{prop7.1}. The implied constants are uniform in $w\in O$ and $h\leq h_0$.%
\par
We also need to check that the inverse function $F_h^{\delta}(w)^{-1}$ is not too large on average: 
\begin{equation}\label{eq7.30}
\begin{split}
	\erw[|F_h^{\delta}(w)|^{-1/J}]  
	&= 
	\int_0^{\infty} \prob[ |F_h^{\delta}(w)|^{-1/J} \geq t] dt \\
	&= \int_0^{\infty} \prob[ |F_h^{\delta}(w)|^{1/J} \leq \tau] \tau^{-2} dt.
\end{split}
\end{equation}
From the convergence in distribution \eqref{eq7.19}, the continuous mapping theorem \cite[Theorem 3.27]{Kal97} 
and the Portmanteau Theorem \cite[Theorem 3.25]{Kal97},
\begin{equation}\label{eq7.31}
	\limsup_{h\to 0} \prob\big[ |F_h^{\delta}(w)|^{1/J} \leq \tau\big]  \leq 
	 \prob\big[ |\det (f_{i,j}^{GAF}(w))_{i,j}|^{1/J}  \leq \tau\big],
\end{equation}
for $h>0$ small enough. 
Recalling the definition of the functions $F_{ij}(w;0)$ appearing in Proposition \ref{lem6.3.3}, the $f_{i,j}^{GAF}(z)$ are independent 
centred complex Gaussian random variables with variances
\begin{equation*}
\begin{split}
	K^{i,j}(w,\bar{w}) \e^{2\Rea F_{i,j}(w;h)} &= 
	\left(\e^{\frac12\sigma^i_-(z_0)|w|^2 + 2\Rea \phi_-^i(w;0)}\right)
	\left(\e^{\frac12\sigma^j_+(z_0)|w|^2 + 2\Rea \phi_+^j(w;0)}\right) \\ 
	&\defeq \beta^i_-(w)\beta^j_+(w) \neq 0.
\end{split}
\end{equation*}
Using this factorization of the variances, we  set $\Lambda_{\pm}(w)=\mathrm{diag} (\beta_{\pm}^i(w); 1\leq i \leq J)$, which allows to factorize the limiting random determinant as follows:
\begin{equation}\label{eq7.31.0}
 \det (f_{i,j}^{GAF}(w)) = \det\Lambda_+(w) \det\Lambda_-(w)\, \det (v_1(w), \dots, v_J(w)),
\end{equation}
where, for $ j=1,\dots, J$, the column vector
\begin{equation}\label{eq7.31.1}
 v_j(w) = (v_{1j}(w),\dots, v_{Jj}(w))^{T},
\end{equation}
and $\{v_{ij}(w),\ 1\leq i,j\leq J\}$ are i.i.d.
complex Gaussian random variables with distribution $\mathcal{N}_{\C}(0,1)$. 
Until further notice we suppress $w$ in the notation. We view 
$v_j$ as a Gaussian random vector in $\C^J$ with the identity $\mathbf{1}$ as 
covariance matrix, i.e with distribution $\mathcal{N}_{\C}(0,\mathbf{1})$. Hence, the real variable $r=|v_{ij}|^2$ has the exponential distribution $f(r)dr$ with 
\begin{equation*}
	f(r) = \e^{-r}H(r), \quad H(r) = \mathds{1}_{[0,\infty[}(r).
\end{equation*}
The Fourier transform $\hat{f}(\rho)=\frac{1}{1+i\rho}$. Since the components $(v_{ij})_{i\leq J}$ are independent, 
the squared norm $|v_j|^2$ is distributed according to $f\ast \dots \ast f(r)dr=f^{\ast J}dr$ 
where $\ast$ is the convolution. A direct computation shows that 
\begin{equation*}
	f^{\ast J}(r) = \frac{r^{J-1}\e^{-r}}{(J-1)!}H(r), 
\end{equation*}
which is the $\chi^2_{2J}$ distribution in the variable $2r$. We now compute the law of $|\det (v_1, \dots, v_J)|$.
For this, we perform $J-1$ linear operations on the matrix $(v_1,\cdots,v_J)$, setting $\widetilde{v}_1=v_1$, and for any $j\geq 2$,
taking for $\widetilde{v}_j$ the orthogonal 
projection of $\widetilde{v}_j$ onto the orthogonal complement of the space 
$V_j \defeq \mathrm{span}_{\C}\{v_1,\dots,v_{j-1}\}$. Elementary linear algebra allows to write
\begin{equation*} 
	\begin{split}
	|\det (v_1, \dots, v_J) |&= |\det (\widetilde{v}_1, \dots, \widetilde{v}_J)| \\ 
		&= \prod_{j=1}^J|\widetilde{v}_j|,
	\end{split}
\end{equation*}
Once $v_1,\cdots,v_j$, hence $V_j$, are chosen, the vector
$\widetilde{v}_j$ is distributed as a complex Gaussian random vector in $V_j^\perp\equiv \C^{J-j+1}$, with 
distribution $\mathcal{N}_{\C}(0,\mathbf{1})$ (this follows from the fact that the 
initial distribution $\mathcal{N}_{\C}(0,\mathbf{1})$ is invariant under unitary transformations on $\C^J$). 
As a result $\{|\widetilde{v}_j|^2,\ 1\leq j\leq J\}$ are independent random variables with 
$\chi^2_{2(J-j+1)}$ distributions.
\par
Setting $\eta=\eta(w)\defeq (\det\Lambda_+(w) \Lambda_-(w))^{1/J}$ and using that the $|\widetilde{v}_j|^2$ 
are independent, we get by \eqref{eq7.31.0} and a straightforward computation that 
\begin{equation*}
\begin{split}
	\prob[ |\det (f_{i,j}^{GAF}(w))|^{1/J}  \leq \tau] 
	&= \prob[|\det (v_1, \dots, v_J) |^{1/J}  \leq \tau/\eta] \\
	& = \prob \Big[ \big(\prod_{j=1}^J|\widetilde{v}_j|^2\big)^{1/J} \leq ( \tau/\eta)^2 \Big] \\ 
	& = 1 -  \prob \Big[ \big(\prod_{j=1}^J|\widetilde{v}_j|^2)\big)^{1/J} > ( \tau/\eta)^2\Big] \\ 
	& \leq 1 - \prod_{j=1}^J
	 \prob \Big[ |\widetilde{v}_j|^2> ( \tau/\eta)^2\Big] \\
	 &= 1 -\e^{-\frac{\tau^2}{\eta^2}}\prod_{j=1}^{J-1}
	  \e^{-\frac{\tau^2}{\eta^2}}\sum_{k=0}^{J-j}\frac{(\tau/\eta)^{2(J-j-k)}}{(J-j-k)!}.
\end{split}
\end{equation*}
The sum on the right hand side is larger or equal to $1$, so we finally get:
\begin{equation*}
	\prob[ |\det (f_{i,j}^{GAF}(w))|^{1/J}  \leq \tau] 
	 \leq 1 -\e^{-\frac{\tau^2}{\eta^2}J} 
	\leq \frac{\tau^2}{\eta(w)^2}J .
\end{equation*}
Combining this with  \eqref{eq7.30}, \eqref{eq7.31} and splitting the integral into $[0,1]\cup  [1,\infty[$, we obtain
\begin{equation}\label{eq7.30.2}
\begin{split}
	\erw[|F_h^{\delta}(w)|^{-1/J}]  = \mO(1), \quad\text{uniformly in $w\in O$ and $h\leq h_0$.}
\end{split}
\end{equation}
Together with \eqref{eq7.30.1}, this proves the bounds \eqref{e:criterion} with $c=1/J$, hence the first part of the Lemma. An easy adaptation of the above computations shows that the function $G_{z_0}(w)$ satisfies similar bounds, and hence the second part of the Lemma.
\end{proof}
Using Lemma \ref{lem7.5}, we see that for all $\varphi\in\mathcal{C}_c(O,\R_+)$ and any 
$p>0$
\begin{equation*}
\begin{split}
	\erw[ |\langle \cZ_{F_h^{\delta}},\varphi \rangle|^p] 
	&\leq \|\varphi\|_{\infty}^p\, \erw[ (n_F^h(\supp\varphi))^p] \\ 
	&\leq \int_0^{+\infty} \prob [ (n_F^h(\supp\varphi))^p \geq t ] \, dt \\
	&\leq \, C_1 \int_0^{1} \e^{-C_2t^{1/p}} \, dt + 
		p\, C_1 \int_1^{+\infty} \e^{-C_2\tau} \tau^{p-1} \, d\tau < +\infty\,.
\end{split}
\end{equation*}
Hence, for all $\varphi_l\in\mathcal{C}_c(O,\R_+)$, $l=1,\dots, k$, the positive random variable
$\langle \cZ_{F_h^{\delta}},\varphi_1 \rangle \cdots \langle \cZ_{F_h^{\delta}},\varphi_k \rangle$ 
is integrable, uniformly w.r.t. $h\leq h_0$. By Proposition \ref{prop6.6} and \eqref{eq7.19.1} it then follows that 
\begin{equation*}
	\erw[ \langle \cZ_{F_h^{\delta}},\varphi_1 \rangle \cdots 
	\langle \cZ_{F_h^{\delta}},\varphi_k \rangle ] 
	\longrightarrow 
	\erw[ \langle \cZ_{G},\varphi_1 \rangle \cdots 
	\langle \cZ_{G},\varphi_k \rangle ],\quad \text{as }
	h\to 0.
\end{equation*}
Since $\bigotimes_{j=1}^k \mathcal{C}_c(O,\R_+)$ is dense in  $\mathcal{C}_c(O^k,\R_+)$, 
we have obtained the following
\begin{thm}\label{thm:kptmes1}
	Let $\mu_k^h$ (resp. $\mu_k$) be the $k$-point density measure (see \eqref{eq2.11}) of the point process $\mathcal{Z}_{h,z_0}^M$ (resp. $\mathcal{Z}_{G_{z_0}}$) defined in Theorem \ref{thm_m2}. Then, 
	for any $O\Subset\C$ open, relatively compact connected domain and 
	for any $\varphi\in\mathcal{C}_c(O^k\backslash\Delta,\R_+)$, 
	\begin{equation*}
		\int \varphi \,d\mu_k^h \longrightarrow \int \varphi \,d\mu_k, 
		\quad h\to 0.
	\end{equation*}
\end{thm}
To end this section on the matrix perturbations, let us check the formula \eqref{eq:MatDens} for the $1$-point density. The Lelong formula states that, in the sense of distributions, 
	\begin{equation*}
		\mathcal{Z}_{G_{z_0}} = \frac{1}{\pi} \partial_{\bar{w}}\partial_w \log|\det G_{z_0}|^2.
	\end{equation*}
Hence, the 1-point measure is given by
	\begin{equation}\label{e:Lelong}
		\mu_1(dw) = \frac{1}{\pi} \partial_{\bar{w}}\partial_w \erw [\log|\det G_{z_0}(w)|^2] L(dw). 
	\end{equation}
Recall from Theorem \ref{thm_m3} that for any fixed $w\in O$, the matrix elements $g_{i,j}(w)$ of $G_{z_0}(w)$ are complex Gaussian 
variables with variance $\e^{\frac{1}{2}(\sigma_+^j(z_0)+\sigma_-^i(z_0))|w|^2}$. 
If we consider the diagonal matrices
\begin{equation*}
	\tLambda_{\pm}(w) =
	\mathrm{diag}( \e^{\frac{1}{4}\sigma_{\pm}^i(z_0)|w|^2}; 1\leq i\leq J ).
\end{equation*}
then the elements of the matrix $\widetilde{G}(w)\defeq  \tLambda_-(w)^{-1}G_{z_0}(w) \tLambda_+(w)^{-1}$
are  i.i.d. $\sim \mathcal{N}_{\C}(0,1)$. As a consequence, $\erw [\log|\det \widetilde{G}(w)|^2]$ 
is independent of $w$. As a result, when applying the Lelong formula \eqref{e:Lelong}, the derivatives will only act on 
$\log|\det \tLambda_-(w)|^2|+\log |\det\tLambda_+(w)|^2$, and yield the 1-point density
	\begin{equation*}
		d^1(z) = \frac{1}{2\pi} \sum_{i=1}^J (\sigma_+^j(z_0)+\sigma_-^i(z_0)).
	\end{equation*}
%
\section{Local Statistics of the eigenvalues of $P_h$ perturbed by a random potential $V_{\omega}$}
\label{sec:LS_V} 
In this section where are interested in the local statistics of the eigenvalues of the 
operator
\begin{equation}\label{eq8.1}
	P^{\delta}=P_h + \delta V_{\omega}, 
\end{equation}
where the random potential $V_{\omega}$ is defined in \eqref{eq1.13}. We
impose that the symbol $p(x,\xi;h)$ of the operator $P_h$ satisfies 
the symmetry \eqref{eq.15},
\begin{equation*}
	p(x,\xi;h) = p(x,-\xi;h).
\end{equation*}
As explained in \eqref{eq1.15.2}, 
we restricted the random variables used in the construction of $V_{\omega}$ to large polydiscs:
\begin{equation}\label{eq8.2}
v=(v_{k})_{k<N(h)}\in \mathrm{PD}_{N(h)}(0,C/h),
\end{equation}
which implies the estimate \eqref{eq5.2.1.6}. Like in section \eqref{sec:LS_M}, we denote by 
$(\mathcal{M}_h,\mathcal{F}_h, \prob_h)$ the probability space 
$(\mathcal{M},\mathcal{F}, \prob)$ restricted to $v_{j}^{-1}(D(0,C/h))$, $j,k<N(h)$. 
The restricted random variables will be denoted by 
\begin{equation}\label{eq8.2b}
	(v_{k}^h)_{k<N(h)}.
\end{equation}
The $v_{k}^h$ are distributed independently and identically. Similarly to \eqref{eq7.3}, \eqref{e:E(q2)}, \eqref{eq7.4}, 
these restricted variables satisfy
\begin{equation}\label{eq8.3}
	\big|\erw[ v_{k}^h ] \big|= \mO(h^{3+\varepsilon_0}), \quad 
	\big | \erw[ (v_{k}^h)^2 ] \big | = \mO(h^{2+\vareps_0}),
\end{equation}
and 
\begin{equation}\label{eq8.4}
	\erw[ |v_{k}^h|^2 ] = 1+ \mO(h^{2+\varepsilon_0}).
\end{equation}
Like in section \ref{sec:LS_M}, we pick 
a $z_0\in\mathring{\Sigma}$; the Grushin problem constructed in 
Proposition \ref{prop5.2.1}, leading to \eqref{eq5.2.11a}, show that that the eigenvalues of the perturbed operator $P^{\delta}$ in 
$W(z_0)$, a relatively compact neighbourhood of $z_0$, are given by the 
zeros of the holomorphic function 
\begin{equation}\label{n2}
   	G^{\delta}(z;h)=\big(1+ R_1(z;h)\big)
	\det \left[ (V_{\omega} h^{-\frac{1}{4}} e_+^{j,hol}(z) |h^{-\frac{1}{4}}e_-^{i,hol} (z))_{i,j\leq J}
	+ R_2(z;h) \right] ,
\end{equation}
with the error terms $R_1$, $R_2$ satisfying the same bounds as in \eqref{eq5.2.12}:
\begin{equation}\label{n2c}
   	\begin{split}
 R_1(z;h)&= \mO( |z-z_0|^{\infty}+ \delta h^{-5/2}) \\ 
 (R_2(z;h))_{i,j} &= \e^{\frac{1}{h}(\Phi_{+,0}^i(z)+ \Phi_{-,0}^j(z) +\mO(|z-z_0|^\infty))} \mO(\delta h^{-3}).
\end{split}
\end{equation}
uniformly in $z\in W(z_0)$ and $v\in PD(0,C/h)$. 
Recall from \eqref{eq4.3.5}, \eqref{eq4.3.4}, that the assumption \eqref{eq.15} implies that the $\pm$ quasimodes are symmetric
	\begin{equation}\label{eq8.5.a}
		e_-^{i,hol}(z;h) = \overline{ e_+^{i,hol}(z;h)}, 
	\end{equation}
The assumptions \eqref{eq1.8.1a}, \eqref{eq4.17.1b} directly show that
	\begin{equation}\label{eq8.5.b}
	\forall i\neq j,\ \forall z\in W(z_0),\qquad	\big( V_{\omega}e_+^{i,hol}(z) \big|  e_-^{j,hol}(z)\big)  =0\,.
	\end{equation}
Since the spectrum of $P^{\delta}$ in $W(z_0)$ is discrete,
$G^{\delta}(\bullet;h)\not\equiv 0$ in $W(z_0)$. 

As in the previous section, to compute the local statistics of the 
eigenvalues we rescale the spectral parameter around $z_0$ by a factor $h^{1/2}$, namely we set
\begin{equation}\label{nota1}
z=z_w=z_0+h^{1/2}w\,,\quad w\in\C. 
\end{equation}
Thus, for any open, simply connected and relatively compact set 
$ O\Subset\C$, we have that, for $h>0$ small enough, the rescaled eigenvalues in $O$
are precisely the zeros of the holomorphic function
\begin{equation}\label{eq8.8}
\begin{split}
F^{\delta}_h(w)&\stackrel{\mathrm{def}}{=} h^{J/4} G^{\delta}(z_w;h) \\
&=(1+ \widetilde{R}_1(w;h))\det\Big[ \mathrm{diag}\big(f^{\delta,h}_{j}(z_w);j=1,\dots,J\big) + \widetilde{R}_2(w;h)\Big],
\end{split}
\end{equation}
with 
\begin{equation}\label{eq8.7}
\begin{split}
&f^{\delta,h}_{j}(w)\defeq \big(V_{\omega}\big| 
  h^{-1/4}\big(e_-^{j,hol}(z_w)\big)^2 \big), 
	\quad 1\leq j\leq J, \\ 
& \widetilde{R}_1(w;h)\stackrel{\mathrm{def}}{=} R_1(z_w;h), \\
& \widetilde{R}_2(w;h)\stackrel{\mathrm{def}}{=}h^{1/4} \, R_2(z_w;h). 
\end{split}
\end{equation}
The need for the normalisation by the factor of $h^{J/4}$ will become apparent later on. 
From now on we let 
\begin{equation}\label{eq.n1a}
	O\Subset\C \text{ be an open, simply connected and relatively compact set } 
\end{equation}
and $h_0>0$ be small enough such that \eqref{eq8.8} is well defined for all $h\leq h_0$. 
From the above discussion, $F^{\delta}_h\not\equiv 0$ 
is a holomorphic function in $O$, so according to section \ref{sec:Conv}
\begin{equation}\label{eq8.9}
	\cZ_h \defeq \sum_{\lambda\in (F^{\delta}_h)^{-1}(0)}\delta_\lambda
\end{equation}
is a well-defined point process on $O$. Our aim is to study the statistical properties of this process, in the limit $h\to 0$.
\subsection{Covariance}\label{sec:Cov2}
In this section we study the covariance of the random functions $f^{\delta,h}_{j}(\bullet)$. 
\begin{prop}\label{prop8.1}
Let $\sigma_+(z)$ be as in \eqref{eq4.4.4}. Let $O$ be as in \eqref{eq.n1a} 
and let $f^{\delta,h}_{i}(w)$, for $w\in O$ and $ 1 \leq i \leq J$, be the rescaled function 
as in \eqref{eq8.7}. Then, for all $v,w\in O$
\begin{equation*}
\begin{split}
\erw\big[f^{\delta,h}_{j}(u) \overline{f^{\delta,h}_{k}(w)}\big] 
\e^{-2 \phi_{s}^j(u;h)-2 \overline{ \phi_{s}^k(w;h)}} 
 = \delta_{j,k} K^{j}(u,\overline{w})(1 + \mO(\sqrt{h})) 
 +\mO(h^2),
 \end{split}
\end{equation*}
where the error terms are uniform in $u,w\in O$,
\begin{equation}\label{eq8.10}
K^{j}(u,\overline{w}) = 
 \e^{\sigma_{+}^j(z_0)u\bar{w}},
\end{equation}
and
\begin{equation}\label{e:phi_s}
	\phi_{s}^j(u;h) = \frac{1}{2}\big[ \log A_{s}^j(z_0;h) + 
				2(\partial^2_{zz}\Phi_{{+},0}^j)(z_0)u^2
	\big],
\end{equation}
with $A_{s}^j(z_0;h)\sim A_0^{j,s}(z_0)+hA_1^{j,s}(z_0)+\dots $ as 
defined in section~\ref{s:symmetric-symbols}.
\end{prop}
Before we continue, recall from \eqref{eq4.21.1b} that $A_0^{j,\pm}(z_0)>0$, so this proposition implies that
\begin{equation}\label{eq8.11}
\erw\big[f^{\delta,h}_{j}(u) \overline{f^{\delta,h}_{k}(w)}\big]
\e^{-2 \phi_{s}^j(u;h)-2 \overline{ \phi_{s}^k(w;h)}}
 \longrightarrow 
 \delta_{j,k} \,K^{j}(u,\overline{w})
\end{equation}
uniformly in $v,w\in O$ as $h\to 0$. 
 \begin{proof}[Proof of Proposition \ref{prop8.1}]
 The proof parallels that of Prop.~\ref{prop7.1}.
We define the following function on $O\times O$:
\begin{equation}\label{eq8.12}
\begin{split}
	h^{1/2}K^{j,k}_h(u,\overline{w})&\defeq h^{1/2}
	\erw \left[f^{\delta,h}_{j}(u)\overline{f^{\delta,h}_{k}(w)}~ \right] \\
	&= 
	\sum_{n,m} \erw\left[v_{n}^h\,\overline{v_{m}^h} \right] 
	\zeta_{n}^{j}(u)\,\overline{\zeta_{m}^{k}(w)},
\end{split}
\end{equation}
where we sum over the index range $1\leq n,m < N(h)$ and where  
\begin{equation}\label{eq8.13}
	\zeta_{n}^{j}(u) = \big(e_n \big| 
(e_-^{j,hol}(z_u))^2 \big)
\end{equation}
with $z_u$ as in \eqref{nota1}. 
Using \eqref{eq8.3}, \eqref{eq8.4}, we see that \eqref{eq8.12} is equal to 
\begin{equation}\label{eq8.13a}
\begin{split}
	&(1+\mO(h^{2+\varepsilon_0}))\sum_{m}
	\big(( e_-^{k,hol}(z_w)^2\big| e_m \big)
\big(e_m \big| (e_-^{j,hol}(z_u))^2 \big)\\
	& + \mO(h^{6+\varepsilon_0}) \sum_{n,m}
	\Big |\big((e_-^{j,hol}(z_u))^2\big| e_n \big)\Big|\,
		\Big | \big(( e_-^{k,hol}(z_w))^2\big| e_m \big)\Big|. 
	\end{split}
\end{equation}
By Lemma \ref{lem7.1},
\begin{equation*}
	 \sum_{n<N(h)}\Big |\big(( e_-^{j,hol}(z_u))^2\big| e_n \big)\Big| \leq 
	 N(h)^{1/2}\big\| ( e_-^{j,hol}
	 	(z_u))^2\big\| (1+ \mO(h^{\infty})).
\end{equation*}
Here, we used as well that $ h^{-1/4}\|(e_-^{j,hol}(z_u))^2\| \asymp 1$, which 
follows directly from \eqref{eq4.21.1} and \eqref{eq8.16} below. Using this to estimate 
the term of order $h^6$ and applying Lemma \ref{lem7.1} to the first term in \eqref{eq8.13a}, one 
gets 
\begin{equation}\label{eq8.14}
\begin{split}
	h^{1/2}K^{j,k}_h(u,\overline{w})= 
	&
	\big(( e_-^{k,hol}(z_w))^2
	 \big| (e_-^{j,hol}(z_u))^2 \big) \\
	& + \mO(h^2) \big\|( e_-^{k,hol}(z_w))^2\big\|\,
	\big \|(e_-^{j,hol}(z_u))^2\big\|.
\end{split}
\end{equation}
Applying Proposition \ref{prop4.4} and \eqref{eq4.21.0} to 
\eqref{eq8.14}, we get that 
\begin{equation}\label{eq8.15}
\begin{split}
h^{1/2}K^{j,k}_h(u,\overline{w}) = 
&\delta_{j,k} \,
 \exp\Big(\frac{2}{h}\Psi^j_s(\overline{z_w},z_u;h)\Big) \\
&+\mO(h^2) \exp\Big(\frac{1}{h}\big(\Phi^k_s(z_w;h) + 
\Phi^j_s(z_u;h)\big)\Big).
 \end{split}
\end{equation}
Next, recall \eqref{eq4.21.1} and write similar to \eqref{eq4.4.1} 
\begin{equation*}
\Phi_{s}^j(z) = 2\Phi_{+,0}^j(z) + h\log A_s^j(z;h), \quad  \Phi_{+,0}^j(z_0) = -\Ima \varphi_+^j(x_+^j(z),z).
\end{equation*}
Notice that by \eqref{eq4.2.0} and \eqref{eq4.4.5}, we have that 
\begin{equation*}
\Phi_{+,0}^j(z_0) = (\partial_z\Phi_{+,0}^j)(z_0) =
(\partial_{\bar{z}}\Phi_{+,0}^j)(z_0) = 0.
\end{equation*}
Moreover, by the discussion after \eqref{eq4.21.1} we see that 
$\partial^{\alpha}h\log \big( h^{1/4} A_s^j(z;h)\big)=\mO(h)$, 
 for all $\alpha\in\N^2$, uniformly in $z\in W(z_0)$. Thus, by Taylor expansion around $z_0$, 
 we have for $h>0$ small enough and $u \in O$ that 
\begin{equation}\label{eq8.16}
	\frac{1}{h}\Phi_{s}^j(z_u;h) = \log h^{1/4} + 
	2(\partial^2_{z\bar{z}}\Phi_{+,0}^j)(z_0)u\overline{u} 
	+ {\phi}_{s}^j(u;h) + {{\phi}_{s}^j(u;h)}
	+\mO(h^{1/2}),
\end{equation}
where the error term is uniform in $u\in O$ and ${\phi}_{s}^j(u;h)$ is as in the statement of the Proposition.
Similarly, by Proposition \eqref{prop4.4} and \eqref{eq4.4.4}, we have for $h>0$ small 
enough and $u,w \in O$ that 
\begin{equation}\label{eq8.18}
	\begin{split}
	\frac{1}{h}\Psi_{s}^j(z_u,\overline{z_w};h) 
	&= \log h^{1/4} + 2(\partial^2_{z\bar{z}}\Phi_{+,0}^j)(z_0)u\overline{w} 
	+{\phi}_{s}^j(u;h) + \overline{{\phi}_{s}^j(w;h)}
	+\mO(h^{1/2})\\
	&= \log h^{1/4} + \frac{1}{2}\sigma_+^j(z_0)u\overline{w} 
	+{\phi}_{s}^j(u;h) + \overline{{\phi}_{s}^j(w;h)}
	+\mO(h^{1/2}),
	\end{split}
\end{equation}
where the error term is uniform in $u,w \in O$. 
Thus, combining \eqref{eq8.15} with (\ref{eq8.16} - \ref{eq8.18}), and using the fact that ${\phi}_{s}^j(u;h)$ is uniformly bounded for $u\in O$, we obtain that 
\begin{equation*}
\erw\big[f^{\delta,h}_{j}(u) \overline{f^{\delta,h}_{k}(w)}\big] \e^{-2\phi_{s}^j(u;h)-2\overline{\phi_{s}^k(w;h)}} 
 =\delta_{j,k} \, \e^{ \sigma_+^j(z_0)u\overline{w}}(1 + \mO(h^{1/2})) 
 	+\mO(h^2).
\end{equation*}
\end{proof}

\subsection{Tightness}\label{sec:Tightness2}
We will follow the same arguments as in section \ref{sec:Tightness1}, 
to show the tightness of the sequence of random analytic functions
 $f_{j}^{h,\delta}$ on $O$ defined in \eqref{eq8.7}. 
\\
\par
Recall \eqref{eq5.2.12} and that the error estimates are uniform in $z\in W(z_0)$ and in 
$v\in\mathrm{PD}_{N(h)}(0,C/h)$, see \eqref{eq6.3.0}, \eqref{eq5.2.0}. 
For $j=1,\dots, J$, let $\Phi_{\pm}^j$ be as in \eqref{eq4.2.1}. Similar as in 
\eqref{eq7.7},  it follows by Proposition \ref{prop4.2} and 
Taylor expansion that 
%
  $\e^{\frac{1}{h}(\Phi_{+,0}^j(z_w)+ \Phi_{-,0}^j(z_w))}= \mO(1)$,
%
uniformly in $w\in O$. By \eqref{eq5.2.12}, \eqref{eq8.7} we have 
  \begin{equation}\label{eqNL2}
  	\begin{split}
	&\widetilde{R}_1(w;h)= \mO(\delta h^{-3/2}), \\
  	&\widetilde{R}_2(w;h) = \mO(\delta h^{-11/4}),
	\end{split}
  \end{equation}
  uniformly in $w\in O$ and in $v\in\mathrm{PD}_{N(h)}(0,C/h)$. %
\\
\par
Let $K\Subset O$ be some compact subset. For $\varepsilon>0$ and let 
$K_{\varepsilon} = K + \overline{D(0,\varepsilon)}$ be the closure of an 
$\varepsilon$-neighbohood of $K$. Pick $\varepsilon>0$ small enough 
so that $K_{\varepsilon}\Subset O$.Thus, by Proposition \ref{prop8.1}, we 
have for $h_0>0$ small enough that 
\begin{equation}\label{eq8.19}
\sup\limits_{0<h<h_0} \erw \left[ \| f_{j}^{h,\delta} \|^2_{L^2(K_{\varepsilon})} \right]  
< C(K_{\varepsilon}) < +\infty.
\end{equation}
Since $F_h^{\delta}$ is holomorphic, Hardy's convexity theorem \cite[Lemma 2.6]{Sh12} 
implies that for any $p>0$ there exists a positive constant $C_{K_{\varepsilon}}>0$ depending only on
 $K_{\varepsilon}$ so that 
\begin{equation}\label{eq8.20}
\| F_h^{\delta} \|^p_{L^{\infty}(K)}   \leq   
C_{K_{\varepsilon}}  \int_{K_{\varepsilon}}  |F_h^{\delta}(w) |^p L(dw). 
\end{equation}
Using Markov's inequality, \eqref{eqNL2} and \eqref{eq8.8} in combination with 
\eqref{eq8.20} for $p=2/J$, one obtains that for $h_0>$ small enough 
\begin{equation}\label{eq8.21}
\begin{split}
 \sup\limits_{0<h<h_0} \prob &\big[\| F_h^{\delta} \|^2_{L^{\infty}(K)} >r \big]  \\
 &\leq r^{-1/J} C_{K_{\varepsilon}} \sup\limits_{0<h<h_0}  
\erw \big[\int_{K_{\varepsilon}}  |F_h^{\delta}(w) |^{2/J} L(dw) \big] \\
&\leq r^{-1/J} \Big( C_1 C_{K_{\varepsilon}} \sup\limits_{0<h<h_0}  
  \sum_{j=1}^J \erw \big[ | f_{j}^{h,\delta}(w)|^{2} \big ] L(dw) + C_2\Big)\\
&\leq \mO(r^{-1/J}), 
\end{split}
\end{equation}
where all the constants are independent of $r>0$. In the above we also 
used that 
\begin{equation*} 
\begin{split}
\Big|\det\Big[ \mathrm{diag}(f^{\delta,h}_{j}(w);j=1,\dots,J) &+ \widetilde{R}_2(w;h)\Big]\Big|^{2/J} \\
 &\leq 2 \| \mathrm{diag}(f^{\delta,h}_{j}(w);j=1,\dots,J)\|_{HS}^2 + 2\|\widetilde{R}_2(w;h)\|_{HS}^2 \\ 
 & \leq 2 \sum_{j=1}^J |f_{j}^{h,\delta}(w)|^2 + 2 \|\widetilde{R}_2(w;h)\|_{HS}^2. 
\end{split}
\end{equation*}
Hence, in view of Proposition 
\ref{prop6.1}, we conclude that $F_h^{\delta}$ is a tight sequence 
of random analytic functions.
\par
For future use, let us remark that similarly to the above, we obtain that 
for any $w_1, \dots, w_n \in O$ and for any $\varepsilon >0$ 
\begin{equation}\label{eq8.21b}
\begin{split}
 \prob \Big[ \sum_{r=1}^n|\det  (
 \mathrm{diag}(f^{\delta,h}_{j}(w_r);j=1,\dots,J) &+\widetilde{R}_2(w_r;h))|^2 >t \Big]  \\
 &\leq \mO(t^{-1/J}n^{1/J})  \erw \big[\| F_h^{\delta} \|^{2/J}_{L^{\infty}(O)}  \big] \\
&\leq \mO(t^{-1/J}n^{1/J}) .
\end{split}
\end{equation}
%
\subsection{Weak convergence to a Gaussian analytic function}
In this section we will show that the random analytic functions $F_h^{\delta}$, cf. \eqref{eq8.8}, converges in  
distribution, when $h\to 0$, to a product of independent Gaussian analytic functions. 
We start by the following result, analogue of Prop.~ \ref{lem6.3.3}:
\begin{prop}\label{prop8.2}
Let $f_{j}^{h,\delta}$ be as in \eqref{eq8.7}, and let 
$K^{j}(u,\overline{w})$ be as in \eqref{eq8.10}. Then 
\begin{equation}\label{eq8.22}
	(f_{j}^{h,\delta}; 1\leq j \leq J) \stackrel{fd}{\longrightarrow} 
	(f_{j}^{GAF}; 1\leq j \leq J), \quad h\to 0,
\end{equation}
where $f_{j}^{GAF}$ are independent Gaussian analytic functions on $O$ 
with covariance kernel 
\begin{equation}\label{eq8.23}
	K^{j}(u,\overline{w})\,\e^{2\phi_s^{j}(u;0)+2\overline{\phi_s^{j}(w;0)}}, \quad u,w\in O, 
\end{equation}
where $\phi_s^{j}(u;0)$ is the quadratic polynomial described in \eqref{e:phi_s}.
\end{prop}
We have the following immediate consequence. 
\begin{cor}\label{cor8.1}
Under the assumptions of Proposition \ref{prop8.2}, we have that 
\begin{equation}\label{eq8.24}
	\det  \big( \mathrm{diag}(f^{\delta,h}_{j};j=1,\dots,J)+\widetilde{R}_2\big)
	\stackrel{fd}{\longrightarrow} 
	\prod_{j=1}^J f_{j}^{GAF} , \quad h\to 0.
\end{equation}
\end{cor}
\begin{proof}
The proof is similar to that of Corollary \ref{cor7.3}.
\end{proof}
\begin{proof}[Proof of Proposition \ref{prop8.2}]
The proof is similar to that of Proposition \ref{lem6.3.3}. For $M\in\N^*$ let  
$w_1,\dots,w_M \in  O$ and $\lambda=(\lambda^{j}_l; 1\leq j \leq J, 1\leq l \leq M) \in \C^{J\cdot M}$. 
Consider the  following complex valued random variable 
\begin{equation*}
 S(\lambda) \defeq \sum_{j,l} \lambda^{j}_lf_{j}^{h,\delta}(w_l) 
 		   = \sum_{n<N(h)} v^h_{n} G_{n}, 
		   \quad 
		   G_{n} = \sum_{j,l}\lambda^{j}_l h^{-1/4} \zeta_{n}^{j}(z_{w_l}),
\end{equation*}
where $\zeta_n^j$ is as in \eqref{eq8.13}. We are going to show that 
$S(\lambda)$ converges in distribution to the complex valued random 
variable 
 \begin{equation*}
 S^{GAF}(\lambda) \defeq \sum_{j,l} \lambda^{j}_l f_{j}^{GAF}(w_l),
\end{equation*}
for all $M\in\N$, let $w_1,\dots,w_M \in  O$ and all 
$\lambda=(\lambda^{i,j}_l; 1\leq i,j \leq J, 1\leq l \leq M) \in \C^{J^2\cdot M}$. 
By the Cram\'er-Wold Theorem this implies \eqref{eq8.22}.
To prove the limit $S(\lambda)\stackrel{d}{\to} S^{GAF}(\lambda)$, we will use the 
central limit theorem expressed in Theorem \ref{thmCLT}. We thus need to check that the family of random variables $(S(\lambda))_{0<h\leq h_0}$ satisfies the four conditions stated in the theorem.
\par
We begin by estimating the average of $S(\lambda)$: from \eqref{eq8.3}, we get
\begin{equation}\label{eclt1}
\begin{split}
\sum_{n<N(h)} |\erw [v^h_{n} G_{n}]| 
&\leq \mO(h^{3+\vareps_0}) \sum_{n<N(h)} |G_{n}|\\
&\leq \mO(h^{3+\vareps_0}) \sum_{j,l}\sum_{n<N(h)}
\left|\left(h^{-1/4}(e_-^{j,hol}(z_{w_l})^2|e_n\right)	\right|\\
&\leq  \mO(h^{3+\vareps_0}) N(h) \sum_{j,l}\|h^{-1/4}(e_-^{j,hol}(z_{w_l}))^2 \|\\
&\leq  \mO(h^{1+\vareps_0})\,.
\end{split}
\end{equation}
In the third line we used two Cauchy-Schwarz inequalities, and in the last one the fact that the states 
$\| h^{-1/4} (e_-^{j,hol}(z_w))^2\|=\mO(1)$ uniformly when $w\in O$. This proves the point $(i)$ 
in Theorem~\ref{thmCLT}.
Let us now check condition $(ii)$ of Theorem \ref{thmCLT}. 
Using \eqref{eq8.3}, we draw the bound
\begin{equation}\label{eclt2}
\begin{split}
\Big| \sum_{n<N(h) }\erw[(v_{n}^h)^2G_{n}^2] \Big|
&=\sum_{n<N(h)}  \Big|\erw[(v_{n}^h)^2]G_{n}^2\Big|\\
&\leq \mO(h^{2+\vareps_0}) \sum_{n} |G_{n}|^2 \\
& \leq  \mO(h^{2+\vareps_0}) \sum_{j,l}\sum_{n<N(h)} 
\Big|\big(h^{-1/4}(e_-^{j,hol}(z_{w_l})^2|e_n\big)	\Big|^2\\
& \leq \mO(h^{2+\vareps_0}) \,.
\end{split}
\end{equation}
Next, we compute the variance of $S(\lambda)$. By \eqref{eq8.3}, \eqref{eq8.4} 
\begin{equation*}
	\erw \Big[\sum_{n<N(h)} |v^h_{n}G_{n}|^2 \Big]
	= (1+\mO(h^{2+\varepsilon_0}))\sum_{j,l,s,t }\lambda^{j}_l\overline{\lambda^{s}_t}
	\sum_{n<N(h)} h^{-1}\zeta_{n}^{j}(z_{w_l})\overline{\zeta_{n}^{s}(z_{w_t})},
\end{equation*}
which, by \eqref{eq8.13} and Lemma \ref{lem7.1} is equal to 
\begin{equation}\label{eq8.25}
	(1+\mO(h^{2+\varepsilon_0}))\sum_{j,l,s,t }\lambda^{j}_l\overline{\lambda^{s}_t}
	h^{-1/2}\big(( e_-^{j,hol}(z_{w_t}))^2|
			 ( e_-^{s,hol}(z_{w_l}))^2  \big)+ \mO(h^\infty).
\end{equation}
Here, we used as well $ h^{-1/4}\|  e_{\pm}^{k,hol}(z_{w_l})\|= \mO(1)$, 
as $h\to 0$, which follows by \eqref{eq4.21.1} and \eqref{eq8.16}. 
\par
By \eqref{eq4.21.0} we see that the terms in \eqref{eq8.25} with $j\neq s$ 
are $\mO(h^\infty)$. Similar to the proof of Proposition \ref{prop8.1}, we then 
obtain that 
\begin{equation}\label{eq8.26}
	\erw \Big[\sum_{n<N(h)} |v^h_{n}G_{n}|^2 \Big]
	\longrightarrow \sum_{j,l,t }\lambda^{j}_l\overline{\lambda^{j}_t} 
	K^{j}(w_l,\overline{w}_t)\,\e^{2\phi_s^{j}(w_l;0)+\overline{2\phi_s^{j}(w_t;0)}}
	\defeq \sigma(\lambda,w)
\end{equation}
as $h\to 0$. 
\par
Continuing, we obtain by \eqref{eq1.9} and the H\"older inequality that 
\begin{equation*}
\begin{split}
	\sum_{n<N(h)} \erw \Big[|v^h_{n}G_{n}|^{4+\varepsilon_0} \Big]
	&\leq C_{\lambda}  \sum_{j,l}
	\sum_{n<N(h)} | h^{-1/4} \zeta_{n}^{j}(z_{w_l})|^{4+\varepsilon_0} \\ 
	&= C_{\lambda}  \sum_{j,l}
	\sum_{m<N(h)}\big|\big(h^{-1/4}(e_-^{j,hol}(z_{w_l}))^2|e_m\big)
		\big|^{4+\varepsilon_0}.
\end{split}
\end{equation*}
Using that $\{e_m\}$ is an orthonormal basis of $L^2(\R)$, as well as the uniform estimates $\mO(h^{1/4})$ for the overlaps (see
Lemma \ref{lem7.2}), one obtains that 
\begin{equation*}
\begin{split}
	\sum_{n<N(h)} \erw \big[|v^h_{n}G_{n}|^{4+\varepsilon_0} \big] 
	&= \mO(h^{(2+\varepsilon_0)/4}) \sum_{j,l} 
		\big\|(h^{-1/4} (e_-^{j,hol}(z_{w_l}))^2\big\|^2
	\e^{\frac{4}{h}\Phi_{+,0}^j(z_{w_l})}\\
	&= \mO(h^{(2+\varepsilon_0)/4}). 
\end{split}
\end{equation*}
In the last line we used as well Taylor expansion around $z_0$ and the 
fact that $O$ is relatively compact, as in \eqref{eq8.16}, to see that 
the exponentials are of order $\mO(1)$. 
We observe that for any $\varepsilon >0$
\begin{equation}\label{eq8.27}
 \sum_{n<N(h)} \erw \big[ |v^h_{n}G_{n}|^2 
 	\mathds{1}_{\{|v^h_{n}G_{n}|>\varepsilon\}} \big]
 < \varepsilon^{-(2+\varepsilon_0)}\sum_{n<N(h)} \erw \big[ |v^h_{n}G_{n}|^{4+\varepsilon_0} \big] 
 \longrightarrow 0,
\end{equation}
as $h\to 0$. In view of \eqref{eclt1}, \eqref{eclt2}, \eqref{eq8.25}, \eqref{eq8.26}, 
\eqref{eq8.27}, we may apply the central limit theorem given in Theorem \ref{thmCLT}, 
to show that $S(\lambda)$ converges in distribution to the random Gaussian variable $\sim\mathcal{N}_{\C}(0,\sigma(\lambda,w))$, with variance given in \eqref{eq8.26}. 
\\
\par
On the other side, since $(f_{j}^{GAF})_{j\leq J}$ are independent Gaussian analytic functions 
with covariance kernel \eqref{eq8.23}, it follows that 
\begin{equation}
	 \sum_{j,l} \lambda^{j}_l f_{j}^{GAF}(w_l) \sim \mathcal{N}_{\C}(0,\widetilde{\sigma}(\lambda,w))\,,
\end{equation}
with
\begin{equation}
\begin{split}
	\widetilde{\sigma}(\lambda,w)  &=\erw\Big[\big| \sum_{j,l} \lambda^{j}_l f_{j}^{GAF}(w_l)\big|^2\Big] = 
	\sum_{j,l,s,t} \lambda^{j}_l \overline{\lambda^{s}_t}\,\erw\left[ f_{j}^{GAF}(w_l) 
	\overline{f_{s}^{GAF}(w_t)} \right] \\
	& = \sum_{j,l,t }\lambda^{j}_l\overline{\lambda^{j}_t} 
	K^{j}(w_l,\overline{w}_t)\,\e^{2\phi_s^{j}(w_l;0)+\overline{2\phi_s^{j}(w_t;0)}} \\
	& = \sigma(\lambda,w),
\end{split}
\end{equation}
where in the last line we also used \eqref{eq8.26}. Since a complex Gaussian random variable 
is uniquely determined by its expectation and its variance, we conclude that $S(\lambda)\stackrel{d}{\to} S^{GAF}(\lambda)$ and, equivalently, \eqref{eq8.22}.
\end{proof}
Next, we show that the finite dimensional distributions of $F^{\delta}_h$ 
and $\det  \big( \mathrm{diag}(f^{\delta,h}_{j})+\widetilde{R}_2\big)$ 
converge in distribution to the same limit.
By \eqref{eq8.20}, \eqref{eq8.21}, it follows that the sequence of random vectors
\begin{equation*}
	\widetilde{F}^{\delta}_h(w)\stackrel{\mathrm{def}}{=}(F^{\delta}_h(w_1),\dots,F^{\delta}_h(w_n)), 
	\quad w = (w_1,\dots,w_n) \in O^n,
\end{equation*}
is tight for any $n\in\N$. 
Let $n\in\N$ and $\phi\in\mathcal{C}_c(\C^n,\R)$, let $w = (w_1,\dots,w_n) \in O^n$,
and set:
\begin{equation*}
	\widetilde{g}^{\delta}_h(w)\stackrel{\mathrm{def}}{=}
	\big(g^{\delta}_h(w_1),\cdots, g^{\delta}_h(w_n)\big), 
	\quad 
	g^{\delta}_h(w_l)= \det  \big( \mathrm{diag}(f^{\delta,h}_{j}(w_l))+\widetilde{R}_2(w_l)\big).
\end{equation*}
Similar to \eqref{eq7.18}, one can then show that for any $\varepsilon>0$ 
there exists $h_0>0$ such that for all $0<h<h_0$, 
\begin{equation}\label{eq8.28.0}
	\big|\erw [ \phi(\widetilde{F}^{\delta}_h(w))] - \erw [ \phi(\widetilde{g}^{\delta}_h(w))] \big| \leq \varepsilon.
\end{equation}
Moreover, since $\widetilde{F}^{\delta}_h(w)$ is a tight sequence of random vectors, it follows from 
Remark \ref{Rem6.1} and Corollary \ref{cor8.1} that 
\begin{equation}\label{eq8.28a}
	F^{\delta}_h \stackrel{fd}{\longrightarrow} \prod_{j=1}^Jf_{j}^{GAF} , 
	\quad 
	\text{as }h\to 0.
\end{equation}
Since $F^{\delta}_h(\bullet)$ is a tight sequence of random analytic functions, see Section \ref{sec:Tightness2}, 
it then follows by Theorem \ref{thm:Pro} and \eqref{eq8.28a} that 
\begin{equation}\label{eq8.28}
	F^{\delta}_h \stackrel{d}{\longrightarrow} \prod_{j=1}^Jf_{j}^{GAF} , 
	\quad 
	\text{as }h\to 0,
\end{equation}
where $ f_{j}^{GAF} $ are as in Proposition \ref{prop8.2}, with covariance kernel 
given by $K^{j}(u,\bar{w})\,\e^{2\phi_s^{j}(u;0)}\e^{2\overline{\phi_{s}^j(w;0)}}$.
Since $\e^{2\phi_s^{j}(u;0)}$ is a nonvanishing deterministic holomorphic function on $O$, 
 \begin{equation*}
 	g_{j}(w)\defeq f_{j}^{GAF}(w)\e^{-2\phi_s^{j}(w;0)}
 \end{equation*}
is a Gaussian analytic function on $O$, with covariance 
kernel $K^{j}(u,\bar w)$. Moreover, setting 
 \begin{equation*}
 	\widetilde{G}(w)\prod_{j=1}^J \e^{2\phi_s^{j}(w;0)}
	\defeq \prod_{j=1}^J g_j(w) \prod_{j=1}^J \e^{2\phi_s^{j}(w;0)}= 
	\prod_{j=1}^J f_{j}^{GAF}(w)\defeq G(w), 
 \end{equation*}
we see that $\widetilde{G}(w)\prod_{j=1}^J \e^{2\phi_s^{j}(w;0)}$ and 
$G(w)$, as random analytic functions, have the same distribution. Since, 
$ \prod_{j=1}^J \e^{2\phi_s^{j}(w;0)}$ never vanishes, it follows by the continuous 
mapping theorem \cite{Kal97} and the discussion after \eqref{eq:PpP} that for all $\phi\in\mathcal{C}_c(O,\R)$,
 \begin{equation*}
 	\sum_{\lambda\in G^{-1}(0)}  \phi(\lambda)  \stackrel{d}{=}
	\sum_{\lambda\in \widetilde{G}^{-1}(0)}  \phi(\lambda).
 \end{equation*}
Hence, by \eqref{eq8.28} and Proposition \ref{prop6.6}, we see 
that for all $\phi\in\mathcal{C}_c(O,\R)$
 \begin{equation}
  	\sum_{\lambda\in (F^{\delta}_h)^{-1}(0)} \phi(\lambda)
	\stackrel{d}{\longrightarrow} 
	\sum_{\lambda\in G^{-1}(0)}  \phi(\lambda) 
	\stackrel{d}{=}\sum_{\lambda\in \widetilde{G}^{-1}(0)}  \phi(\lambda), \quad 
	\text{ as } h\to 0.
 \end{equation}
We notice that the random function $\widetilde{G}$ is identical with $G_{z_0}$ in the Theorem~\ref{thm_m2}. Together with the discussion at the beginning of Section \ref{sec:LS_M}
and with Proposition \ref{prop7.1}, this ends the proof of that Theorem. 
 %
%
\subsection{Correlation functions}
Let $\mu_k^h$ be the $k$-point density measure of $\mathcal{Z}_h^M$, 
defined in Theorem \ref{thm_m3}, and let $\mu_k$ be the $k$-point density measure of 
the point process $\mathcal{Z}_{G_{z_0}}$, given in Theorem \ref{thm_m3}. 
Following the exact same arguments as in Section \ref{sec:kptmes}, we obtain the first part of Theorem~\ref{thm_m3}: for any $O\Subset\C$ open, 
relatively compact connected domain and for all $\varphi\in\mathcal{C}_c(O^k\backslash\Delta,\R_+)$, %
\begin{equation*}
	\int \varphi \,d\mu_k^h \longrightarrow \int \varphi \,d\mu_k
	\quad \text{as }h\to 0.
\end{equation*}
Recall from Theorem \ref{thm_m2} that $G_{z_0}(z)=\prod_{j=1}^Jg_j(z)$ where the 
$g_j$ are independent Gaussian analytic functions. To complete the proof of 
Theorem \ref{thm_m3}, we will show, using the simple product structure of 
$G_{z_0}(z)$, that the $k$-point measure $\mu_k$ of $G_{z_0}$ has a
continuous Lebesgue density $d_G^k$, called the 
$k$-point density, which can be explicitly computed, using the expressions of the $k'$-point densities of the GAFs $g_j$. This is not surprising, remembering that the process $\cZ_{G_{z_0}}$ is the superposition of the $J$ independent processes $\cZ_{g_j}$. This contrasts with the case of the 
point process associated with random matrix perturbations of $P_h$ (see Thm~\ref{thm_m1}), for which the computation of the limiting $k$-densities remains an open problem. 
\par
Denoting $\Delta_{w}=\partial_{\bar w}\partial_{w}$ and $w=(w_1,\ldots,w_k)$, the generalization of \eqref{e:Lelong} reads:
\begin{equation}\label{eq8.30.1}
 \begin{split}
		\mu_k(dw) &= (2\pi)^{-k}\big(\prod_{i=1}^k \Delta_{w_i}\big) \erw[ \log |G_{z_0}(w_1)|\cdots \log |G_{z_0}(w_k)|] 
		L(dw)\\ 
		&\stackrel{\mathrm{def}}{=}d^k(w)\,L(dw),
\end{split} 
\end{equation}
in the sense of distributions. This follows from the Poincar\'e Lelong formula 
and Fubini's theorem, see e.g. \cite{HoKrPeVi09}. Applying the product structure of $G_{z_0}$,  we obtain
 \begin{equation}\label{eq8.30}
		d^k(w)
		=\sum_{\beta_1,\ldots,\beta_k=1}^J
		\big(\prod_{i=1}^k \Delta_{w_i}\big) \erw[ \log |g_{\beta_1}(w_1)|\cdots \log |g_{\beta_k}(w_k)|].
\end{equation}
Since the $g_j$ are mutually independent, for each term $\beta=(\beta_1,\ldots,\beta_k)$ the expectation in \eqref{eq8.30} 
factorizes, each factor grouping together the identical GAF $g_{\beta_i}=g_j$. Upon applying the relevant derivatives $\Delta_{w_i}$, each such factor yields a certain density associated with the GAF $g_j$.

According to  \cite[Cor.3.4.2]{HoKrPeVi09}, for a Gaussian analytic function $f$ on an open set $U\subset\C$ 
with covariance kernel 
$K_f(u,\overline{w})$, if $\det K_f(w_i,\overline{w}_j)_{1\leq i,j\leq k}$ does not vanish 
anywhere on $U^k\setminus \Delta$,  then the $k$-point measure of $\cZ_f$  has a continuous Lebesgue density $d_f(w)$ 
given by 
 \begin{equation}\label{eq8.31}
 \begin{split}
 	d_f^k(w) = \frac{1}{(2\pi)^k} \big( \prod_{i=1}^k \Delta_{w_i} \big) \erw[ \log |f(w_1)|\cdots\log |f(w_k)|] \\ 
		  = \frac{\mathrm{perm}\big(C(w)- B(w)A^{-1}(w)B^*(w) \big) }{\pi^k\det  A(w)},
\end{split} 
 \end{equation}
 where $\mathrm{perm}(\cdot)$ denotes the Permanent of a matrix and $A,B,C$ are 
 complex $k\times k$-matrices defined by 
  \begin{equation*}
  	A_{i,j}(w) = K(w_i,\overline{w_j}), \quad 
	B_{i,j}(w) = (\partial_w K)(w_i,\overline{w_j}), \quad 
	C_{i,j}(w)= (\partial_{w \overline{w}}K)(w_i,\overline{w_j}).
  \end{equation*}
In our case, Theorem \ref{thm_m2} shows that for any of $j=1,\ldots,J$ and any $k\geq 1$,
 \begin{equation*}
  	\det \big(K^j(w_n,\overline{w}_m)\big)_{1\leq n,m\leq k} 
	= \det (\e^{\sigma_+^j(z_0)w_n\overline{w}_m} )_{1\leq n,m\leq k}.
\end{equation*}
This determinant vanishes if and only if $w_i=w_j$ for some $i\neq j$, that is if $w$ is in the diagonal $\Delta$.
Therefore, the $k$-point density functions $d^k_{g_j}$ of the $g_j$ exist, 
and can be expressed in the form \eqref{eq8.31}. Let $\mathfrak{S}_k$ be the permutation  
group of $k$ elements. Any $k$-density is symmetric with respect to permutations, hence 
   \begin{equation}\label{eq8.30.3} 
  	d^k_f(w) = \frac{1}{k!} \sum_{\tau\in \mathfrak{S}_k} d^k_f(w_{\tau(1)},\dots, w_{\tau(k)}).
  \end{equation}
Simple combinatorics shows that we can decompose the index set in 
the following disjoint way
 \begin{equation}\label{eq8.30.2}
 \begin{split}
  	&\{1\leq \beta_1,\ldots,\beta_k \leq J\}
	= 
	\bigsqcup_{\alpha \in \N^J, |\alpha|=k}  A_{\alpha},\\
  A_{\alpha} &= \{\beta\in \N^k; \text{for each $j=1,\ldots,J$, there are $\alpha_j$ components $\beta_i$ equal to $j$}  \}.
  \end{split}
 \end{equation}
 Here we used the notation $|\alpha|=\alpha_1+\cdots+\alpha_J$ for a multiindex $\alpha\in\N^J$.
Using \eqref{eq8.30}, \eqref{eq8.30.2}, \eqref{eq8.30.3} we write 
   \begin{equation*}
  	d^k(w) = 
	\sum_{\alpha \in \N^J, |\alpha|=k} 
	\sum_{\beta\in A_{\alpha}}
	\frac{1}{k!} \sum_{\tau\in \mathfrak{S}_k} 
	\big(\prod_{i=1}^k \Delta_{w_i}\big) \erw[ \log |g_{\beta_1}(w_{\tau(1)})|\cdots\log |g_{\beta_k}(w_{\tau(k)})|].
  \end{equation*}
Since we sum over all possible permutations $\tau$ and since the sum over $\beta\in A_{\alpha}$ 
is just a sum over all possible orderings for a fixed configuration $\alpha$, it follows together 
with \eqref{eq8.31} that 
   \begin{equation*}%
  	d^k(w) = 
	\sum_{\alpha \in \N^J, |\alpha|=k} 
	\frac{1}{k!}\binom{k}{\alpha} \sum_{\tau\in \mathfrak{S}_k} 
	\prod_{j=1}^J d_{g_j}^{\alpha_j}(w_{\tau(\alpha_1+\cdots+\alpha_{j-1}+1)},\dots, w_{\tau(\alpha_1+\cdots+\alpha_{j})}),
  \end{equation*}
where $\binom{k}{\alpha}=\frac{k!}{\alpha!}\defeq \frac{k!}{\alpha_1!\alpha_2!\cdots\alpha_J!}$ is the multinomial coefficient. 
This is the expression \eqref{eq2.11.4} in Theorem \ref{thm_m3}.
\medskip

To end this presentation, we provide the explicit formula for the $2$-point correlation function 
of the limiting point process $\cZ_{G_{z_0}}$. Notice first that 
\begin{equation*}
	 d^1(w) = \sum_{j=1}^Jd^1_{g_{j}}(w).
\end{equation*}
From \eqref{eq8.31} one calculates directly that 
\begin{equation*}
	d^1_{g_j}(w) = \frac{1}{\pi}\partial_w\partial_{\bar{w}}\log K^j(w,\overline{w}) = \frac{\sigma_+^j(z_0)}{\pi}.
\end{equation*}
Hence, 
\begin{equation*}
	 d^1(w) =\frac{1}{\pi} \sum_{j=1}^J \sigma_+^j(z_0),\quad\text{for all $w\in \C$}.
\end{equation*}
Next, we have that 
\begin{equation*}
	 d^2(z_1,z_2) = \sum_{j=1}^Jd^2_{g_{j}}(z_1,z_2) + 
	 \sum_{\substack{i,j=1 \\ i\neq j} }^Jd^1_{g_{i}}(z_1)d^1_{g_{j}}(z_2).
\end{equation*}
A cumbersome but straightforward calculation, using \eqref{eq8.31}, shows that for any $w_1\neq w_2$,
\begin{equation*}
	d^2_{g_{j}}(w_1,w_2) = \big(\frac{\sigma_+^j(z_0)}{\pi}\big)^2
					   \kappa \big( \frac{\sigma^j_+(z_0)}{2}|w_1-w_2|^2 \big),
\end{equation*}
where 
\begin{equation*}
	\kappa (t)  = \frac{(\sinh^2 t + t)\cosh t - 2t \sinh t}{\sinh^3 t}, \quad t\geq 0.
\end{equation*}
A Taylor expansion shows that $\kappa(t)=t(1+\mO(t^2))$, as $t\to 0^+$, and 
$\kappa(t)=1+\mO(t^2\e^{-2t})$, as $t\to +\infty$. This expression has been obtained before
by \cite{Ha96,BlShiZe00}. We then obtain the following formula for the $2$-point correlation function of $\cZ_{G_{z_0}}$:
\begin{equation*}
\begin{split}
	K^2(w_1,w_2) &\defeq \frac{ d^2(w_1,w_2)}{d^1(w_1)d^1(w_1)}\\
			     &=1
			     +\sum_{j=1}^J\frac{(\sigma_+^j(z_0))^2}
			     {\big(\sum_{j=1}^J\sigma_+^j(z_0)\big)^2}
			      \Big[\kappa\big( \frac{\sigma^j_+(z_0)}{2}|w_1-w_2|^2 \big)
			      -1\Big].
\end{split}
\end{equation*}
For $|w_1-w_2| \gg 1$, we have the asymptotics
\begin{equation*}
	K^2(w_1,w_2) = 1+
	\mO\Big(\big( \sigma^j_+(z_0)|w_1-w_2|^2 \big)^2\e^{-\min\limits_{j}
	\sigma^j_+(z_0)|w_1-w_2|^2}\Big),
\end{equation*}
while for $|w_1-w_2| \ll1 $,
\begin{equation*}
	K^2(w_1,w_2) =1 - \sum_{j=1}^J\frac{(\sigma_+^j(z_0))^2}
			     {\big(\sum_{j=1}^J\sigma_+^j(z_0)\big)^2}
			      \Big[1 -  \frac{\sigma^j_+(z_0)}{2}|z_1-z_2|^2 \big(1+ 
			      \mO\big(|z_1-z_2|^4)
			      \big)
			      \Big].
\end{equation*}
In particular, this 2-point correlation does not vanish when $|w_1-w_2|\to 0$.
\bigskip

Let us finally prove Proposition \ref{thm_Tran}, which expresses the invariance of our limiting processes with respect to the direct isometries of $\C$.
\begin{proof}[Proof of Proposition \ref{thm_Tran}]
	We will show the result in case of $G_{z_0}$ as in Theorem \ref{thm_m2}, since 
	the case of Theorem \ref{thm_m1} is similar. 
	Let $\alpha,\beta\in\C$ with $|\alpha|=1$ and let $\varphi(w)=\alpha w + \beta$. 
	By the continuous mapping theorem \cite{Kal97} and the discussion after 
	\eqref{eq:PpP}, it is sufficient to show that there exists 
	a deterministic function holomorphic function $\Phi(w)$ such that 
	\begin{equation}\label{e:multipl}
		G_{z_0}(\varphi(w)) \stackrel{d}{=} G_{z_0}(w) \e^{\Phi(w)}, \quad w\in\C.
	\end{equation}
	Recall that  
	 \begin{equation*}
 		G_{z_0}(w)=\det (g_j(w))
	 \end{equation*}
where $g_j$, for $1\leq j\leq J$, are independent Gaussian analytic functions on 
 $\C$ with covariance kernel 
 \begin{equation}
K^{j}(v,\overline{w}) = 
 \e^{\sigma_+^j(z_0) v\overline{w}}.
\end{equation}
Then, $g_j\circ \varphi$ is also an independent Gaussian analytic function on $\C$ with covariance  
kernel 
 \begin{equation}
K^{j}_{\varphi}(v,\overline{w}) = 
 \e^{\sigma_+^j(z_0) v\overline{w}} 
 \e^{\phi_+^i(v)+\overline{\phi_+^i(w)}}
\end{equation}
with 
 \begin{equation*}
 	\phi_{+}^j(w) = \sigma_{\pm}^i(z_0)\big(\alpha \overline{\beta} w + 
		\frac{1}{2}|\beta|^2\big).
 \end{equation*}
Hence, the GAFs $g_j\circ \varphi$ and $g_j\e^{\phi_+^j}$ are equal in distribution, 
since they have the same covariance kernel. By the continuous mapping theorem, 
the random analytic functions $G_{z_0}\circ\varphi$ and 
 $G_{z_0}\e^{\Phi}$ are equal in distribution, provided we take
$\Phi(w) = \sum_{j=1}^J\phi_{+}^j(w)$. 
\end{proof}

\bigskip
\noindent\textbf{Acknowledgements.}
Both authors acknowledge the support of the grant Gerasic-ANR-13-BS01-0007-02 awarded by the Agence Nationale de la Recherche. The second author was also supported by the Erwin Schr{\"o}dinger Fellowship J4039-N35 and by the National Science Foundation grant DMS-1500852 and he is grateful to the Laboratoire de Math\'ematiques d'Orsay for providing 
a stimulating environment in which most of this paper has been written.  
\providecommand{\bysame}{\leavevmode\hbox to3em{\hrulefill}\thinspace}
\providecommand{\MR}{\relax\ifhmode\unskip\space\fi MR }
\providecommand{\MRhref}[2]{%
  \href{http://www.ams.org/mathscinet-getitem?mr=#1}{#2}
}
\providecommand{\href}[2]{#2}

\end{document}